\documentclass{amsart}
\usepackage{mathtools}
\usepackage{amsfonts}
\usepackage{amsmath}
\usepackage{amsthm}
\usepackage[square, numbers]{natbib}
\usepackage{hyperref}
\usepackage{cleveref}
\usepackage{amssymb}
\usepackage{babel}
\usepackage{tikz-cd}
\tikzcdset{scale cd/.style={every label/.append style={scale=#1},
		cells={nodes={scale=#1}}}}
\usetikzlibrary{cd, babel}
\usepackage{tikz}
\usepackage{color}
\usepackage{array}
\usepackage{multirow}
\usepackage{mathrsfs}
\usepackage{gensymb}
\usepackage{tabularx}
\usepackage{graphicx}
\graphicspath{ {images/} }
\usepackage{caption}
\usepackage{enumerate}
\usepackage{setspace}
\usepackage{float}
\usepackage[margin=1.5in]{geometry}
\usepackage[utf8]{inputenc}

\newcommand\xqed[1]{%
	\leavevmode\unskip\penalty9999 \hbox{}\nobreak\hfill
	\quad\hbox{#1}}
\newcommand\lozengeend{\xqed{$\lozenge$}}

\title{Iwasawa Theory for GU(2,1) at inert primes}
\author{Muhammad Manji}

\newtheorem{thm}{Theorem}[section]
\newtheorem{lem}[thm]{Lemma}
\newtheorem{prop}[thm]{Proposition}
\newtheorem{cor}[thm]{Corollary}
\newtheorem{conj}[thm]{Conjecture}
\theoremstyle{definition}
\newtheorem{defn}[thm]{Definition}
\theoremstyle{remark}
\newtheorem{rmk}[thm]{Remark}
\newtheorem{ass}[thm]{Assumption}

\begin{document}
    \maketitle
	
	\section*{Abstract}
		
        \par Many problems of arithmetic nature rely on the computation or analysis of values of $L$-functions attached to objects from geometry. Whilst basic analytic properties of the $L$-functions can be difficult to understand, recent research programs have shown that automorphic $L$-values are susceptible to study via algebraic methods linking them to Selmer groups. Iwasawa theory, pioneered first by Iwasawa in the 1960s and later Mazur and Wiles provides an algebraic recipe to obtain a $p$-adic analogue of the $L$-function. In this work we aim to adapt Iwasawa theory to a new context of representations of the unitary group GU(2,1) at primes inert in the respective imaginary quadratic field. This requires a novel approach using the Schneider--Venjakob regulator map, working over locally analytic distribution algebras. Subsequently, we show vanishing of some Bloch--Kato Selmer groups when a certain $p$-adic distribution is non-vanishing. These results verify cases of the Bloch--Kato conjecture for GU(2,1) at inert primes in rank 0.
		
		\tableofcontents

		\section{Introduction}
			\par In the past decade there has been a series of developments in the theory of Euler systems, in particular the construction of new Euler systems for Shimura varieties attached to reductive linear algebraic groups such as $\mathrm{GL}_2 \times_{\mathrm{GL}_1} \mathrm{GL}_2$, $\mathrm{GSp}_4$, as well as the construction for GU(2,1) in \citep{LSZ21} discussed in this thesis.  With an increasing array of examples, new techniques are being developed to extend these ideas to new cases. We will focus on the case where a prime $p$ is inert in the imaginary quadratic reflex field $E$ of the GU(2,1) Shimura variety we work over. The split prime case in \citep{Man22} follows existing ideas, largely of Kato in \citep{Kat04}, while the inert case needs fresh ideas. 
			\par This work provides techniques which can be applied to other settings with non-split primes over quadratic reflex fields. Partial results have been obtained, proving the main conjectures up to an explicit reciprocity law; we will have to use placeholders for $p$-adic $L$-functions whose constructions do not exist yet, however this still takes us a step closer to the Bloch--Kato conjecture for GU(2,1). 
			\par In \S1 we discuss automorphic representations $\Pi$ for GU(2,1) defined with respect to $E$ and properties of associated Galois representations $V=V_\mathfrak{P}(\Pi)$ coming from the Langlands program.  We recall the statement of the Loeffler--Skinner--Zerbes GU(2,1) Euler system $c^\Pi \in H^1_{\text{Iw}}(E[p^\infty]/E,V)$, where $E[p^\infty]$ is the maximal extension of $E$ unramified away from $p$, so that $\Gamma_E=\mathrm{Gal}(E[p^\infty]/E) \cong \mathbb{Z}_p^2 \times \Delta$ for a finite group $\Delta$. Let $\Lambda$ be the two variable Iwasawa algebra $\mathcal{O}_{E_p}[[\Gamma_E]]$. We adapt versions of Mazur, Rubin and Kato's Euler system machinery from \citep{MR04} and \citep{Kat04} to prove one divisibility of Kato's Iwasawa main conjecture for such $V$, as a statement of $\Lambda$-modules (i.e. without restricting to the cyclotomic 1-variable Iwasawa algebra).
			\par In \S2, inspired by work of Pottharst, we recall the construction via rigid geometry of a candidate distribution algebra $\Lambda_\infty^L$ from Schneider--Teitelbaum's $p$-adic Fourier theory where the $p$-adic $L$-function can live. We will use these results when $L=E_p=\mathbb{Q}_{p^2}$. Since this algebra is defined by an `$L$-analyticity condition' on distributions rather than a condition on the Galois group (in contrast to cyclotomic or anticyclotomic Iwasawa theory), we do not have a way to `pull back' to a quotient of the Iwasawa algebra. We need to understand the module theory and prime ideal theory of $\Lambda_\infty^L$ to base change Kato's main conjecture to this algebra.
			\par \begin{thm}[See Corollary \ref{analytic_kmc}] \label{intro1}
				Suppose $V$ satisfies the conditions listed in Corollary \ref{analytic_kmc}. Then, for a certain Selmer complex $\widetilde{C}_1$ of $\Lambda$-modules
				$$\mathrm{char}_{\Lambda^{E_p}_{\infty}} \left(H^2(\widetilde{C}_1) \otimes_{\Lambda} \Lambda^{E_p}_\infty \right) \bigm\vert \mathrm{char}_{\Lambda^{E_p}_{\infty}} \left( \dfrac{H^1(\widetilde{C}_1)\otimes_{\Lambda} \Lambda^{E_p}_\infty}{\mathrm{im}(c^\Pi)} \right) $$
			\end{thm}
		\par In \S3 we review theories of Lubin--Tate $(\phi,\Gamma)$-modules for $L/\mathbb{Q}_p$ over different base rings and properties of some cohomology theories of such modules. In particular, we highlight a discrepancy between Iwasawa cohomology of the Tate dual of an $L$-analytic representation (one whose Hodge--Tate weights away from the identity embedding are all 0) and the $\psi=1$ invariants of a $(\phi, \Gamma$)-module over the Robba ring which does not appear in the cyclotomic theory. This discrepancy is observed from calculations of Steingart \citep{Ste22b} stemming from the difference between locally $L$-analytic and overconvergent representations as soon as $L \neq \mathbb{Q}_p$. This will be crucial to constructing Selmer groups down in the next chapter. The following result is a conjecture, which will be the focus of a follow up project.
		
		\begin{conj} [See Conjecture \ref{steingart_conj}] 
			Given a trianguline $\mathrm{Gal}(\overline{L}/L)$-representation $S$ whose local Tate dual is $L$-analytic. Then
			$$ H^1_{\text{Iw}}(L_\infty/L,S) \otimes_\Lambda \Lambda_\infty^L \rightarrow D^{\dagger}_{\text{rig}}(S(\tau^{-1}))^{\psi=1} $$ is surjective with kernel of $\Lambda_\infty^L$-rank $([L:\mathbb{Q}_p]-1)\text{dim}(S)$.
		\end{conj}
        
		\par In \S4 we state the Schneider--Venjakob locally analytic regulator map $\mathcal{L}_{\text{SV}}$ which relies on the $p$-adic Fourier theory of Schneider--Teitelbaum and replaces the Perrin--Riou big logarithm map which no longer works in this setting. Using this regulator and the cohomology results of \S3 we state an $L$-analytic main conjecture for a class of Galois representations of GU(2,1) satisfying the expected technical condition plus one additional fixed Hodge--Tate weight coming from the $L$-analytic input. We prove one divisibility up to an `explicit reciprocity law' linking this distribution to complex $L$-values.
		\begin{thm}[See Theorem \ref{analytic_imc}] \label{intro3}
			Suppose $V$ satisfies the assumptions of Theorem \ref{intro1} and an extra `$E_p$-analyticity condition' (Assumption \ref{assumption_L-an}). Then there is a Selmer complex $\widetilde{C}_0^{E_p-\text{an}}$ of $\Lambda^{E_p}_\infty$-modules equipped with a complex map into $R\Gamma_{\text{Iw}}(E,V) \otimes_\Lambda^{\mathbb{L}} \Lambda_\infty^{E_p}$ such that $$\mathrm{char}_{\Lambda_{\infty,K}^{E_p}}H^2(\widetilde{C}_0^{E_p-\text{an}}) \bigm\vert (L^*_{p,E_p\text{-an}}) $$
		\end{thm}
		\par The recipe for going from $\widetilde{C}_1$ to $\widetilde{C}_0^{E_p-\text{an}}$ will use the kernel of the comparison map in the statement of the Conjecture (note that the conjecture tells us what the $\Lambda_\infty^L$-module rank of this kernel is).  Finally in \S5 we define analytic and overconvergent `descended' versions of our new Selmer group and describe their relations to the Bloch--Kato Selmer group. We need both versions since only one satisfies the right duality property and only the other satisfies the right Euler characteristic property. We obtain an upper bound on certain Bloch--Kato Selmer groups for $V$ in terms of the analytic distributions $L^*_{p,E_p\text{-an}}$ coming from the image of the Euler system under $\mathcal{L}_{\text{SV}}$.
		\begin{thm}[See Theorem \ref{analytic_BK_bound}]
			Suppose $V$ is as in Theorem \ref{intro3} and $\eta$ is an $E_p$-analytic character satisfying the technical conditions (listed in the full statement of the Theorem). If  $\eta(L^*_{p,E_p\text{-an}}) \neq 0 $ then $$ H^1_f(E,V(\eta^{-1})^*(1)) = 0 $$ 
		\end{thm}
	
		\section*{Acknowledgements} 
			\par This paper grew out of my PhD thesis project, largely thanks to the advice of my advisor David Loeffler, and also the valuable contributions of Otmar Venjakob, Ju-Feng Wu, Rustam Steingart, Antonio Lei, Xenia Dimitrakopoulou and the support of many others.
			
		\section{Kato's main conjecture for GU(2,1)}
		
		\subsection{Ordinary Galois representations for GU(2,1)}
		
		\par Suppose $E$ is an imaginary quadratic field of discriminant $-D$, with non-trivial automorphism $c: x \mapsto \bar{x}$. Identify $E \otimes_{\mathbb{Q}} \mathbb{R} \cong \mathbb{C}$ such that $\delta= \sqrt{-D}$ has positive imaginary part. Let $J \in \mathrm{GL}_3$ be the hermitian matrix
		\begin{equation*} J = \begin{pmatrix}
				0 & 0 & \delta^{-1} \\
				0 & 1 & 0\\
				-\delta^{-1} & 0 & 0 
			\end{pmatrix}.
		\end{equation*}
		Let $G$ be the group scheme over $\mathbb{Z}$ such that for a $\mathbb{Z}$-algebra $R$,
		$$ G(R) = \{ (g, \nu) \in \mathrm{GL}_3(\mathcal{O}_E \otimes_{\mathbb{Z}} R) \times R^\times : \prescript{t}{}{\bar{g}}Jg=\nu J \}.$$
		Then $G(\mathbb{R})$ is the unitary similitude group GU(2,1), which is reductive over $\mathbb{Z}_l$ for all $l \nmid D$.

        \par Following the exposition of \citep[\S1]{Man22}, given a regular algebraic essentially conjugate self-dual cuspidal representation $\Pi$ of GU(2,1) with coefficients in a number field $F$, we can associate a Galois representation $V=V_{\mathfrak{P}}(\Pi)$ defined over a $F_{\mathfrak{P}}$ dependent on a choice of prime $\mathfrak{P} \mid p$. Unlike \textit{op. cit.} we will work solely in the case where $p$ is inert in $E$. 

        \par In defining the Galois representation $V$, we have fixed an `identity' embedding $\text{id}: E \hookrightarrow F_{\Pi,\mathfrak{P}}$. We correspondingly get a conjugate embedding $\sigma$ induced by the non-zero element $\sigma \in \text{Gal}(E/\mathbb{Q})$. We can consider $V$ as a $G_{E_p}$-representation through either embedding, and we get two different local Galois representations, which we denote by $V=V_{\text{id}}$ and $V^\sigma$.
        
		\begin{prop}\textup{\citep[Theorem 1.2 (4)]{BLGHT11}} \label{galois_HT} The representation $V$ is de Rham at $p$. $V_{\text{id}}$ has Hodge--Tate weights $\{ 0, 1+b, 2+a+b \}$ and $V^\sigma$ has Hodge--Tate weights $\{ 0, 1+a, 2+a+b \}$. We call these the Hodge--Tate weights of $V$ at the identity and conjugate embeddings. Moreover the coefficients of $P_p(\Pi,p^2X)$ lie in $\mathcal{O}_{F_\Pi}$. \end{prop}
		
		\par From here on we will assume $\Pi_p$ is unramified, so the Hecke polynomial at $p$ defined. We want to make further assumptions about the local behaviour of $\Pi$.
		\begin{defn} \label{def_ord} We say $\Pi$ is ordinary at $p$ (with respect to $\mathfrak{P} \mid p$ of $F_\Pi$) if $P_p(\Pi,p^2X)$ has a factor $(1-\alpha_p X)$ with $\mathrm{val}_{\mathfrak{P}}(\alpha_p)=0$. \end{defn}
		
		\par Following arguments in \citep[\S2]{BLGHT11}, one can show that $\Pi$ is ordinary at $p$ if and only if the dual representation $V_{\mathfrak{P}}(\Pi)^*$ has a codimension 1 Galois-invariant subspace $\mathcal{F}^1_p V_{\mathfrak{P}}(\Pi)^*$ at $v$ such that $V_{\mathfrak{P}}(\Pi)^*/\mathcal{F}^1_v$ is unramified and $\mathrm{Frob}_v$ acts on the quotient by $\alpha_v$. Taking dual, such a subspace exists if and only if $V_{\mathfrak{P}}(\Pi)$ has a codimension 2 invariant subspace at $v$ with compatible action of $\mathrm{Frob}_{\bar{v}}$. By conjugate self duality, ordinary at all primes above $p$ implies a full flag of invariant subspaces at each place, which we will denote by $\mathcal{F}^i_v$:
		$$ 0 \subset \mathcal{F}^2_v V_{\mathfrak{P}}(\Pi) \subset  \mathcal{F}^1_v V_{\mathfrak{P}}(\Pi) \subset V_{\mathfrak{P}}(\Pi). $$
		\par Throughout this work we will assume ordinarity at all primes above $p$ as we will need the full flag in what follows.
		
		\begin{rmk} \label{crys_ord} If $V$ is crystalline and ordinary, there exist crystalline characters $\chi_{i,v}$ such that 
			\begin{equation*} V|_{\mathrm{Gal}(\overline{E_v}/E_v)} \sim \begin{pmatrix}
					\chi_{1,v} & * & * \\
					0 & \chi_{2,v} & * \\
					0 & 0 & \chi_{3,v} \\
				\end{pmatrix}.
			\end{equation*}
			The crystalline characters are 1-dimensional subquotients of $V|_{\mathrm{Gal}(\overline{E_v}/E_v)}$ with Hodge--Tate weights in decreasing order as $i$ increases, and this will help us understand how the crystalline Frobenius map $\phi$ acts on the Dieudonné module $D_{\mathrm{cris}}(V)$. When $p$ is split in $E$ these $\chi_{i,v}$ are powers of the cyclotomic character times a finite order character by \citep[Proposition B.04]{Con11}. When $p$ is inert they are powers of the Lubin--Tate character times powers of its conjugate times a finite order character by the same result.  \lozengeend \end{rmk}
		
		\subsection{Local Conditions}
		We refer to \citep{LZ20} for the definition of the rank $r$ Panchishkin subrepresentation for $r\in \mathbb{N}$. Below we recall a result which highlights why these are important.
		
		\begin{lem} \label{fukaya_kato}
			Suppose $V$ is a de Rham $p$-adic representation of $\mathrm{Gal}(\bar{E}/E)$ which admits a Panchishkin subrepresentation $V^+_v$ at an unramified prime $v \mid p$, satisfying $H^0(E_v, V/V^+_v)=H^0(E_v,(V^+_v)^*(1))=0$. Then locally at $v$,
			$$H^1_f(E_v, V) = H^1(E_v, V^+)$$
		\end{lem}
		\begin{proof}
			When $p$ is split in $E$ this is simply \citep[4.1.7]{FK06}. For $p$ non-split this is an application of Shapiro's Lemma to the above result, cf \citep[Lemma 1.11]{Man22}.
		\end{proof}

		\par We twist the GU(2,1) Galois representation $V$ by a Hecke character $\eta$ of $\mathbb{A}_E$ and see the different Panchishkin subrepresentations that can arise depending on our choice of $\eta$. We have a set of 3 Hodge--Tate weights attached to each of $\text{id}$ and $\sigma$. Suppose $\eta$ has infinity type $(s,r)$, then the Hodge--Tate weights of $V(\eta^{-1})$ are:
		
		\begin{align*} \{-s, 1+b-s, 2+a+b-s \} & \enspace \mathrm{at} \enspace id \\
			\{-r, 1+a-r, 2+a+b-r \} & \enspace \mathrm{at} \enspace \sigma. 
		\end{align*}
		
	We summarise the different possible Panchishkin subrepresentations below --- plotting the regions where different rank $r$ Panchishkin conditions are met in Figure \ref{fig1}, and tabulating the ranks of subrepresentations attached to each region. Any unlabelled regions are not $r$-critical for any $r$. \vspace{5pt}\\
		
		\begin{figure}[H]
			\includegraphics[width=14cm]{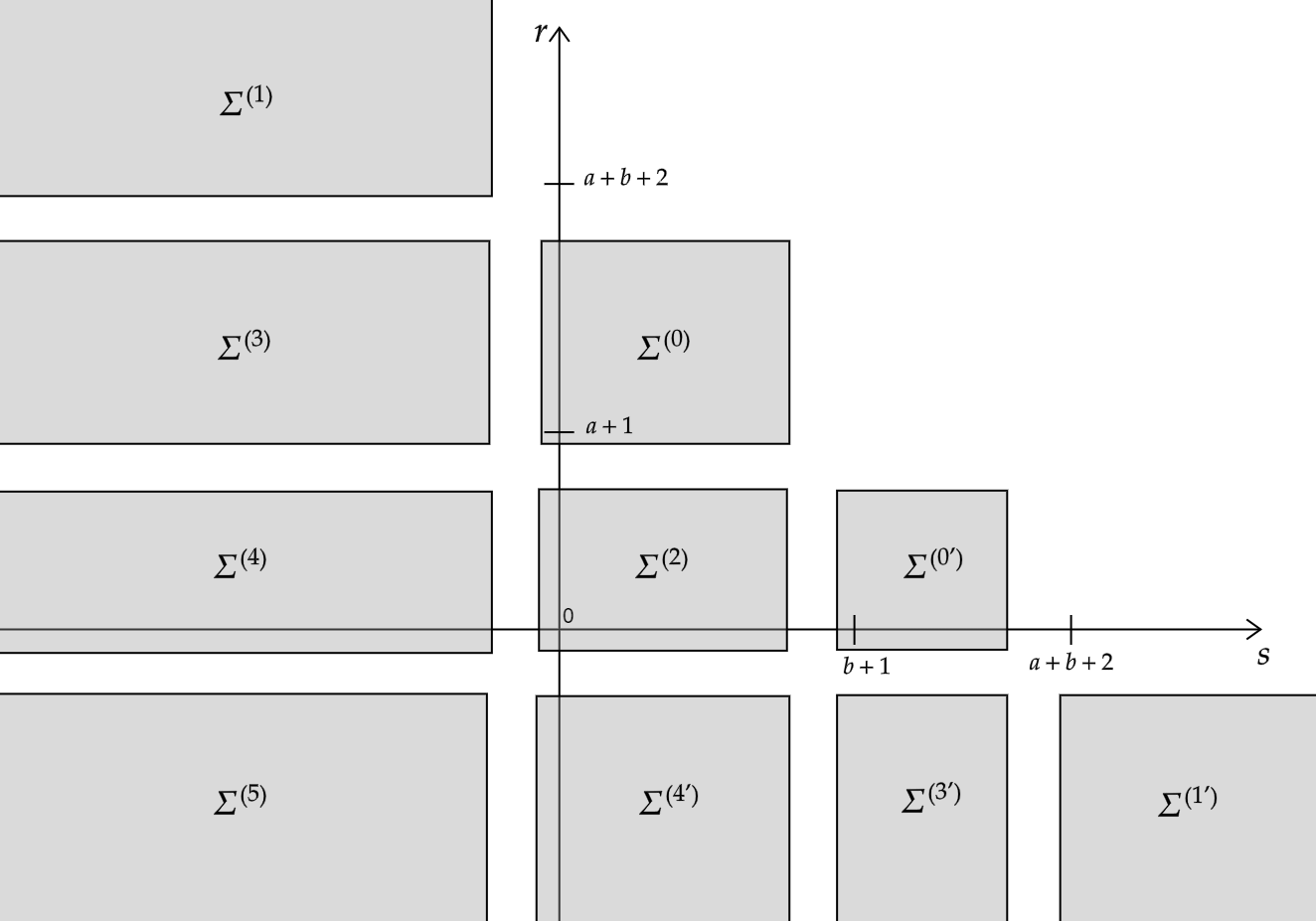}
			\caption{All possible (twisted) $r$-critical regions $V(\eta^{-1})$ as we vary $\infty$-type of $\eta$.}
			\label{fig1}
		\end{figure}
		%THIS DIAGRAM IS FLIPPED - NEED TO REDO IT 
		
		\begin{figure}[H] 	\label{panch_reps}
			{\renewcommand{\arraystretch}{1.2}
				\begin{tabular}{c|c|c|c}
					\bf Region & \bf Rank $r$ & \bf Panchishkin subrep. at $id$ & \bf Panchishkin subrep. at $\sigma$  \\ \hline
					$\Sigma^{(0)}$ & 0 & $\mathcal{F}^1$ & $\mathcal{F}^2$ \\
					$\Sigma^{(0')}$ & 0 & $\mathcal{F}^2$ & $\mathcal{F}^1$ \\
					$\Sigma^{(1)}$ & 0 & $V$ & 0 \\
					$\Sigma^{(1')}$ & 0 & 0 & $V$ \\
					$\Sigma^{(2)}$ & 1 & $\mathcal{F}^1$ & $\mathcal{F}^1$ \\
					$\Sigma^{(3)}$ & 1 & $V$ & $\mathcal{F}^2$ \\
					$\Sigma^{(3')}$ & 1 & $\mathcal{F}^2$ & $V$ \\
					$\Sigma^{(4)}$ & 2 & $V$ & $\mathcal{F}^1$ \\
					$\Sigma^{(4')}$ & 2 & $\mathcal{F}^1$ & $V$ \\
					$\Sigma^{(5)}$ & 3 & $V$ & $V$ \\
			\end{tabular}}
			
			\caption{The Panchishkin subrepresentations for $V(\eta^{-1})$ in each region.}
		\end{figure}

		\par When $p$ is inert, for $V^+_p$ to exist we must satisfy the extra condition that $\dim(V^+_{id})=\dim(V^+_{\sigma})$ since $p$ is preserved by conjugation. Hence $V^+_p$ only exists for twists attached to the regions intersecting the `diagonal' - i.e. preserved after applying conjugation map then dualising. The rank 1 Panchishkin subrepresentation $V^1$ for region $\Sigma^{(2)}$ satisfies this and therefore can exist. This matches the following Euler system construction of \citep[Theorem 12.3.1]{LSZ21}. We denote this Euler system as a class $c^\Pi \in H^1_{\mathrm{Iw}}(E[p^\infty]/E, V)$ 
		
		\par Part $(iii)$ of \textit{op. cit.} tells us that the localisation of the Euler system at $p$ vanishes when projected to the cohomology of the quotient $V/V^1$, so it lives entirely in the cohomology of $V^1$.
		
	\subsection{Kato's main conjecture}
		
		\par We need to make some assumptions on $V$ and the Hecke character twist we take. These are developed and explained in \citep[]{Man22}, and we will list them below for convenience. Here we take $\delta$ to be a character of $\text{Gal}(E[p^\infty]/E)_{\text{tors}}$, which uniquely determines a choice $\mathbb{Z}_p^2 \hookrightarrow \mathrm{Gal}(E[p^\infty]/E)$.
		
		\begin{ass} \label{assumption_V_delta} We assume $V$ and $\delta$ satisfy the following;
			\begin{itemize}
                \item $p \geq 5$,
				\item $p \nmid \# \mathrm{Cl}(E)$, 
                \item $p$ is unramified in $E$,
				\item $V$ is an irreducible $\mathrm{Gal}(\bar{E}/E)$ representation,
                \item $V$ is ordinary at all primes above $p$,
				\item $\Pi_w$ is unramified at all places $w \mid p$, in particular $V$ is crystalline at such $w$
				\item $V$ satisfies big image criteria,
				\item $e_{\delta}c_1^\Pi \neq 0$.
			\end{itemize}
		\end{ass}
		
		\begin{ass} \label{assumption_V_eta} Now suppose we have a character $\eta: \Lambda \rightarrow \mathcal{O}$ for a suitably large integral extension $\mathcal{O}/ \mathbb{Z}_p$ and denote by $\bar{\eta}$ its restriction to $\Gamma_{\text{tors}}$. We assume $V$ and $\eta$ satisfy the following;
			\begin{itemize}
				\item $V$ and $\bar{\eta}$ satisfy Assumption \ref{assumption_V_delta},
				\item For any subquotient $X$ of $\overline{T(\eta^{-1})}$, any $v \mid p$, $H^0(E_v, X)=H^2(E_v,X)=0$.
			\end{itemize}
		\end{ass}
		
		\par We will also use Nekov\'{a}\v{r}'s language of Selmer complexes, where it will be important to note the following definitions. Whilst this is all work from \citep{Nek06}, we find it convenient to quote results from \citep{KLZ17} where it has been restated for the purposes of Iwasawa theory.
		
		\begin{defn} A Selmer structure $\Delta = (\Delta_v)_{v \in S}$ for $M$ is a collection of homomorphisms $$\{\Delta_v: U_v^+ \rightarrow R\Gamma(E_v,M)\}_{v \in S}$$ from a complex $U_v^+$ into the complex of continuous $M$-valued cochains. \label{selmer_complex} \end{defn} 
		
		\par Using the notation of \citep[\S11.2]{KLZ17} we can write the associated Selmer complex by $\widetilde{\mathrm{R\Gamma}}(E,M;\Delta)$. This is defined as the mapping fibre of 
		$$ R\Gamma(E,M) \oplus \bigoplus_{v \in S} U_v^+ \xrightarrow{\text{loc}_v - \Delta_v} \bigoplus_{v \in S} R\Gamma(E_v,M). $$  We can compute the cohomology of this Selmer complex, which we will denote by $\widetilde{H}^i(E,M;\Delta).$
		
			\begin{defn} \label{selmer_simple} A Selmer structure is called simple if for all $v$ the map $$\Delta_v: H^i(U_v^+) \rightarrow H^i(E_v,M) $$ is an isomorphism for $i=0$, an injection for $i=1$, and the zero map for $i=2$. The local condition at $v$ for a simple Selmer structure is determined by the subspace $H^1(E_v,M;\Delta_v):=\Delta_v(H^1(U_v^+)) \subseteq H^1(E_v,M)$. \end{defn}
		
		\par Simple Selmer structures satisfy a nice statement of Poitou--Tate duality, and as such we think of these as classical Selmer groups dressed up as Selmer complexes. Non-simple Selmer complexes are often discussed as `extended Selmer groups' in Nekov\'a\v{r}'s work; in degree 1 they give groups which are not just cut out of $H^1$ of the original complex. We can also understand the $H^2$ of a simple Selmer complex easily using Poitou--Tate duality. Let us denote $M^\vee$ as the Pontryagin dual of a module $M$ and $\Delta^\vee$ as the dual Selmer structure, defined as the collection $\Delta_v^\vee$ of complexes which are the orthogonal complements of $\Delta_v$ under local Tate duality, as in \citep[Definition 11.2.6]{KLZ17}.
		
		\begin{prop} \textup{ \citep[Proposition 11.2.9]{KLZ17}} \label{selmer_pt_dual}
			Suppose $\Delta$ is a simple Selmer structure defined by local conditions $U_v \rightarrow R\Gamma(E_v, M)$ for $M$ a finite rank $\mathcal{O}$-module $T$ or $T \otimes \Lambda(-\bf{j})$. Then 
			$$ \widetilde{H}^i(E,M;\Delta) = \begin{cases}H^0\left(E, M\right) & \text { if } i=0, \\ \ker\left(H^1\left(E, M\right) \rightarrow \bigoplus_{v \in S} \frac{H^1\left(E_v, M\right)}{H^1(E_v,M;\Delta_v)}\right) & \text { if } i=1, \\ \ker\left(H^1\left(E, M^{\vee}(1)\right) \rightarrow \bigoplus_{v \in S} \frac{H^1\left(E_v, M^{\vee}(1)\right)}{H^1(E_v,M^\vee;\Delta_v^\vee)}\right)^{\vee} & \text { if } i=2 .\end{cases} $$
			
		\end{prop}

		\par We can define a Selmer structure $\Delta^1$ by the unramified local condition for $v \nmid p$, and above $p$ by: $$\Delta^1_v: R\Gamma(E_v, T^1_v) \hookrightarrow R\Gamma(E_v,T).$$ This is the local condition attached to the rank 1 Panchishkin subrepresentations $T_v^1$ attached to twists by characters in the `rank 1' $\Sigma^{(2)}$ region of Figure \ref{fig1}. This will be used in the statement of Kato's main conjecture, i.e. the Euler system version of the main conjecture. \par When $p$ splits we can similarly define $\Delta^0$ for the inclusion $$\Delta^0_v: R\Gamma(E_v, T_v^0) \hookrightarrow R\Gamma(E_v,T)$$ attached to the `rank 0' region. This is used to state a main conjecture in this setting, developed in \citep{Man22}. This local condition does not exist for $p$ inert, and a large amount of this paper will be dedicated to constructing its replacement. Note that  $\Delta^1$ is a simple local condition inside the Galois cohomology of $V(\eta^{-1})$ by the cohomology vanishing of Assumption \ref{assumption_V_eta}. We won't come across a non-simple Selmer structure until much later. We also rely on certain base change and perfectness properties of Selmer complexes listed in \citep{Man22}.
		
		\par Now we can state Kato's main conjecture, i.e. a statement without a $p$-adic $L$-function. This is a corollary of the work of B\"uy\"ukboduk and Ochiai which generalises arguments of Mazur--Rubin. When we define characteristic ideal, it is the usual definition for a Noetherian Krull domain $R$; given a torsion $R$-module $M$ and height 1 prime $\mathfrak{p}$ we first take $l_\mathfrak{p}$ to be the length of $M_\mathfrak{p}$ as an $R_\mathfrak{p}$-modules. For $M$ finitely generated, this will be 0 outside a finite set of primes. Then we can define
		$$ \mathrm{char}_R(M) = \prod_{\mathfrak{p} \subset \Lambda \enspace \text{height} \enspace 1} \mathfrak{p}^{l_\mathfrak{p}}. $$ This tells us that a divisibility of characteristic ideals is really a comparison of lengths of localised modules at height 1 primes. Thus we can interpret such a divisibility as a finite list of inequalities of such lengths. Now Kato's main conjecture for GU(2,1) is as follows.
		
		\begin{thm} \label{kato_mc} Assume $\delta$ is a character of $\Gamma_{\text{tors}}$ such that $V$ and $\delta$ satisfy Assumption \ref{assumption_V_delta}. Then the following statements hold:
			\begin{enumerate}[i)]
				\item $e_{\delta} \widetilde{H}^2(E,\mathbb{T};\Delta^1)$ is $\Lambda_{\delta}$-torsion,
				\item $e_{\delta} \widetilde{H}^1(E,\mathbb{T}; \Delta^1)$ is torsion-free of $\Lambda_{\delta}$-rank 1,
				\item $\mathrm{char}_{\Lambda_{\delta}} \left( e_{\delta} \widetilde{H}^2(E,\mathbb{T};\Delta^1) \right) \bigm\vert \mathrm{char}_{\Lambda_{\delta}} \left( \dfrac{e_{\delta} \widetilde{H}^1(E,\mathbb{T};\Delta^1)}{ e_{\delta}c^\Pi_1} \right). $
			\end{enumerate} 
		\end{thm}
		\begin{proof}
			Apply \citep[Theorem 3.6]{BO17} with $R=\Lambda_{\delta} \cong \mathbb{Z}_p[[X_1,X_2]]$.
		\end{proof}
	
	\par In Kato's work, given an ordinary modular form $f$ (satisfying suitable assumptions) with respective Galois representation $V_f$, he used a Perrin--Riou big logarithm map $$ \mathcal{L}: H_{\text{Iw}}^1(\mathbb{Q}_p^{\text{cyc}}/\mathbb{Q}_p,V_f) \rightarrow \Lambda(\mathbb{Z}_p^\times) \otimes_{\mathbb{Q}_p} D_{\text{cris}}(S) $$  for a suitable character $S$. Under this map, Kato's Euler system goes to the $p$-adic $L$-function. With some exact sequence computations, one obtains one divisibility of a main conjecture with the $p$-adic $L$-function. This strategy has been successfully imitated for GU(2,1) with $p$ split in \citep[Theorem 4.3]{Man22}, up to an explicit reciprocity law, and this is used to bound a Bloch--Kato Selmer group when $p$-adic distribution taken as the image of the Euler system is non-vanishing. One hopes in the future that this distribution will be related to a $p$-adic $L$-function, then we will get a statement verifying cases of the Bloch--Kato conjecture.
	
	\begin{rmk} \label{tor_spectral_sequence}
		\par The descent of Kato's main conjecture for GU(2,1) requires a two step base change from $\Lambda(\mathbb{Z}_p^2)$-modules to $\Lambda(\mathbb{Z}_p)$-modules for a choice of $\mathbb{Z}_p$-subextension down to $\mathcal{O}$-algebras for a suitable integral extension $\mathcal{O}/\mathbb{Z}_p$. We can write out the base change in the form of a more general spectral sequence which we use later. Begin with the homological first quadrant Tor spectral sequence from \citep[Example 15.62.1]{SP} for a commutative ring $S$, an $S$-module $R$ and a bounded chain complex $K_\bullet$ of $S$-modules;
		$$ (E_2)_{i,j} = \mathrm{Tor}^S_i(H_j(K_\bullet), R) \Rightarrow H_{i+j}(K \otimes^{\mathbb{L}}_S R). $$
		\par We aim to turn this into a cohomological sequence by substituting $K_i=C^{d-i}$ for $d$ the homological degree of $K$, but it would no longer be a first quadrant spectral sequence
		$$  \text{Tor}^{-j}_{\Lambda_\infty^L}(H^{j}(C^\bullet),K) \Rightarrow H^{i+j}(C^\bullet \otimes^{\mathbb{L}}_{\Lambda_\infty} K) .$$
		This spectral sequence is dual under adjunction to the Groethendieck spectral sequence coming from composition of the right derived functors of Galois cohomology and $\text{Ext}^i_S(-,R)$ - since $-\otimes_S R$ is right exact, this is the natural way to get a spectral sequence composing derived tensor product with cohomology. Computing the 7-term long exact sequence of the homological spectral sequence and then dualising, we obtain;
		
		\[ \begin{tikzcd}[column sep=small]
			0 \rightarrow H_1(K_\bullet) \otimes_S R & H_1(K_\bullet \otimes^{\mathbb{L}}_S R) & \mathrm{Tor}_1^S(H_0(K_\bullet),R) \\
			H_2(K_\bullet) \otimes_S R & & \ker \left(H_2(K_\bullet \otimes^{\mathbb{L}}_S R) \rightarrow \mathrm{Tor}_2^S(H_0(K_\bullet),R)\right) \\
			\mathrm{Tor}_1^S(H_1(K_\bullet),R) & & \mathrm{Tor}_3^S(H_0(K_\bullet),R)
			\arrow[from=1-1, to=1-2]
			\arrow[from=1-2, to=1-3]
			\arrow[from=1-3, to=2-1]
			\arrow[from=2-1, to=2-3]
			\arrow[from=2-3, to=3-1]
			\arrow[from=3-1, to=3-3]
		\end{tikzcd}\]
	
		 We will lie in the simplifying case where $\tau: S \twoheadrightarrow R$ is a map with a principally generated kernel, say $\ker(\tau)=a$ (if the kernel has is a regular sequence of length $n$ then we can simply iterate this process and think of it as an $n$-step descent argument). We have computations of the Tor groups for an arbitary $S$-module $M$:
		$$ \operatorname{Tor}_i^S(R , M) \cong \begin{cases}M /aM & i=0 \\ M[a] & i=1 \\ 0 & \text { otherwise }\end{cases}. $$ 
		
		The above exact sequence of homology groups turns into short exact sequences of cohomology groups for $0<i\leq d$;
		\[ 	0 \rightarrow H^{i-1}(C^\bullet)/a \rightarrow H^{i-1}(C^\bullet \otimes^{\mathbb{L}}_S R ) \rightarrow H^{i-1}(C^\bullet)[a] \rightarrow 0 \] 
		When the complexes $C^\bullet$ are in fact Selmer complexes $\widetilde{R\Gamma}(\Delta)$ for a Selmer structure $\Delta$, this is a standard implementation of base change properties of Selmer complexes.
	\end{rmk}
	
	\par We would like to do this for inert $p$, but have some technical obstacles; notably that there is no Perrin--Riou type big logarithm map we can use since $E_p$ is a non-trivial extension of $\mathbb{Q}_p$. The replacement we use is by work of Schneider and Venjakob, and with some extra technical assumption fixing one Hodge--Tate weight of $V$ we can obtain a result. The following three sections will develop the necessary theory to state and implement the Schneider--Venjakob regulator map to state a main conjecture. This spectral sequence will still prove useful, where $S$ is taken to be a certain flat $\Lambda$-algebra rather than $\Lambda$ itself.

		\section{Kato's main conjecture over distribution algebras}

	\subsection{The two-variable distribution algebra}
		
		\par We want to develop our theory to work over a base ring $\Lambda \hookrightarrow A$ such that $A$ will contain unbounded and $p$-adically non-integral distributions, but will also have a module theory and prime ideal theory rich enough to not lose information when we base change. A first candidate algebra exists due to work of Pottharst in \citep{Pot12}. When developing Iwasawa theory for modular forms at supersingular primes, which have unbounded $p$-adic $L$-functions, in \textit{op. cit.} the cyclotomic Iwasawa algebra $\Lambda$ is replaced with a distribution algebra $\Lambda_\infty$ coming from rigid geometry - after applying the Mellin transform (and forgetting about torsion in the Galois group) these both live in $\mathbb{Q}_p[[T]]$ but the distribution algebra sees a large class of unbounded distributions. Pottharst's construction was for a $\mathbb{Z}_p$-extension, whereas we must work with a $\mathbb{Z}_p^2$-extension variant which is of larger dimension. The goal of this section is to base change Theorem \ref{kato_mc} via the inclusion $ \Lambda \hookrightarrow \Lambda_\infty$, after which we adapt Pottharst's construction further to work over a carefully chosen quotient of $\Lambda_\infty$. 
		
		\par We will adapt Pottharst's definition for our $\mathbb{Z}_p^2$-extension, presented as follows. Let $\mathfrak{X}=\mathrm{Spf}(\Lambda)^{\text{rig}}$, the cyclotomic $p$-adic weight space, a rigid analytic space of dimension 2. By construction this space parameterises characters of $\Lambda$, or by local class field theory Hecke characters of $E$ of $p$-power conductor. We can then define $\Lambda_\infty = \Gamma(\mathfrak{X}, \mathcal{O}_{\mathfrak{X}})$ like Pottharst does; although note we are working with a $\mathbb{Z}_p^2$-extension  whereas Pottharst works with a cyclotomic $\mathbb{Z}_p$-extenstion. Our $\Lambda_\infty$ is a Fr\'echet--Stein algebra, a notion developed in \citep{ST02}, and an inverse limit of dimension 2 rings $\Lambda_\infty = \varprojlim_n \Lambda_n$ where $\Lambda_n =\widehat{\Lambda[\tfrac{\mathfrak{m}^n}{p}]}[\tfrac{1}{p}]$, for $\mathfrak{m}$ the maximal ideal of $\Lambda$ (which is a local ring since we assume the Galois group is torsion free). These $\Lambda_n$ are in turn the sections of an increasing admissible affinoid cover $\{ \mathfrak{X}_n \}_{n \in \mathbb{N}}$ of $\mathfrak{X}$. 
        \par Recall that  inside $\Gamma=\text{Gal}(E[p^\infty]/E)$ we can determine a torsion free $\mathbb{Z}_p^2$-extension $e_\delta\Gamma$ by a choice of character $\delta$ of $\Gamma_{\text{tors}}$. Similarly, this idempotent lifts to one on $\Lambda_\infty$, and we can define $ \Lambda_{\infty,\delta}=\mathcal{O}(\mathfrak{X}^0_\delta)$ where $\mathfrak{X}^0_\delta$ is the connected component of $\mathfrak{X}$ corresponding to $\text{Spf}(\Lambda_{\delta})^{rig}$. In Schneider--Teitelbaum's $p$-adic Fourier theory \citep[Theorem 3.6]{ST01} they show that each $\mathfrak{X}^0_{\delta}$ is $\mathbb{Q}_p$-isomorphic to an open rigid unit ball, therefore $\mathfrak{X}$ is a finite disjoint union of open balls. The $\mathfrak{X}_n$ are the images of closed polydisks of radius $r_n<1$ under this isomorphism. Thus each $\Lambda_n$ is therefore isomorphic to a disjoint union of Tate algebras $L \langle T_1,T_2 \rangle$ and is therefore noetherian.

        \begin{lem} \label{cyc_flat}
			The injection $\Lambda \hookrightarrow \Lambda_\infty$ is flat, and  $\Lambda[\tfrac{1}{p}] \hookrightarrow \Lambda_\infty$ is faithfully flat. Thus, for a Selmer structure $\Delta$, $$ \widetilde{H}^i(E,\mathbb{T};\Delta) \otimes_\Lambda \Lambda_\infty = \widetilde{H}^i(E,\mathbb{T} \otimes_\Lambda \Lambda_\infty;\Delta)  $$
		\end{lem}
		\begin{proof} 
			\par First note that $\Lambda \hookrightarrow \Lambda[\tfrac{1}{p}]$ is a localisation map therefore flat, and the composition of flat maps is also flat. Thus we only need to show faithful flatness of $\Lambda[\tfrac{1}{p}] \hookrightarrow \Lambda_\infty$. We note first that flatness is a local property, so we will work over localisations at maximal ideals $\mathfrak{m} \subset \Lambda[\tfrac{1}{p}]$. By \citep[Lemma 7.1.9]{dJ95} we have a bijection $$ \mathfrak{m} \in \text{maxSpec}(\Lambda[\tfrac{1}{p}]) \leftrightarrow \bigcup_n \mathfrak{X}_n \ni x_{\mathfrak{m}}$$ and a morphism of local rings $$ f_{\mathfrak{m}}:\Lambda[\tfrac{1}{p}]_{\mathfrak{m}} \rightarrow \Lambda_{\infty,x_{\mathfrak{m}}}$$ inducing an isomorphism $\widehat{f_{\mathfrak{m}}}$ on completions. Therefore by \citep[Definition 41.9.1 (9)]{SP} $f_{\mathfrak{m}}$ itself is flat and by locality $\Lambda[\tfrac{1}{p}] \hookrightarrow \Lambda_\infty$ is flat. Since \citep[Lemma 7.1.9]{dJ95} shows a surjection of closed points, the injection is moreover faithfully flat. The isomorphism of cohomology groups then follows from \citep[Theorem 1.12 case (4)]{Pot13}.
			
		\end{proof}

       \par  Since the $\Lambda_n$ are noetherian we can use \citep[Ch. VII \S4.5]{Bou98} to define $\mathrm{char}_{\Lambda_n}$ for torsion $\Lambda_n$-modules. Now we need a notion of $\Lambda_\infty$-modules which are `suitably nice' so this can be stitched together to define a characteristic ideal over $\Lambda_\infty$. This notion will come from the underlying geometry. Following \citep[\S3]{ST02} we call a $\Lambda_\infty$-module $M$ coadmissible if it arises as the global sections of a coherent analytic sheaf $\mathscr{M}$ over $\mathfrak{X}$. We can equivalently interpret this algebraically; that $M=\varprojlim_n M_n$ with each $M_n$ finitely generated as a $\Lambda_n$-module such that $\Lambda_n \otimes_{\Lambda_{n+1}} M_{n+1} \xrightarrow{\sim} M_n$ for all $n$. To define a suitable characteristic ideal we will need to work with coadmissible torsion modules.  We note that Pottharst provides an explicit description for coadmissible torsion $\Lambda_\infty$-modules when working over a $\mathbb{Z}_p$-extension which we can no longer use in the $\mathbb{Z}_p^2$ setting; however it will be useful in the next section.

		\begin{lem} \label{cyc_char}
			Let $M=\varprojlim_n M_n$ be a coadmissible torsion $\Lambda_\infty$-module. Then $$\mathrm{char}_{\Lambda_\infty}(M) = \varprojlim_n \mathrm{char}_{\Lambda_n}(M_n)$$ is a well defined ideal of $\Lambda_\infty$. If $M= N \otimes_\Lambda \Lambda_\infty$ for a finitely generated torsion $\Lambda$-module $N$, $M$ is a coadmissible torsion $\Lambda_\infty$-module and $$ \mathrm{char}_{\Lambda_\infty}(M) = \mathrm{char}_{\Lambda}(N)\cdot\Lambda_\infty.$$
		\end{lem}
		\begin{proof}
			\par First we show this definition is well-defined.	Note that each $M_n$ must be torsion and each $\Lambda_n$ noetherian, so we can write for each $n$, $$\mathrm{char}_{\Lambda_n} (M_n) = \prod_{\mathfrak{p} \subset \Lambda_n,\enspace \text{ht}(\mathfrak{p})=1}\mathfrak{p}^{l_{\mathfrak{p},n}}$$ following the ideas of \citep[\S4.7]{Bou98}, where this is a finite product. We want to show this forms an inverse system when $M$ is coadmissible. Note that such a finite product need not exist over $\Lambda_\infty$, which is why we really need coadmissibility.
			
			\par We let $p_{n+1}: \Lambda_{n+1} \rightarrow \Lambda_n$, then $$ \mathrm{char}_{\Lambda_{n+1}}(M_{n+1}) \otimes_{\Lambda_{n+1}} \Lambda_n = \prod_\alpha \mathfrak{p}_{\alpha}^{l_\alpha} \otimes_{\Lambda_{n+1}} \Lambda_n. $$
			Note that we have inclusions for each $n$ $$ \mathrm{Spec}(\Lambda_n) \hookrightarrow \mathrm{Spec}(\Lambda_{n+1}). $$
			By coadmissibility, if one of the $\mathfrak{p}_\alpha$ lies in the image of this inclusion, $$ \mathfrak{p}_{\alpha}^{l_\alpha} \otimes_{\Lambda_{n+1}} \Lambda_n = \mathfrak{p}_{\alpha}^{l_\alpha} \cdot \Lambda_n = p_{n+1}(\mathfrak{p}_{\alpha}^{l_\alpha}). $$
			If not, then $p_{n+1}(\mathfrak{p}_\alpha) = \mathfrak{p}_{\alpha} \otimes_{\Lambda_{n+1}} \Lambda_n = \Lambda_n$ so this term will not contribute. Thus they form an inverse system and $\mathrm{char}_{\Lambda_\infty}$ is well defined as an ideal.
			
			\par Now we show the second part. Coadmissibility of $M$ is clear since $N$ is finitely generated and so each $N \otimes_\Lambda \Lambda_n$ is too. For torsionness; given $m \in N$ there is $\lambda \in \Lambda \hookrightarrow \Lambda_\infty$ such that $\lambda \cdot(m \otimes \lambda^\prime)=0$ for any $\lambda^\prime$. By flatness of Lemma \ref{cyc_flat}, for any prime ideal $\mathfrak{p} \subset \Lambda$, integers $n$ and $r$, $\left( \Lambda/\mathfrak{p}^r \right) \otimes_\Lambda \Lambda_n \cong \Lambda_n / (\mathfrak{p}^r \cdot \Lambda_n)$. By the proof of Lemma \ref{cyc_flat}, there is a surjection $$ \bigcup_n \mathrm{Spec}(\Lambda_n) \twoheadrightarrow \mathrm{Spec}(\Lambda)\backslash \{p\}.$$
			\par Suppose $P$ is a finitely generated pseudo-null $\Lambda$-module, then we claim that $P \otimes_\Lambda \Lambda_n$ is a pseudo-null $\Lambda_n$-module for each $n$. By hypothesis we have $P_{\mathfrak{p}}=0$ for every height 1 prime $\mathfrak{p}$ of $\Lambda$. Thus for $n$ large enough such that $\mathfrak{p}$ has preimage $\mathfrak{q}$ in Spec($\Lambda_n$), $$0 = P_{\mathfrak{p}} \otimes_\Lambda \Lambda_n \cong (P \otimes_\Lambda \Lambda_{\mathfrak{p}}) \otimes_\Lambda \Lambda_n \cong P \otimes_\Lambda (\Lambda_{\mathfrak{p}} \otimes_\Lambda \Lambda_n) \cong (P \otimes_\Lambda \Lambda_n)_{\mathfrak{q}} . $$
			Choose $\mathfrak{q}$ a generic height 1 prime of $\Lambda_n$ for $n$ large enough and let $\mathfrak{p}$ be its preimage in $\mathrm{Spec}(\Lambda)$, which is of height $\leq 1$ since chains of prime ideals decrease in length (perhaps non-strictly) when we pull back by the inclusion map. When it is height 1 we apply the above chain of equalities to this $\mathfrak{p}$; hence $(P \otimes_\Lambda \Lambda_n)_{\mathfrak{q}}=0$. When $\mathfrak{p}$ is height 0, since $P$ is a torsion module we find that $P_\mathfrak{p} = P_{(0)} = 0$ also, so the claim is proved.
			
			\par By the structure theorem of torsion $\Lambda$-modules, our torsion $\Lambda$-module $N$ is pseudo-isomorphic to one of the form $\prod_\alpha \Lambda/\mathfrak{p}^{r_\alpha}_\alpha$ for a finite set of $\alpha$ indexing a multiset of powers of height 1 primes. Taking flat base change by $\otimes_\Lambda \Lambda_n$ kills the pseudo-null kernel and cokernel by the above claim, thus giving an actual isomorphism $$ M_n \cong \prod_\alpha \left( \Lambda/\mathfrak{p}^{r_\alpha}_\alpha \right) \otimes_\Lambda \Lambda_n \cong \prod_\alpha \Lambda_n /(\mathfrak{p}^{r_\alpha}_\alpha \cdot \Lambda_n). $$ as can we swap the tensor product and finite product. Taking characteristic ideals therefore tells us that 
			$$ \text{char}_{\Lambda_n}(N \otimes_\Lambda \Lambda_n) = \text{char}_{\Lambda}(N) \cdot \Lambda_n.$$
			Now we can take inverse limits (using the first part of the Lemma) to obtain the full result.
			
		\end{proof}
		
		\par For the following remark we will define two properties of non-noetherian rings which are important to distinguish.
		
		\begin{defn} \label{bezout_prufer}
			Let $A$ be a commutative integral domain. We say $A$ is:
			\begin{itemize}
				\item A Pr\"ufer domain if every non-zero finitely generated ideal is invertible.
				\item A Bezout domain if the sum of any two principal ideals is also principal.
			\end{itemize}
			If $A$ is moreover noetherian then being Pr\"ufer is equivalent to being a Dedekind domain and being Bezout is equivalent to being a principal ideal domain.
		\end{defn}
		\par There are many equivalent definitions to being a Pr\"ufer domain; a ring $A$ is Pr\"ufer if and only if all localisations at $\mathfrak{p} \in \text{Spec}(A)$ are valuation domains (that is for any $x \in \text{Frac}(A)$ either $x \in A$ or $x^{-1} \in A$), if and only if every torsion free $A$-module is flat. These all resemble different properties of Dedekind domains when the noetherian criterion is removed. These definitions and several other equivalent ones are listed in \citep[VII.2 Ex.12]{Bou98}.
		
		\begin{rmk}
			In Pottharst's work \citep[Proposition 1.1]{Pot12} he uses $\Lambda_\infty$ for a $\mathbb{Z}_p$-extension which is a projective limit of dimension 1 rings. Here he shows that the ideal $\mathrm{char}_{\Lambda_\infty}(M)$ is a principal ideal. We can no longer do this as it relies on the property of his $\Lambda_\infty(\mathbb{Z}_p)$being a Bezout domain, whereas in this work it is a limit of dimension 2 rings and cannot be Bezout. However the main conjecture over $\Lambda_\infty(\mathbb{Z}_p^2)$ has the corollary that the characteristic ideals of certain Selmer groups are principally generated by the given $p$-adic distribution. A similar (but more subtle) issue about showing principality of characteristic ideals will appear later on. \lozengeend
		\end{rmk}
		
		\par We let $\mathbb{V}_\infty = \mathbb{V} \otimes_{\Lambda} \Lambda_\infty = \mathbb{T} \otimes_{\Lambda} \Lambda_\infty$, similar to universal twisting by $\Lambda$. Note that we have inverted $p$ so we lose the notion of integrality and the choice of integral lattice $T$ in $V$ no longer matters. We can then specialise at a character $\eta:\Lambda \rightarrow \mathcal{O}$, which lifts to a closed point of the character variety $x_\eta \in \mathfrak{X}$. Evaluation at $x_\eta$ then gives a unique character on its global sections $\tilde{\eta}: \Lambda_\infty \rightarrow K$, where $K=\text{Frac}(\mathcal{O})$. We will often drop the tilde from our notation. We are interested in the base change map $- \otimes_{\Lambda_\infty, \eta} K $ and how the corresponding derived tensor product interacts with taking cohomology of the Selmer complexes we have seen so far.
		\

		\begin{cor} \label{cyc_kmc}
			Suppose $V$ and $\delta$ satisfy Assumption \ref{assumption_V_delta} and $p$ is inert in $E$. Then,
			\begin{enumerate}[i)]
				\item $e_\delta \widetilde{H}^2(E,\mathbb{V}_\infty; \Delta^1)$ is coadmissible $\Lambda_{\infty, \delta}$-torsion
				\item $e_\delta \widetilde{H}^1(E,\mathbb{V}_\infty; \Delta^1)$ is coadmissible and torsion-free of $\Lambda_{\infty, \delta}$-rank 1
				\item $\mathrm{char}_{\Lambda_{\infty,\delta}} \left( e_\delta \widetilde{H}^2(E,\mathbb{V}_\infty; \Delta^1) \right) \bigm\vert \mathrm{char}_{\Lambda_{\infty, \delta}} \left( \dfrac{e_\delta \widetilde{H}^1(E,\mathbb{V}_\infty; \Delta^1)}{\mathrm{im}(c^\Pi)} \right) $
			\end{enumerate}
		\end{cor}
		\begin{proof}
			\par All parts follow from Lemma \ref{cyc_flat} and Lemma \ref{cyc_char} applied to Theorem \ref{kato_mc}.
		\end{proof}

		\subsection{The locally analytic distribution algebra}
		
		\par One key reason why Pottharst's approach with $\Lambda_\infty$ of a $\mathbb{Z}_p$-extension is successful is that he is working with a limit of dimension 1 rings, whereas our $\Lambda_\infty$ is a limit of dimension 2 rings. To exploit the ring-theoretic properties he uses, we must find some way to cut the $\mathbb{Z}_p^2$ version down to the right size. By using $\Lambda_\infty$ we also fail to incorporate the choice of identity embedding and action of complex conjugation which fixes our inert prime $p$ in $E_p$. These two points suggest that we ought to restrict attention to characters which depend on $E_p$ in a more specific way.
		\par We can think of the weight space $\mathfrak{X}$ as an analytic character variety defined in \citep[\S1.2]{BSX20}, and recall the following terminology; for a finite extension $L/\mathbb{Q}_p$ we say a $p$-adic Galois representation $W$ is $L$-analytic if it is Hodge--Tate with all weights at all non-identity embeddings equal to 0. We will call a Galois representation $S$ co-$L$-analytic if $S^*(1)$ is $L$-analytic. 
		\par For a generic Galois representation and general $[L:\mathbb{Q}_p]=d$ this is quite restrictive, requiring us to fix $([L:\mathbb{Q}_p]-1) \times \dim(W)$ Hodge--Tate weights, but in our application we will take $W$ to be an $L$-analytic character where $L=E_p$ has degree 2 over $\mathbb{Q}_p$; this imposes a condition on a single Hodge--Tate weight which is more acceptable. This numerology is important and will crop up repeatedly as we study the consequences of the $L$-analytic theory. The structure theorem of \cite[Proposition B.04]{Con11}, which tells us that the crystalline $L$-analytic $\mathrm{Gal}(L_\infty/L)$-characters are exactly powers of the Lubin--Tate character $\chi_{\text{LT}}$ associated to $L$ times an unramified character - so in a sense they are characters `generated' by the Lubin--Tate character. We contrast this with the study of $G_{\mathbb{Q}_p}$-representations where these are `generated' by $\chi_{\text{cyc}}$.
        \par The character $\chi_{\text{LT}}$ is defined in Lubin and Tate's formulation of local class field theory and generalises the function of the cyclotomic character in cutting out cyclotomic integers. In particular, after fixing a uniformiser $\varpi_L$ of $L$ (on which it is dependent) it can be written as
		$$ \chi_{\text{LT}}: \Gamma_L \cong L^\times = \mathcal{O}_L^\times \times \varpi_L^{\mathbb{Z}} \rightarrow \mathcal{O}_L^\times,$$ and when such $\varpi_L$ satisfies $N_{L/\mathbb{Q}_p}(\varpi_L) \in p^{\mathbb{Z}}$ we have the equality of characters $N_{L/\mathbb{Q}_p} \circ \chi_{\text{LT}} = \chi_{\text{cyc}}$.
		\par We now take $\mathfrak{X}_L$ to be the 1-dimensional locus of $L$-analytic characters in $\mathfrak{X}$, which is the character variety defined in \citep[\S1.2]{BSX20}. Now we define the $L$-analytic distribution algebra as the corresponding ring of global sections, equipped with the surjection $$\pi_L: \Lambda_\infty \rightarrow \Lambda_\infty^L := \Gamma(\mathfrak{X}_L, \mathcal{O}_{\mathfrak{X}_L}).$$
		\par This $\Lambda_\infty^L$ is also a Fr\'echet--Stein algebra, i.e. we can express $\Lambda^L_\infty = \varprojlim_n \Lambda^L_n$; equivalently $\mathfrak{X}_L$ is a quasi-Stein space with admissible cover $\mathfrak{X}_{L,n} = \mathfrak{X}_n \cap \mathfrak{X}_L$. However, as soon as $L \neq \mathbb{Q}_p$, $\Lambda_\infty^L$ does not have a model as a formal scheme of an Iwasawa algebra of a $\mathbb{Z}_p$-subextension of $E[p^\infty]$. Needless to say, the non-existence of this subextension is exactly why we need to resort to using these distribution algebras rather than a working over a quotient of $\Lambda$ (compare this again with anticyclotomic Iwasawa theory where such a subextension does exist). We will sometimes discuss $\Lambda_\infty^L$ for general $L$, and sometimes we will focus only on the case where $L=E_p=\mathbb{Q}_{p^2}$; in context it will be clear when something applies to general $L$.
		\par The lack of such a formal scheme model is an important reason as to why the theory goes wrong. For example, the flatness argument presented in Lemma \ref{cyc_flat} will no longer work if we replace $\mathfrak{X}$ by $\mathfrak{X}_L$. Despite the surprising fact that there is an embedding $\Lambda \hookrightarrow \Lambda_\infty^L$ for arbitrary $L$, given in \citep{BSX20} where the ring of bounded functions on $\mathfrak{X}$ is just $\Lambda[\tfrac{1}{p}]$, this map is not flat. By \citep[Corollary 5.9]{Laz67}, since $\text{Spec}(\Lambda_\infty)$ has a finite number of connected components (precisely $\# \Gamma_{\text{tors}}$), $\Lambda_\infty^L$ is flat over $\Lambda_\infty$ only if the map $\mathfrak{X}_L \subset \mathfrak{X}$ is both open and closed. It is by definition closed but is certainly not open - and hence the map is not flat.
		
		\begin{rmk} This is not the only algebraic issue we face. \label{padic_fourier} The work of \citep{ST01} shows that $\mathfrak{X}_L$ is isomorphic to the rigid open unit disk $\mathbb{B}_{[0,1)}$ only after base extending by a complete field containing the Lubin--Tate period $\Omega_{\text{LT}}$. When $L=\mathbb{Q}_p$ then $\Omega_{\text{cyc}} \in \mathbb{Q}_p$, but when $L \neq \mathbb{Q}_p$ $\Omega_{\text{LT}}\in \widehat{L_\infty}$ which is transcendental over $\mathbb{Q}_p$, hence $\mathfrak{X}_L$ can be thought of as a transcendentally twisted disk. Since the open disk is a limit of closed disks of radius converging to 1, this tells us that when for $K/\widehat{L_\infty}$ is complete, $\Lambda_\infty^L \hat{\otimes}_L K \cong \lim_{|r| \rightarrow 1_{-}} K\langle \tfrac{T}{r} \rangle$ i.e. the $L$-analytic distribution algebra can be expressed as the limit of 1-variable Tate algebras after extending scalars. Similarly each $\Lambda_n^L \hat{\otimes}_L K$ is $K$-isomorphic to a 1-variable Tate algebra $K\langle T \rangle$.\lozengeend
		\end{rmk}
		
		\par Like the $\mathbb{Q}_p$-analytic case of Lemma \ref{cyc_char} above we can define characteristic ideals of torsion coadmissible modules. By Remark \ref{padic_fourier} since $\Lambda_\infty^L$ is a limit of dimension 1 rings, we can use Pottharst's original ideas to give a structure theorem for such modules, which lets us define characteristic ideal more explicitly.
		
		\begin{lem} \label{analytic_char} 
			\begin{itemize}
				\item Let $A=\varprojlim_n A_n$ be a Fr\'echet--Stein algebra such that each $A_n$ is noetherian and $A$ is a Pr\"ufer domain. An $A$-module $M$ is torsion and coadmissible if and only if $$M \cong \prod_{\alpha \in I} A/\mathfrak{p}_\alpha^{l_\alpha}$$ for some collection of closed points $\{\mathfrak{p}_\alpha\} \subset \bigcup_{n \geq 1} \text{Spec}$ $A_n$ such that each $\text{Spec}$ $A_n$ contains $\mathfrak{p}_\alpha $ for only finitely many $\alpha$.
				\item $\mathrm{char}_{A}(M)$ is well defined. 
			\end{itemize}
		\end{lem}
		\begin{proof}
			
			\par 
			For the second part we can run the argument of the well-definedness part of Lemma \ref{cyc_char} almost word for word since $M$ is coadmissible. Proof of the first part can be thought of a version of \citep[Prop. 1.1]{Pot12} with weakened assumptions. Importantly, we do not assume $A$ is a Bezout domain. Since $A$ is a Pr\"ufer domain, and the $A_n$ are noetherian, they must be Dedekind domains. We thus have a unique factorisation of each $\mathrm{char}_{A_n}(M_n)$ into height 1 prime ideals and a unique pseudo-isomorphism $$ M_n \cong \prod_{\alpha \in I_n} A_n/\mathfrak{p}_{n,\alpha}^{r_{n,\alpha}} $$ for a finite indexing set $I_n$, which form an inverse system by coadmissibility. Let $I = \bigcup_n I_n$ which may be infinite. Taking limits gives the result. 
			
		\end{proof}
		
		\par We will use this when $A=\Lambda_\infty^L$ in order to define $\text{char}_{\Lambda_\infty^L}$. Similar to Pottharst's version for his $\Lambda_\infty$, this is multiplicative in exact sequences. Note that we can't apply this Lemma to $A=\Lambda_\infty$ in our setting when $L \neq \mathbb{Q}_p$ because it is the limit of dimension 2 rings and therefore the $\Lambda_n$ are not Dedekind domains. We now aim to compare these definition with $\mathrm{char}_{\Lambda_\infty}$.
		
		\begin{rmk} \label{char_prufer}
			\par In Pottharst's $\mathbb{Q}_p$-analytic version he invokes a `theorem of Lazard' which shows the existence of a single element $f$ such that $\mathrm{div}(f)$ is a formal sum of exactly the (possibly infinite number of) height 1 primes dividing $\mathrm{char}_{\Lambda_n^L}(M_n)$ for all the $n$, counting multiplicity, thus $\mathrm{char}_{\Lambda_\infty^L}(M)=(f)$ is principal. This argument relies on the fact that $\mathfrak{X}^{\mathbb{Q}_p}$ is isomorphic to the unit disk over a finite extension of $\mathbb{Q}_p$ (more generality allows a spherically complete extension), which makes his $\Lambda_\infty$ a Bezout domain. As seen in \ref{padic_fourier}, $\mathfrak{X}_L$ does not satisfy this property as we need to take the base change by a transcendental extension $\widehat{L_\infty} \subset \mathbb{C}_p$ which is not discretely valued or spherically complete. This is shown in \citep[Lemma 3.9]{ST01} and we conclude that $\Lambda_\infty$ is not a Bezout domain. It is shown that it is still a Pr\"ufer domain at the end of \S3 of op. cit. This is important, because it means we will be able to invert our (finitely generated) characteristic ideals over $\Lambda_\infty^L$ as fractional ideals, something we can't do over $\Lambda_\infty$. \lozengeend
		\end{rmk}

		%different definitions of L-analyticity
		\par In the specific situation we analyse, where $L=E_p$ is a degree 2 extension of $\mathbb{Q}_p$, we have some more ways of understanding $E_p$-analyticity. Drawing parallels to complex analysis, we can think about $E_p$-analyticity as `holomorphicity' compared to $\mathbb{Q}_p$-analyticity as `real analyticity'. This motivates the following categorisation of $E_p$-analytic distributions:

		\begin{lem} \label{lie_element}
			There is an element $\nabla \in \mathrm{Lie}(\mathcal{O}_{E_p}$) such that $\{ \nabla, \overline{\nabla}\}$ generate $\mathrm{Lie}(\mathcal{O}_{E_p})$ over $\mathbb{Q}_p$, where $\bar{\cdot}$ is conjugation by $\mathrm{Gal}(E_p/\mathbb{Q}_p)$, and $\ker{\pi_{E_p}}$ is exactly the distributions $f \in \Lambda_\infty$ for which $f(\chi)=0$ for all characters $\chi$ such that $\overline{\nabla} \cdot \chi = 0$. Thus we have an exact sequence $$ 0 \rightarrow \Lambda_\infty \xrightarrow{\overline{\nabla}} \Lambda_\infty \xrightarrow{\pi_{E_p}} \Lambda_\infty^{E_p} \rightarrow 0. $$
		\end{lem}
		\begin{proof}
			The operators $\nabla$ and $\overline{\nabla}$ come from \citep[\S2]{Ber16} and Remark 2.7 of op. cit. identifies $\Lambda_\infty^{E_p}$ with the dual of the analytic $\mathbb{Q}_p$-valued functions which $\overline{\nabla}$ sends to 0. Note that $\mathrm{Lie}(\mathcal{O}_{E_p})$ embeds Galois equivariantly into $\Lambda_\infty$ by \citep[\S1]{ST01} so $\nabla, \enspace \overline{\nabla} \in \Lambda_\infty$ act via multiplication and the exact sequence in the Lemma is an exact sequence of $\Lambda_\infty$-modules.
		\end{proof}
		
		\begin{rmk} \label{cauchy_riemann}
			To see the analogy to Cauchy--Riemann equations; we can consider $\Lambda_\infty^{E_p}$ as the continuous $K$-valued distributions of the Lubin--Tate Galois group $E_p^\times$, for large enough $p$-adic field $K$. We can think of this as a power series in one `complex' variable $Z$. Similarly, $\Lambda_\infty$ is the $K$-valued $\mathbb{Q}_p$-linear distributions of $(\mathbb{Q}_p^\times)^2$ Given an arbitrary basis of $L/\mathbb{Q}_p$ we can write elements out as power series of two `real' variables $x$ and $y$. We can think of $\nabla$ as `$\tfrac{d}{dz}$' and $\overline{\nabla}$ as `$\tfrac{d}{d\bar{z}}$' for the `complex' variable $z$. 
			\par For a two-dimensional real power series to be $\mathbb{C}$-differentiable, it is not enough to be real differentiable in each real variable; we need an added compatibility condition  given by $\tfrac{d}{d\bar{z}}=0$. The condition that $\overline{\nabla}=0$ should be thought of as an analogous operator measuring the difference between $\mathbb{Q}_p^2$-analyticity and $\mathbb{Q}_{p^2}$-analyticity. This analogy was originally discussed in \citep{ST01} and has been developed since then, notably the work of Berger cited in Lemma \ref{lie_element} and in \citep{Ste23}. Since $\Lambda_\infty$ is a dimension 2 ring and the embeddings of these Lie algebra operators give a regular sequence of length 2, we find that $\left(\dfrac{\Lambda_\infty}{\langle \nabla, \overline{\nabla} \rangle}\right)^{\Gamma_L}$ is a field extension of $\mathbb{Q}_p$. This is analogous to saying that a function on $\mathbb{C}$ which is holomorphic and anti-holomorphic must be constant. \lozengeend
		\end{rmk}
		
		\begin{prop} \label{inert_descent}
			Let $M$ be a torsion coadmissible $\Lambda_\infty$-module and suppose $\overline{\nabla} \notin \mathrm{Supp}(M)$. Then $$ \mathrm{char}_{\Lambda^{E_p}_\infty}(M[\overline{\nabla}]) \pi_{E_p}(\mathrm{char}_{\Lambda_\infty}(M))=\mathrm{char}_{\Lambda^{E_p}_\infty}(M/\overline{\nabla}M). $$
		\end{prop}
		\begin{proof}
			\par If $M$ is a torsion = $\Lambda_\infty$-module, given $m \in M$ we can find $\lambda \in \Lambda_\infty$ such that $\lambda \cdot m =0$; but to have torsionness for $m \in M[\overline{\nabla}]$ or $M/\overline{\nabla}$ as $\Lambda_\infty^L$-modules we need $\overline{\nabla} \nmid \lambda$. This is equivalent to saying that $\text{Supp}(M) \cap \mathfrak{X}_L = \emptyset$, or that $\overline{\nabla} \notin \text{Supp}(M)$. Assuming this condition, if $M$ is a torsion module then $M[\overline{\nabla}]$ and $M/\overline{\nabla}$ are torsion  $\Lambda_\infty^{E_p}$-modules and we can define characteristic ideal. 
			\par Now we apply \citep[Lemma 14.15]{Kat04} for each $A=\Lambda_n$  and $a=$ image of $\overline{\nabla}$ in $\Lambda_n$. We define $N$ as the $\Lambda_\infty^{E_p}$-module $\prod_\mathfrak{q} \left( A/(\mathfrak{q} + \overline{\nabla}A) \right)^{l_\mathfrak{q}}$, where $l_\mathfrak{q} = \text{length}_{\Lambda_{\infty, \mathfrak{q}}^{E_p}} N_{\mathfrak{q}}$. then we can verify by Lemma \ref{cyc_char} that $\mathrm{char}_{\Lambda_n^{E_p}}(N) = \pi_{E_p}(\mathrm{char}_{\Lambda_n}(M)).$ Since $\mathrm{char}_{\Lambda_n}$ is multiplicative, it gives a well defined function on $G(C)$. Equality in $G(C)$ provides our result at the level of $\Lambda_n$ for each $n$. By coadmissibility we get the result over $\Lambda_\infty$.
		\end{proof}
		
		% Need compatibility of char ideals from lambda to lambda_\infty - Have this mostly written up now! FOLLOWING KATO ASTERISQUE 14.15

		\begin{prop} \label{inert_descent2}
			With notations above, suppose $\overline{\nabla} \notin \mathrm{Supp}(\dfrac{\widetilde{H}^1(E,\mathbb{V}^{E_p}_\infty;\Delta^1) }{\mathrm{im}(c^\Pi)})$. Then we have the following exact sequences of torsion $\Lambda_\infty^{E_p}$-modules:
			\begin{itemize}
				\item $ \widetilde{H}^1(E,\mathbb{T};\Delta^1)[\overline{\nabla}] = 0  $
				\item $ 0 \rightarrow \dfrac{\widetilde{H}^1(E,\mathbb{T};\Delta^1)}{\mathrm{im}(c^\Pi) \cdot\overline{\nabla}} \rightarrow  \dfrac{\widetilde{H}^1(E,\mathbb{V}^{E_p}_\infty;\Delta^1) }{\mathrm{im}(c^\Pi)} \rightarrow \widetilde{H}^2(E,\mathbb{T};\Delta^1)[\overline{\nabla}] \rightarrow 0 $
				\item $ 0 \rightarrow \widetilde{H}^2(E,\mathbb{T};\Delta^1) \otimes_{\Lambda_\infty} \Lambda^{E_p}_\infty \rightarrow \dfrac{\widetilde{H}^2(E,\mathbb{T};\Delta^1)}{\overline{\nabla}} \rightarrow 0$
			\end{itemize}
		\end{prop}
		\begin{proof}
			\par These sequences will arise from computing the exact sequences attached to the spectral sequence in Remark \ref{tor_spectral_sequence}, when $S=\Lambda_\infty$ and $R=\Lambda_\infty^L$ specifically for $L=E_p$; although this can be done for more general $L$ by iterating the spectral sequence with different generators of $\text{Lie}(\mathcal{O}_L)$ embedded into $\Lambda_\infty$. However, since $R$ is a quotient of $S$ by a single element we can compute this more explicitly as follows. It will be clear to the reader that this matches the spectral sequence calculations of the Remark.
			\par First note that $C^\bullet = \widetilde{R\Gamma}^\bullet(E,\mathbb{T}_\infty;\Delta^1)$ is a perfect complex of $\Lambda_\infty$-modules - this follows from the Selmer complex being perfect over $\Lambda$ and flat base change to $\Lambda_\infty$. Thus it is also a flat complex of $\Lambda_\infty$-modules and we can apply $\otimes^{\mathbb{L}} C^\bullet$ to the exact sequence of Lemma \ref{lie_element} to get $$ 0 \rightarrow C^\bullet \xrightarrow{\overline{\nabla} \otimes \mathrm{id}} C^\bullet \rightarrow C^\bullet \otimes^{\mathbb{L}}_{\Lambda_\infty} \Lambda_\infty^{E_p} \rightarrow 0. $$
			\par We can take cohomology of this short exact sequence of complexes to get a long exact exact sequences as in Remark \ref{tor_spectral_sequence}, which degenerate to the following:
			
			\begin{eqnarray*}
				& \qquad 0 \rightarrow \widetilde{H}^1(E,\mathbb{T};\Delta^1)[\overline{\nabla}] \rightarrow 0 \qquad &\\
				0 & \rightarrow  \widetilde{H}^1(E,\mathbb{T};\Delta^1)[\overline{\nabla}] \rightarrow H^1(C^\bullet \otimes^{\mathbb{L}}_{\Lambda_\infty} \Lambda_\infty^{E_p} )  \rightarrow \widetilde{H}^2(E,\mathbb{T};\Delta^1)[\overline{\nabla}]  \rightarrow  & 0 \\ 
				& 0  \longrightarrow  \widetilde{H}^2(E,\mathbb{T};\Delta^1) [\overline{\nabla}] \rightarrow H^2(C^\bullet \otimes^{\mathbb{L}}_{\Lambda_\infty} \Lambda_\infty^{E_p} ) \longrightarrow  0 &
			\end{eqnarray*}

			\par We have two short exact sequences of $\Lambda_\infty^{E_p}$-modules almost resembling the statement of the proposition but it remains to check which groups are $\Lambda_\infty^{E_p}$-torsion. From Corollary \ref{cyc_kmc} we can see the corresponding torsion results over $\Lambda_\infty$ before base change by the surjection $\pi_{E_p}$. To ensure the modules we work with are really $\Lambda_\infty^{E_p}$-torsion, we apply the support condition from Proposition \ref{inert_descent}. By the divisibility of Corollary \ref{cyc_kmc} we only need this for $M=\dfrac{\widetilde{H}^1(E,\mathbb{V}^{E_p}_\infty;\Delta^1) }{\mathrm{im}(c^\Pi)}.$ By the divisibility of Corollary \ref{cyc_kmc} we only need this support condition on $\left( \dfrac{e_\delta \widetilde{H}^1(E,\mathbb{T}_\infty; \Delta^1)}{\mathrm{im}(c^\Pi)} \right)$. 
            \par  Since we need a sequence of torsion modules we have to divide out by $\text{im}(c^\Pi)$ in the first two terms of the sequence of $H^1$ terms to get torsion exact sequences of $\Lambda_\infty^{E_p}$-modules.
		\end{proof}
		
		\par We can use the exact sequences in Proposition \ref{inert_descent2} and the comparison of characteristic ideals in Proposition \ref{inert_descent} for $M=\widetilde{H}^i(E,\mathbb{T};\Delta^1)$, $i=1,2$, to run a descent argument like in the split prime case of \citep[Corollary 3.8]{Man22}. First we extend the definition of the $\Delta^1$ Selmer complex, defining the base changed representation $\mathbb{V}_\infty^{E_p} = T \otimes_{\Lambda} \Lambda_\infty^{E_p}=V \otimes_{\Lambda} \Lambda_\infty^{E_p}$ (again since we are inverting $p$ in the definition of $\Lambda_\infty$ we lose integral information of $T$), we define locally at $p$: $$ \widetilde{R\Gamma}(E,\mathbb{V}_\infty^{E_p};\Delta^1) = \widetilde{R\Gamma}(E,\mathbb{V}_\infty;\Delta^1) \otimes^{\mathbb{L}}_{\Lambda_\infty} \Lambda_\infty^{E_p} .$$ This is really a Selmer complex of $\Lambda_\infty^{E_p}$-modules since Selmer complexes commute with derived base change, see \citep[Theorem 1.12 case (3)]{Pot13} where $f_0=\pi_{E_p}$, but we will omit the pushforward $\pi_{E_p,*}$ from our notation. We can now state the $E_p$-analytic variant of Kato's main conjecture in terms of this Selmer complex. 
		
		\begin{cor} \label{analytic_kmc}
			Suppose $V$ and $\delta$ satisfy Assumption \ref{assumption_V_delta}, $p$ is inert in $E$ and the support condition from Proposition \ref{inert_descent2}. Then,
			\begin{enumerate}[i)]
				\item $e_\delta \widetilde{H}^2(E,\mathbb{V}^{E_p}_\infty; \Delta^1)$ is coadmissible $\Lambda_{\infty,\delta}^{E_p}$-torsion
				\item $e_\delta \widetilde{H}^1(E,\mathbb{V}^{E_p}_\infty; \Delta^1)$ is coadmissible and torsion-free of $\Lambda_{\infty,\delta}^{E_p}$-rank 1
				\item $\mathrm{char}_{\Lambda^{E_p}_{\infty, \delta}} \left(e_\delta \widetilde{H}^1(E,\mathbb{V}_\infty^{E_p}; \Delta^1) \right) \bigm\vert \mathrm{char}_{\Lambda^{E_p}_{\infty \delta} } \left( \dfrac{e_\delta \widetilde{H}^2(E,\mathbb{V}_\infty^{E_p}; \Delta^1)}{\mathrm{im}(c^\Pi)} \right) $
			\end{enumerate}
		\end{cor}
		\begin{proof} 
			\par Since $\Lambda_\infty^L$ is a Pr\"ufer domain, every non-zero finitely generated ideal is invertible, which we will use throughout the proof. As $\Lambda_\infty$ is not a Pr\"ufer domain, we cannot invert ideals arbitrarily, we can only consider a fraction of ideals when the denominator divides the numerator. See Remark \ref{char_prufer} for a discussion on this.
			\par The first two parts follow from Corollary \ref{cyc_kmc} base changed by $\pi_{E_p}$; coadmissibility is preserved as base change corresponds to pullback of coherent sheafs by $\mathfrak{X}_L \hookrightarrow \mathfrak{X}$, and torsion/torsion-freeness is preserved by the support condition. We therefore prove $(iii)$, omitting the $e_\delta$ notation for convenience. Taking characteristic ideals of both exact sequences in Proposition \ref{inert_descent2}, we get an equality of fractional ideals
			\begin{eqnarray*} \mathcal{I}_1 = & \dfrac{\mathrm{char}_{\Lambda_\infty^{E_p}}\left(\dfrac{\widetilde{H}^1(E,\mathbb{V}_\infty^{E_p};\Delta^1)}{\mathrm{im}(c^\Pi)} \right)}{\mathrm{char}_{\Lambda_\infty^{E_p}}(\widetilde{H}^2(E,\mathbb{V}_\infty^{E_p};\Delta^1))} \\ = & \dfrac{\mathrm{char}_{\Lambda_\infty^{E_p}}\left( \dfrac{\widetilde{H}^1(E,\mathbb{T};\Delta^1)}{\mathrm{im}(c^\Pi) \cdot\overline{\nabla}}\right) \cdot \mathrm{char}_{\Lambda_\infty^{E_p}}\left(\widetilde{H}^2(E,\mathbb{T};\Delta^1)[\overline{\nabla}]\right)}{\mathrm{char}_{\Lambda_\infty^{E_p}}\left(\dfrac{\widetilde{H}^2(E,\mathbb{T};\Delta^1)}{\overline{\nabla}}\right)}. 
			\end{eqnarray*}
			Meanwhile we can apply $\pi_{E_p}$ to the fractional ideal coming from the divisibility of Corollary \ref{cyc_kmc}, then apply base change of characteristic ideals Proposition \ref{inert_descent} to get
			\begin{eqnarray*} \mathcal{I}_2 = & \pi_{E_p} \left[ \dfrac{\mathrm{char}_{\Lambda_\infty}\left(\dfrac{\widetilde{H}^1(E,\mathbb{T};\Delta^1)}{\mathrm{im}(c^\Pi)} \right)}{\mathrm{char}_{\Lambda_\infty}(\widetilde{H}^2(E,\mathbb{T};\Delta^1))} \right] \\ = & \dfrac{\mathrm{char}_{\Lambda_\infty^{E_p}}\left( \dfrac{\widetilde{H}^1(E,\mathbb{T};\Delta^1)}{\mathrm{im}(c^\Pi) \cdot\overline{\nabla}}\right) \cdot \mathrm{char}_{\Lambda_\infty^{E_p}}\left(\widetilde{H}^2(E,\mathbb{T};\Delta^1)[\overline{\nabla}]\right)}{\mathrm{char}_{\Lambda_\infty^{E_p}}\left(\dfrac{\widetilde{H}^2(E,\mathbb{T};\Delta^1)}{\overline{\nabla}}\right)}. 
			\end{eqnarray*}
			\par We see that $$ \mathcal{I}_1 = \mathcal{I}_2.$$ Since $\mathcal{I}_2$ is a proper ideal of $\Lambda_\infty^{E_p}$ by Corollary \ref{cyc_kmc}, $\mathcal{I}_1$ must be a proper ideal too, hence the divisibility.
			
		\end{proof}
		
		\begin{rmk} \label{general_analytic_descent}
			This descent theory of locally analytic distribution algebras will generalise to any $[L:\mathbb{Q}_p]=d$. Consider $\mathbb{Q}_p \subset L^{\prime} \subset L$ - we will have that $ \Lambda_\infty^{L^\prime}$ is a Frecher--Stein algebra of dimension $[L:L^\prime]$, and the quotient of $\Lambda_\infty \rightarrow \Lambda_\infty^{L^{\prime}}$ is cut out by the vanishing of $\frac{d}{[L:L^{\prime}]}$ Lie algebra operators by \citep[\S2]{Ber16}. We can carry out the descent steps of Propositions \ref{inert_descent} and \ref{inert_descent2} $\frac{d}{[L:L^{\prime}]}$ times, assuming the right support conditions. However, unless $L=L^{\prime}$ we will not be able to use the properties of a Pr\"ufer domain in order to define $\mathrm{char}_{\Lambda_\infty^{L^\prime}}$ explicitly - we would instead work with lengths of modules over localisations in the style of Kato. It strikes the author that there is not much more difficulty in running these arguments for general $L$, but there are no arithmetic examples that come to mind with $d>2$. \lozengeend 
		\end{rmk}
		
		\section{Lubin--Tate $(\phi, \Gamma)$-modules}
		
			\par Whilst the literature on cyclotomic theory of $(\phi,\Gamma)$-modules is well established, the topic of Lubin--Tate analogues is relatively recent and still in development. In this subsection we will outline a few of the relevant theories, their properties and how they relate to each other. An important difference in the Lubin--Tate setting is the introduction of $L$-analyticity and the fact that \'{e}tale $L$-analytic $(\phi_L, \Gamma_L$)-modules over the Robba ring no longer see all of Iwasawa cohomology. This will become apparent in \ref{steingart_conj} and the surrounding discussions.
		
		\par There are a few different theories of Lubin--Tate $(\phi,\Gamma)$-modules, using different base rings. We are interested in the theories over the period ring $\mathbb{B}_L$ and the Robba ring $\mathcal{R}_L$; the former compares well with Iwasawa cohomology and the latter allows us to think about $L$-analyticity. Loosely, a $(\phi,\Gamma)$-module over an appropriate ring $R$ is an $R$-module $M$ with a `Galois' action of the $p$-adic Lie group group $\Gamma$ and a `Frobenius' action by a semilinear operator $\phi \in \text{End}(M)$ such that these actions commute. For a reasonable theory of such modules, we need to choose a ring $R$ which has such an action (as a trivial module over itself). We define the rings involved below, largely following \citep[\S2]{SV23b} and \citep[\S2]{BSX20}. As with any Lubin--Tate theory, we must first fix a uniformiser $\varpi_L$ of $L$; when $L=E_p$ we will take $\varpi_{E_p}=p$. We will consider the Lubin--Tate extension $L_\infty/L$ attached to $\varpi_L$.
		\par First we define $\mathbb{B}_L= \text{Frac}(\mathbb{A}_L)$, where $\mathbb{A}_L$ is defined by the following. We let $\mathscr{A}_L $ be the $\varpi_L$-adic completion of $\mathcal{O}_L[[Z]][Z^{-1}]$ which embeds Galois equivariantly into the Witt vector ring $W(\mathbb{C}_p^\flat)$ via the map given in \citep[Lemma 4.1]{SV15}. We take $\mathbb{A}_L$ to be the image of this embedding. This inherits the Galois action and Frobenius action from the Witt vectors $W(L_\infty^\flat)$, thus we can think of $(\phi,\Gamma)$-modules over this ring.
		
		\begin{defn} \citep[Theorem 1.6]{KR09} \label{LT_kisin-ren}
			The functor $$W \mapsto D_{\text{LT}}(W) := (\mathbb{B}_L \otimes_{\mathcal{O}_L} W)^{\ker \chi_{\text{LT}}} $$ is an functor $\mathrm{Rep}_{L}(G_L) \rightarrow \mathfrak{M}(\mathbb{B}_L)$ the category of $(\phi_L,\Gamma_L)$-modules over $\mathbb{B}_L$. We denote its categorical image by $\mathfrak{M}^{et}(\mathbb{B}_L)$ the \'etale $(\phi_L,\Gamma_L)$-modules over $\mathbb{B}_L$.
		\end{defn}
		
		\par By \citep[Remark 4.6]{SV15}  we see that this functor commutes with twisting by characters $\eta$ of $\Gamma_L=\text{Gal}(L_\infty/L)$, i.e. Hecke characters of $p$-power conductor. The authors of op. cit. use this result to state Iwasawa cohomology in terms of the invariants of a $(\phi,\Gamma)$-module, which again mirrors classical theory, stated in Theorem \ref{LT_iwasawa}.
		
		\par Recalling our $L$-analytic weight space $\mathfrak{X}_L$, we can define the Robba ring as follows. Let $B=B[0,1)$ be the $L$-affinoid unit disk, and for $r \in (0,1) \cap p^\mathbb{Q}$ let $B(r)$ be the $L$-affinoid disk in $L$ of radius $r$. For an extension $K/L$ we define the Robba ring following \citep[\S2.1]{BSX20} as the union of rings of regular functions $$ \mathcal{R}_K = \bigcup_{r} \mathcal{O}_K(B \setminus B(r)), $$ which is the limit of sections of annuli and can be interpreted as a ring of Laurent series. By this definition we also see that $\Lambda_\infty^L \hat{\otimes}_L K \subset \mathcal{R}_K$ for $K/\widehat{L_\infty}$ complete, as functions converging on the twisted unit disk base extended to $K$ will converge on the unit disk and therefore on each of the annuli. The Robba ring also has Galois and Frobenius actions as described in \S2.2 of op. cit. There is a unique linear continuous $\Gamma$-equivariant endomorphism $\psi$ of $\mathcal{R}_K$ such that $\psi \circ \phi = 1$, which is fixed from now. Later we will want to consider a scaled version of the $\psi$-operator, which we will denote by $\Psi$. The distinction between their use is subtle and one needs to be careful.
		
		\begin{defn}
			\citep[Theorem 10.4]{Ber16} \label{analytic_kisin-ren} 
			Let $W \in \mathrm{Rep}_{\mathcal{O}_L}(\Gamma_L)$ be $L$-analytic. The functor $$W \mapsto D^{\dagger}_{\text{rig}}(W) := (W \otimes_L \mathcal{R}_K)^{\ker(\chi_{\text{LT}})} $$ is an equivalence between the category of $L$-analytic Galois representations $\mathrm{Rep}^{L-\text{an}}_{\mathcal{O}_L}(G_L)$ and $\mathfrak{M}_{L-\text{an}}^{et}(\mathcal{R}_K)$ the category of $L$-analytic \'{e}tale $(\phi_L,\Gamma_L)$-modules over $\mathcal{R}_K$.
		\end{defn}
		
		\par Note that by \citep[Theorem 3.3.11]{BSX20} this is equivalent to the construction of $L$-analytic $(\phi,\Gamma)$-modules of Berger--Schneider--Xie where the Robba ring is defined as a limit of sections of annuli twisted by the $p$-adic Fourier transform, i.e. $$ \ \mathcal{R}^*_K := \bigcup_{r} \mathcal{O}_K(\mathfrak{X}_L \setminus \mathfrak{X}_L(r)) $$ where $\mathfrak{X}_L(r)$ is the image of the rigid closed disc of radius $r$ inside $\mathfrak{X}^L$. Due to this equivalence we will only use the former definition. These functors only see the category of $L$-analytic $(\phi_L,\Gamma_L)$-modules, i.e. where the derived $\text{Lie}(\Gamma)$-action is $L$-bilinear. We hope to recover the $L$-analytic base change of Iwasawa cohomology as the $\psi_L=1$ invariants of such a $(\phi_L, \Gamma_L)$-module $M$ over the Robba ring, but this doesn't work for $L \neq \mathbb{Q}_p$. We will drop the $L$ from subscript as we only work in the Lubin--Tate setting. The complex attached to $M \xrightarrow{\psi-1} M$ still remains of arithmetic interest, we will later interpret its cohomology as a quotient of Iwasawa cohomology.

		\subsection{Cohomology of $L$-analytic $(\phi,\Gamma)$-modules}
		
		\par In the established cyclotomic and $\mathbb{Q}_p$-analytic theory, many people have studied Iwasawa cohomology over the cyclotomic extension via the theory of $(\phi,\Gamma)$-modules over $\mathbb{B}_{\mathbb{Q}_p}$. This is generalised in two steps for our purpose; first extending the theory to $G_L$-representations for general finite $L/\mathbb{Q}_p$ and then to take into account $L$-analyticity. We will state things at first for general $L$ with choice of uniformiser $\varpi_L$ but in applications we will have $L=E_p=\mathbb{Q}_{p^2}$ and $\varpi_L=p$. In this section we first survey cohomology theories of these $(\phi, \Gamma)$-modules and in the \'etale case we investigate their comparisons to Galois cohomology and Iwasawa cohomology. To see the the relationship between the different theories of Lubin--Tate $(\phi,\Gamma)$-modules and conditions on Galois representations we refer the reader to the web of morphisms in \citep[\S6]{SV23b}. After defining the necessary objects We develop comparison maps between Herr complex cohomology groups and group cohomology which are isomorphisms only when $L=\mathbb{Q}_p$, and whose non-zero kernels will be important in defining the local conditions we need to state a main conjecture.
		
		\par Schneider and Venjakob give a Lubin--Tate version of the map embedding Galois representations into the category of $(\phi,\Gamma)$-modules over $\mathbb{B}_L$, where $\Gamma = \Gamma_L \cong \mathcal{O}_L^\times$ is the Lubin--Tate Galois group of $L$. Note that given a general module $M$ over a ring $R$, and $f \in \text{Aut}_R(M)$, we will denote the complex $$ M[f] = [0 \rightarrow M \xrightarrow{f} M \rightarrow 0]. $$

		\begin{thm}\citep[Theorem 5.2.51]{KV22} \label{LT_iwasawa} 
			Let $S \in \mathrm{Rep}_{\mathcal{O}_L}(G_L)$, $\tau= \chi_{\text{LT}}^{-1}\chi_{\text{cyc}}$ and $\psi = \psi_{D_{\text{LT}}(S(\tau^{-1}))}$ the unique endomorphism arising from $\mathcal{R}_L$ such that $\psi \circ \phi=\tfrac{p^{[L:\mathbb{Q}_p]}}{\varpi_L}$ (in our case of $L=E_p$ and uniformiser $p$ this constant is $p$). Then
			$$R\Gamma_{Iw}(L_\infty/L,S) \cong_{\text{q-iso}} D_{\text{LT}}(S(\tau^{-1}))[\psi-1]$$
		\end{thm}
		
		\par This allows us to compare our Lubin--Tate Iwasawa cohomology over $\Lambda$ with constructions from $p$-adic Hodge theory that have proved useful in the $\mathbb{Q}_p$-analytic case. The twist by $\tau$ appears here again, which is not a coincidence. However, we still need an analogous comparison using $L$-analytic families of $(\phi_L,\Gamma_L)$-modules (i.e. modules over the Robba ring $\mathcal{R}_L$) using Definition \ref{analytic_kisin-ren}. One needs to be careful here, as it is stated in terms of an endomorphism $\Psi=\tfrac{\varpi_L}{p^{[L:\mathbb{Q}_p]}}\psi$ such that $\Psi \circ \phi = 1$. This scaled $Psi$-operator is accounted for as a twist by a ramified non-etale $L$-analytic character in \citep[Lemma 3.10]{Ste22b} (hence this $\Psi$ operator which may appear more natural on $(\Phi,\Gamma)$-modules does not come from the underlying Galois representations.
		
		%State + Explain this in more detail
		%Need rigid analytic Kisin-Ren equivalence defined properly (+ notation fixed)
		
		\begin{defn} \label{comp_defn} For any co-$L$-analytic Galois representation $S$, we have a natural comparison map defined in \citep[\S5]{Ste22b}
			$$ \text{comp}: R\Gamma_{Iw}(L_\infty, S) \otimes \Lambda_\infty^{L} \rightarrow  D^{\dagger}_{\text{rig}}(S(\tau^{-1}))[\Psi-1] $$
		\end{defn}
		
		In the classical case $L=\mathbb{Q}_p$, this is known to be a quasi-isomorphism by work of Colmez \citep[V.1.18]{Col10} which is an essential input in the work of Pottharst. We can compute the Euler characteristics of complexes on each side and we will show this can no longer be a quasi-isomorphism as soon as $L \neq \mathbb{Q}_p$. A conjecture of Steingart claims this natural map to be a surjection after taking $H^1$; this tells all classes of the $\Psi-1$ complex necessarily come from Galois cohomology classes. There is a discussion of the difficulty of proving this in generality in \citep[\S5]{Ste22b}.
		
		\begin{conj} \textup{\citep[Conjecture 5.3]{Ste22b}} \label{steingart_conj}
			Suppose $S^*(1)$ is $L$-analytic. Then the map $\mathrm{comp}$ is surjective after taking $H^1$, i.e.
			$$ H^1_{\text{Iw}}(L_\infty/L,S) \otimes_\Lambda \Lambda_\infty^L \twoheadrightarrow D^{\dagger}_{\text{rig}}(S(\tau^{-1}))^{\Psi=1}. $$ 
		\end{conj}
		
		This conjecture is more accessible for rank 1 $(\phi,\Gamma)$-modules, which would be sufficient for our applications, or more generally trianguline $S$ since the author of op. cit. has developed further results in this setting. We define the trianguline condition here, emulating the Definition \ref{def_ord} but in the larger category of $(\phi,\Gamma)$-modules. 
		
		\begin{defn} \label{triang}
			An overconvergent representation $S$ is called trianguline if $D^{\dagger}_{\text{rig}}(S(\tau^{-1}))$ admits a full flag of $(\phi,\Gamma)$-submodules over $\mathcal{R}_K$ with graded pieces given by $\mathcal{R}_K(\delta)$ for characters $\delta:L^\times \rightarrow K^\times$.
		\end{defn}

        \par Upcoming work joint with Rustam Steingart \citep{MR25} will study surjectivity of the comparison map when $S^*(1)$ is a trianguline $L$-analytic representation, and establish some consequences. For this current work we will assume this as a Conjecture \ref{steingart_conj}. We remark that this conjecture is not necessary in the following, but without the assumption we have to carry around the cokernel term in exact sequence computations, leading to more error terms (which we expect to vanish anyway). The conjecture also allows us to compute the ranks of the Selmer groups we define, and in particular shows that the Euler characteristics of the underlying Selmer complexes match what we expect.
		
		\par For trianguline $S$ we can compute how large the kernel $\mathfrak{K} = \ker(\text{comp})$ is using the conjecture. The following result is from discussions with David Loeffler and Rustam Steingart, also mentioned in \citep[\S5]{Ste22b}. In order to use results of Steingart computing the rank of $D^{\dagger}_{\text{rig}}(S(\tau^{-1}))[\Psi-1] $ when $S$ is a trianguline co-$L$-analytic representation, we will need to restrict from the full Lubin--Tate Galois group $\Gamma_L$ to a torsion-free subgroup $U\cong \mathbb{Z}_p^2.$ We recall that for a $\Lambda_\infty^L$-module $M$, we have a decomposition $$ M = \oplus_\delta e_\delta M =: \oplus_\delta M_\delta, $$ as before for $\delta$ ranging over the (finitely many) characters of $\Delta$.
		
		\begin{lem} \label{steingart_kernel}
			For $L$ a finite extension of $\mathbb{Q}_p$, $S$ a trianguline co-$L$-analytic representation, assume $H^0(L,S)=H^2(L,S)=0$.Then $\mathfrak{K} = \ker(\text{comp})$ is a $\Lambda_\infty^L$-module of rank $([L:\mathbb{Q}_p]-1)\text{dim}(S)$.
		\end{lem}
		
		\begin{proof}
			\par By \citep[Prop 4.7]{Ste22b}, for each character $\delta$ of $\Gamma_{\text{tors}}$, taking $U=e_\delta \Gamma$ we have $$\text{rank}_{e_\delta \Lambda_\infty^L} e_\delta D^{\dagger}_{\text{rig}}(S(\tau^{-1}))^{\Psi=1} \enspace = \text{dim} \enspace S.$$ We sum these together to find $$\text{rank}_{\Lambda_\infty^L} D^{\dagger}_{\text{rig}}(S(\tau^{-1}))^{\Psi=1} \enspace = \text{dim} \enspace S.$$ It remains to compute the ranks of the Iwasawa cohomology groups.
			\par By \citep[Lemma 5.8]{SV15} we have a quasi-isomorphism of complexes of $\Lambda$-modules:
			$$ R\Gamma_{\text{Iw}}(L_\infty/L,S) \cong_{\text{q-iso}} R\Gamma(L,S \otimes_{\mathcal{O}_L} \Lambda). $$ By vanishing of local $H^0$ and $H^2$ Galois cohomology of Assumption \ref{assumption_V_eta}, this complex is concentrated in degree 1 and its $H^1$ is free. Since $\mathcal{O}_L \rightarrow \Lambda$ is (faithfully) flat, we have 
			$$ \mathrm{rank}_\Lambda (H^1_{\text{Iw}}(L_\infty/L,S)) = \dim_L(H^1(L,S)) = [L:\mathbb{Q}_p] \dim{S} $$
            a free $\Lambda$-module. After base changing by $\Lambda_\infty^L$ we have the same dimension by freeness and thus get the result.
		\end{proof}
		
		\par We can think of this argument in the following way; in the cyclotomic setting ($L=\mathbb{Q}_p$) there is a short exact sequence $$ 0 \rightarrow D_{\text{rig}}^\dagger(S(\tau^{-1}))^{\phi=1} \rightarrow D^\dagger_{\text{rig}}(S(\tau^{-1}))^{\psi-1} \xrightarrow{1-\phi} \mathcal{C}(S(\tau^{-1})) \rightarrow 0 $$ 
		where $\mathcal{C}(S(\tau^{-1}) = (1-\phi)D_{\text{rig}}(S(\tau^{-1}))^{\psi-1}$ is called the heart (coeur in Colmez's work) in existing literature (note that in this setting $\psi=\Psi$ too). 
	   \par Colmez uses freeness of the heart of the rank to compare the $(\phi,\Gamma)$-module over the cyclotomic Robba ring with Iwasawa cohomology, see for example \citep[Theorem I.5.2, Theorem V.1.18]{Col10}. In the $L$-analytic setting when $L \neq \mathbb{Q}_p$ we do not have a similar argument deducing freeness of the heart $ \mathcal{C}_{\text{rig}}$. This is why Colmez's arguments have not been generalised to compare the $\Psi-1$ invariants of the Robba ring $(\phi, \Gamma)$-module and Iwasawa cohomology for a co-$L$-analytic representation (and why we must assume the conjecture). We will denote this kernel by $\mathfrak{K}=\text{ker}(\text{comp})$, and any time it is mentioned we assume Conjecture \ref{steingart_conj} so we can use dimension counting arguments such as the above Lemma.

		\begin{rmk} \label{psi_euler_char}
			It will be useful for us to consider the Euler characteristic of the complex $C^\bullet = M[\Psi-1]$ for an $L$-analytic $(\phi, \Gamma)$-module over $\mathcal{R}_L$; when $L \neq \mathbb{Q}_p$ we can consider this as a complex of modules over both the $L$-analytic and $\mathbb{Q}_p$-analytic distribution algebras. We know from \citep[Proposition 4.7]{Ste22b} that $$\chi_{\Lambda_\infty^L}(C^\bullet) = \text{rank}_{\Lambda_\infty^L}(M^{\Psi=1}) - \text{rank}_{\Lambda_\infty^L}(M/(\Psi-1)M)  = 1. $$
			And since $M$ is $L$-analytic, it is torsion under the action of $\overline{\nabla}$ from Remark \ref{cauchy_riemann} and therefore 
			$$\chi_{\Lambda_\infty}(C^\bullet) = \text{rank}_{\Lambda_\infty}(M^{\Psi=1}) - \text{rank}_{\Lambda_\infty}(M/(\Psi-1)M) = 0. $$
			The difference between the Euler characteristic of $C^\bullet$ over these two different rings will be an important observation later. \lozengeend
		\end{rmk}
		
		\par For the purpose of this work, $(\phi, \Gamma)$-modules and Iwasawa cohomology are tools used to study our Galois representations, as such we need to find comparisons between the cohomology of certain complexes of a $(\phi,\Gamma)$ module $D$ and the Galois cohomology of a $p$-adic Galois representation $W$ when $D = D_{\text{rig}}^\dagger(W)$. In the classical setting of $G_{\mathbb{Q}_p}$-representations this is straightforward and can be seen in \citep{Her98}. When dealing with $L$-analyticity we have already seen the disconnect between Iwasawa cohomology and the $\Psi=1$ complex; we will now `descend' to the level of Galois cohomology and Herr complexes and see how the comparison holds up. In particular we will distinguish between continuous and analytic Herr complexes; whilst in the $G_{\mathbb{Q}_p}$-case they coincide and are both quasi-isomorphic to the complex of continuous cochains computing Galois cohomology, in the Lubin--Tate setting neither one computes Galois cohomology, demonstrated in \citep[Remark 4.7]{Ste23} where the author shows a discrepancy in their Euler characteristics. In degree 1 the respective continuous and analytic $\phi$-complexes compute overconvergent and analytic extension classes inside Galois cohomology, discussed in more detail in Remark \ref{analytic_ext}. This gives two ways of `descending' the comparison map above.
		
		\par First we deal with the continuous Herr complex, which we denote $C^{\text{cts}}_{f, \Gamma}(D)$ for $D$ a $(\phi, \Gamma)$-module over $\mathcal{R}_L$, $f$ a continuous $L$-linear operator on $D$ and $\Gamma=\Gamma_L$, following the constructions in \citep[\S5.2.3]{SV23} where it is denoted by $\mathcal{T}_{f, G}(D)$, the mapping cone of $f$ on the complex of continuous cochains, where $G=\Gamma$.
		We will denote its cohomology by $H^i_{f, \Gamma}(D)=h^i(C^{\text{cts}}_{f, \Gamma}(D)). $
		
		\par We will only use this complex for $f=\phi$, $\psi$ or $\Psi$, and therefore it is useful to present definitions of some quasi-isomorphic complexes which are more useful in computation. By \citep[(171)]{SV23}, in the $f=\phi$ case one can compute $$C^{\text{cts}}_{\phi, \Gamma}(D) \cong_{\text{q-iso}} K_{\phi,\Gamma}(D) $$ for a certain Koszul complex $K_{f,D}$ which we therefore define below. Let $\Gamma$ be topologically generated by $\gamma_1, \dots \gamma_d$ for $d=[L:\mathbb{Q}_p]$. The Koszul complex is defined by $$ K_{f,\Gamma}(D)= \text{Cone}_{f-1}\left(\bigotimes^d_{i=1} [D \xrightarrow{\gamma_i-1} D] \right)[1]. $$ Explicitly, when $d=2$ this complex can be written as \[ 0 \rightarrow D \otimes D \xrightarrow{A} (D \otimes D)^{\oplus 3} \xrightarrow{B} (D \otimes D)^{\oplus 3} \xrightarrow{C} D \otimes D \] where $$A = \begin{pmatrix}
			1- \gamma_1\\
			1-\gamma_2\\
			f-1
		\end{pmatrix}, \qquad B = \begin{pmatrix}
			1- \gamma_2 & \gamma_1-1 & 0 \\
			\Psi-1 & 0 & \gamma_2-1 \\
			0 & f-1 & \gamma_1-1
		\end{pmatrix},  \qquad C = \begin{pmatrix} f-1 & \gamma_2-1 & 1-\gamma_1  \end{pmatrix}. $$
		
		We have a corresponding comparison of the $\Psi$ complex from \citep[(173)]{SV23b} as 
		$$ C^{\text{cts}}_{\Psi, \Gamma}(D)) \cong_{\text{q-iso}} K_{\Psi, \Gamma}(D^*(\chi_{\text{LT}}))[d+1] $$ where $d=[L:\mathbb{Q}_p]$. In our case this is a degree shift by 3 which ensures that both the continuous $\phi$ and $\Psi$ complexes concentrated in degrees [0,3].

		\par This definition doesn't capture $L$-analytic information and we therefore need to define an analytic variant, which comes from \citep[\S3.1]{Ste22a}. After extending scalars to $K/\widehat{L_\infty}$ we can define the element $Z \in \Lambda_\infty^{L} \hat{\otimes}_L K$ as the preimage of $T$ in the isomorphism $\Lambda_\infty^L \hat{\otimes}_L K \cong \mathcal{R}_K^{[0,1)}$ presented in op. cit. This $Z$ is in effect a replacement for the topological generator of $\Gamma_{\mathbb{Q}_p}$ used in the cyclotomic case - we can think of this as choosing one `analytic generator' of the $\mathbb{Z}_p^2$ extension rather than two `continuous' topological generators. However this necessarily uses the $p$-adic Fourier transform and therefore needs such a transcendental base extension.  If we assume $e(L/\mathbb{Q}_p)<p-1$ (which is true in our application as we have assumed $p \geq 5$) and we recall the notation of the torsion subgroup $\Delta$ of $\Gamma_L$. For $D$ a $(\phi, \Gamma)$-module over $\mathcal{R}_K$, $f$ a continuous $L$-linear operator on $D$, we can define
		$$ C^{\text{an}}_{f, \Gamma}(D)= [0 \rightarrow D^\Delta \xrightarrow{(f-1, Z)} D^\Delta \oplus D^\Delta  \xrightarrow{Z \oplus (1-f)} D^\Delta \rightarrow 0] $$ where again $f$ is usually either $\phi$, $\psi$ or $\Psi$. Since $Z$ is defined using the $p$-adic Fourier transform, the analytic Herr complex is necessarily a complex of $\widehat{L_\infty}$-vector spaces for $\widehat{L_\infty}$   - and when discussing analytic cohomology we need to extend scalars to an extension of this. To ensure this is satsfied, from now we will work over $\mathcal{R}_K$ specifically for $K/\widehat{L_\infty}$, rather than $\mathcal{R}_L$.
		\par We use $\mathcal{H}$ to denote analytic cohomology groups to avoid conflation with the continuous cohomology groups. If $U_\delta$ is the $\mathbb{Z}_p^2$ extension cut out by the character $\delta$ of $\Delta$, we write $$\mathcal{H}^1_{f,\Gamma}(D) = \bigoplus_{\delta} \mathcal{H}^1_{f, U_\delta}(D) = \bigoplus_\delta h^1(C^{\text{an}}_{f, U_\delta}(D)) .$$ By applying the $e_\delta$ operator we therefore return to one of these components. Note that these cohomology groups are $K$-vector spaces, and therefore the $\Lambda_{\infty,K}^L$ action factors through a character $\Lambda_{\infty,K}^L \rightarrow K$ for $K/\widehat{L_\infty}$ large enough. In particular, such a character gives us a fixed choice of $\delta$ - and so in applications below we will be implicitly applying the idempotent $e_\delta$ determined by such a choice. We will therefore omit it from notation.

		\begin{rmk} \label{analytic_ext}
			From \citep[5.2.10]{SV23} we can understand the degree 1 cohomology groups of these two complexes. Suppose $D=D_{\text{rig}}^\dagger(W)$ for an $L$-analytic Galois representation $W$. There are isomorphisms \[ \mathcal{H}^1_{\phi, \Gamma}(D) \xrightarrow{\sim} \mathrm{Ext}_{L-an}^1(L,W). \]
			\[  H^1_{\phi, \Gamma}(D) \xrightarrow{\sim} \mathrm{Ext}_{\dagger}^1(L,W) =: H^1_{\dagger}(L,W)  \]
			where $\mathrm{Ext}_{L-\text{an}}(L,W) \subset \mathrm{Ext}_{\dagger}(L,W) \subset H^1(L,W)$ denote the $L$-analytic extension classes inside overconvergent classes inside Galois cohomology.
			\par In the $\mathbb{Q}_p$-analytic setting these all coincide but they are generally strict inclusions when $L \neq \mathbb{Q}_p$. In \citep[Theorem 0.6]{FX12} the authors demonstrate examples of non-overconvergent $G_L$-representations, and for dimension reasons there exist overconvergent non-analytic ones too. This is a necessary part of the theory. If either of the complexes computed Galois cohomology, much of the upcoming work would be redundant. However, Theorem 0.2 of op. cit. demonstrates that $\mathrm{Ext}_{L-\text{an}}(L,W) = \mathrm{Ext}_{\dagger}(L,W)$ if $H^0(L,W)=0$. This condition for example follows from the our existing cohomology vanishing conditions in Assumption \ref{assumption_V_eta} when $V$ satisfies Assumption \ref{assumption_L-an} and $W=S(\tau^{-1}\eta^{-1})$. \lozengeend
		\end{rmk}
		
		\par Note that overconvergence is preserved by local Tate duality but $L$-analyticity for $L \neq \mathbb{Q}_p$ is not. We denote for an overconvergent co-$L$-analytic representation $S$ the quotients of the above extension groups using local Tate duality:
		\[  H^1(L,S) \twoheadrightarrow H^1_{\backslash\dagger}(L,S) := H^1_{\dagger}(L,S^*(1))^\vee  ,\]
		\[ H^1(L,S) \twoheadrightarrow \mathrm{Ext}_{L-\text{an}}(L,S^*(1))^\vee .\]
	We also denote the orthogonal complement of overconvergent classes with respect to Tate duality as $$ H^1_{\perp}(L,S)=H^1_\dagger(L,S^*(1))^\perp.$$ 
		
		\begin{rmk} \label{herr_phi_computations}
			\par The exact relationship between cohomology of the continuous and analytic $\phi$-complexes is known, given in \citep[Corollary 4.4]{Ste23}. In particular it shows us that for $\mathbb{Q}_{p^2}$-analytic $D$, if $D$ satisfies $\mathcal{H}^0_{\phi,\Gamma}(D)=\mathcal{H}^2_{\phi,\Gamma}(D)=0$, we have that $$H^0_{\phi,\Gamma}(D)=H^3_{\phi,\Gamma}(D)=0, \enspace H^1_{\phi,\Gamma}(D)=H^2_{\phi,\Gamma}(D)=\mathcal{H}^1_{\phi, \Gamma}(D). $$ 
			Using this, we can compute exactly the dimensions of all Herr complex cohomology groups when $D=\mathcal{R}_K(\delta)$ for a generic character. By `generic' we mean that $\delta$ is not a character of the form $x \mapsto \chi_{\text{LT}}^{-i}$ or $x \mapsto \chi_{\text{LT}}^{i}\chi_{\text{an}}$ for $i \in \mathbb{N}$; these are exactly the conditions of \citep[Theorem 5.19]{FX12} which put us in case $(c)$, and we therefore have that $$\dim _L \mathcal{H}_{\phi, \Gamma_L}^j(D)= \begin{cases}0, & j \neq 1 \\ 1, & j=1  \end{cases} \qquad . $$
			It therefore follows that 
			$$\dim _L H_{\phi, \Gamma_L}^j(D)= \begin{cases}0, & j \neq 1,2 \\ 1, & j=1,2\end{cases} \qquad . $$ \lozengeend
		\end{rmk}
		
		\begin{ass} \label{assumption_generic}
			The character $\eta$ is chosen so that the character associated to $S(\tau^{-1}\eta^{-1}))$ is not of the form $x \mapsto \chi_{\text{LT}}^{-i}$ or $x \mapsto \chi_{\text{LT}}^{i}\chi_{\text{an}}$ for $i \in \mathbb{N}$. We say a character is generic if it satisfies this. 
		\end{ass}
		\par Since $\mathcal{H}_{\phi, \Gamma_L}^0(D^{\dagger}_\text{rig}(S(\tau^{-1}))=H^0(L,S)$, the first condition that the character is not of the form $\chi_{\text{LT}}^{-i}$ follows from vanishing of local Galois cohomology (in Assumption \ref{assumption_V_delta}). However the second condition that it is not of the form $\chi_{\text{LT}}^i \chi_{\text{an}}$ is a new assumption we have to place, coming from analytic Herr complex cohomology. We will later state this duality and later see it does not preserve etale-ness.
		
		\par We have results above for the continuous and analytic $\phi$-complexes, and to translate these to results on respective $\Psi$-complexes we need duality pairings of these complexes. For the continuous setting, this is established by work of Schneider and Venjakob using $\chi_{\text{LT}}$ as a dualising character; this is analogous to local Tate duality, as the next Lemma will show, and accounts for the twist by $\tau=\chi_{\text{LT}}^{-1}\chi_{\text{cyc}}$ that keeps appearing.
		
		\begin{lem} \label{cts_herr_duality}
			There is a commutative diagram of perfect pairings, for co-$L$-analytic $G_{E_p}$-representation $S$:
			
			\[ \begin{tikzcd}
				H^1(E_p, S) & \times & H^1(E_p, S^*(1)) &  K \\
				H^1_{\Psi,\Gamma}(D_{\text{rig}}^{\dagger}(S(\tau^{-1}))) & \times & H^1_{\phi,\Gamma}(D_{\text{rig}}^{\dagger}(S^*(1))) & K \\
				H^1_{\backslash\dagger}(E_p,S) & \times & H^1_{\dagger}(E_p,S^*(1)) & K \\
				\arrow[from=1-1, to=2-1, "\text{pr}"]
				\arrow[from=2-3, to=1-3, hook, "i"]
				\arrow[from=1-3, to=1-4]
				\arrow[from=2-3, to=2-4]
				\arrow[from=3-3, to=3-4]
				\arrow["\sim", from=2-1, to=3-1]
				\arrow["\sim", from=2-3, to=3-3]
				\arrow[ equal, from=1-4, to=2-4]
				\arrow[ equal, from=2-4, to=3-4]
			\end{tikzcd} \]
			
		\end{lem}
		\begin{proof}
			This follows from the diagram of \citep[Proposition 5.2.19]{SV23}, recalling the quasi-isomorphism between $\phi$-Koszul complexes and the continuous $\phi$-Herr complexes from (171) of op.cit. as well as the dualised $\Psi$-complex quasi-isomorphism.
		\end{proof}
		
		\begin{rmk} \label{an_LT_duality}
			The above duality, unlike local Tate duality, comes from a pairing of Koszul complex cohomology groups of degree 1 and $d=[L:\mathbb{Q}_p]$. This is then `corrected' to a pairing of two degree 1 groups by the degree shift of $d+1$ on Koszul complexes, after which we can compare with local Tate duality. The fact that the continuous Herr complexes for $\phi$ and $\Psi$ have cohomological degree $d+1$ is an important departure from the classical $\mathbb{Q}_p$-analytic setting; now we have a canonical embedding $\mathcal{H}^1_{\phi, \Gamma}(D) \hookrightarrow H^j_{\phi, \Gamma}(D) $, seen in \citep[Corollary 4.4]{Ste23} for both $j=1$ and $j=d$, which are distinct. A useful duality theory on continuous Herr cohomology groups will make use of the $j=d$ embedding, and therefore may not correspond to a duality on analytic Herr cohomology groups. \lozengeend
		\end{rmk}

		\par Duality for the analytic Herr complex is not as simple as Lemma \ref{cts_herr_duality}. We have by Remark \ref{herr_phi_computations} that $\mathcal{H}^2_{\phi,\Gamma}(\mathcal{R}_k(\chi_{\text{LT}}))=0$ and thus the analogous $\chi_{\text{LT}}$-pairing would be zero valued. There is a working alternative duality theory which uses a different dualising character $\chi_{\text{an}}$ defined by $x \mapsto x|x|$, introduced by Colmez and developed in the Lubin--Tate setting in \citep[\S3.4]{MSVW24}. Importantly, they compute that $\mathcal{H}^2_{\phi,\Gamma}(\mathcal{R}_k(\chi_{\text{LT}}))=K$. We will present their duality pairings below, which is Theorem 4.13 of op. cit.;
		
		\begin{lem} \label{an_herr_duality}
			Let $D$ be an $L$-analytic $(\phi, \Gamma)$-module over $R_L$, then there exist perfect pairings $$ \mathcal{H}^i_{\phi, \Gamma}(D) \times \mathcal{H}^{2-i}_{\Psi, \Gamma}(D^*(\chi_{\text{an}})) \rightarrow K $$
		\end{lem}
		
		\begin{rmk} 
			Ideally we want to use such a pairing to compare analytic Herr complexes for $S(\tau^{-1})$ and $S^*(1)$ for a co-$L$-analytic representation $S$. However, the pairing above doesn't seem to help because $\mathcal{R}_K(\chi_{\text{an}})$ is not an \'{e}tale $(\phi, \Gamma)$-module. Any duality statement obtained must involve at least one non-\'{e}tale $(\phi,\Gamma)$-module, which limits how much we can use this for global arithmetic. Still the existence of such a pairing is useful in many ways, and we use it to obtain some computations of $\Psi$-Herr complex cohomology. \lozengeend
		\end{rmk}
		
		\begin{cor} \label{herr_psi_computations}
			Suppose $D$ is a generic $L$-analytic $(\phi,\Gamma)$-module over $\mathcal{R}_K$ in the sense of Assumption \ref{assumption_generic}, then $$\dim _L \mathcal{H}_{\Psi, \Gamma_L}^j(D)= \begin{cases}0, & j \neq 1 \\ 1, & j=1  \end{cases} \qquad . $$
			$$\dim _L H_{\Psi, \Gamma_L}^j(D)= \begin{cases}0, & j \neq 1,2 \\ 1, & j=1,2  \end{cases} \qquad . $$
		\end{cor}
		\begin{proof}
			Our conditions ensure we are in the situation of Remark \ref{herr_phi_computations}, and we apply the respective duality theories of Lemma \ref{cts_herr_duality} and Lemma \ref{an_herr_duality}.
		\end{proof}
		
		\par We could think of either continuous or analytic Herr complexes as the `right versions' depending on whether we care about the module structure over $\Lambda_\infty$ or $\Lambda_\infty^L$. We are more interested in the latter, but we will have to make use of both versions to study the analytic one. By Corollary \ref{herr_psi_computations} we have that the Euler characteristics of the analytic and continuous Herr complexes are -1 and 0 respectively. The computations of Remark \ref{psi_euler_char} strongly suggest to us that they are the respective base changes of $D^\dagger_{\text{rig}}(S(\tau^{-1}))$ via the maps $\Lambda_\infty \rightarrow K$ and $\Lambda_\infty^L \rightarrow K$ given by an $L$-analytic character $\eta$ and $K$ a large enough extension of $L$. This is almost true, as the following proposition and remark will show, once we can find the right degree shift.
		
		\begin{prop} \label{herr_descent}
			Given an $L$-analytic $(\phi,\Gamma)$-module $D$ and an $L$-analytic character $\eta: \Lambda \rightarrow K$ there is a quasi-isomorphism
			\[	C^{\text{an}}_{\Psi,\Gamma}(D(\eta^{-1})) \cong_{\text{q-iso}} 	D(\eta^{-1})[\Psi-1][1] \otimes^{\mathbb{L}}_{\Lambda_\infty^L, \eta} K \]
		\end{prop}
		\begin{proof}
			Recall $L$-analytic $\eta$ as above lifts to a character $\widetilde{\eta}$ of $\Lambda_\infty$ factoring through $\Lambda_\infty^L$ and valued in $K$. Therefore we will show $\tilde{\eta}$ induces a quasi-isomorphisms
			$$  C^{\text{an}}_{\Psi,\Gamma}(D(\eta^{-1})) \cong_{\text{q-iso}} D[\Psi-1][1] \otimes^{\mathbb{L}}_{\Lambda_\infty^L} K .$$ The natural map of complexes defined by $\eta$ is just given by
			
			\par Recall that by Remark \ref{cauchy_riemann} we have $$ \widetilde{\eta}:\Lambda_\infty \xrightarrow{\overline{\nabla}=0} \Lambda_\infty^L \xrightarrow{\nabla=c} K $$ which factors through $\pi_{L}:\Lambda_\infty \rightarrow \Lambda_\infty^L$ by $L$-analyticity (i.e $\nabla\widetilde{\eta}=0$) and $c=\nabla\widetilde{\eta}$ is determined by $\eta$.  By \citep[Corollary 3.5 (2)]{Ste23} taking $R=\mathcal{R}_K$, $D^\prime=K$, $T=\Psi$ and $D= \Lambda_\infty^L$ and $\Lambda_\infty$ respectively, we get spectral sequences $$ \text{Tor}^{1-i}_{\Lambda_\infty^L}(H^j(D[\Psi-1]),K) \Rightarrow \mathcal{H}^{i+j}_{\Psi,\Gamma}(D(\eta^{-1})). $$
			
			\par  We can compare this with the spectral sequence of Remark \ref{tor_spectral_sequence} corresponding to base change by $\eta: S \rightarrow R$. Note that this is a homological spectral sequence as stated in terms of a bounded chain complex $K_\bullet$, as we wanted it to be a first quadrant spectral sequence. Now we must state it cohomologically, taking the substitution $K_j=C^{-j}$:
			$$  \text{Tor}^{-i}_{\Lambda_\infty^L}(H^{j}(C^\bullet),K) \Rightarrow H^{i+j}(C^\bullet \otimes^{\mathbb{L}}_{\Lambda^{L}_\infty} K) .$$
			
			We now take $C^\bullet = D[\Psi-1][1]$. By uniqueness of convergence, noting the degree shift of $i$ by 1 in the $E_2^{i,j}$ terms, we obtain the quasi-isomorphisms we need.
			
		\end{proof}

		\begin{rmk} \label{herr_descent_oc}
			\par It would be nice to have a continuous analogue of this. In fact in the setting of general $L/\mathbb{Q}$ with $d=[L:\mathbb{Q}_p]$ the results of Proposition \ref{herr_descent} will stay the same but the spectral sequence of \citep[Corollary 3.5 (2)]{Ste23} for the continuous Herr complex will be $$ \text{Tor}^{d-i}_{\Lambda_\infty}(H^j(D[\Psi-1]),K) \Rightarrow H^{i+j}_{\Psi,\Gamma}(D(\eta^{-1})).$$ Thus we will have a quasi-isomorphism 
			$$  C^{\text{cts}}_{\Psi,\Gamma}(D(\eta^{-1})) \cong_{\text{q-iso}} D[\Psi-1][d] \otimes^{\mathbb{L}}_{\Lambda_\infty^L} K. $$ This is not so strange when we recall that $C^{\text{cts}}_{\Psi,\Gamma}(D(\eta^{-1})$ will have cohomological dimension $d+1$ generically. \lozengeend
			
		\end{rmk}

		\par We now want to turn our attention to `descended' versions of the theory of the comparison map of cohomology in Definition \ref{comp_defn}, but we have a choice to make between using continuous and analytic Herr complexes to do so. Since these cohomology groups are coherent sheaves over either $\mathfrak{X}$ or $\mathfrak{X}_L$, we can specialise at an $L$-analytic character $\eta:\mathcal{O}_L^\times \rightarrow K$ by considering it as a closed point of either rigid space. The former doesn't capture the analyticity condition we need to use the analytic distribution algebra, and the latter is lacking in tools such as Pontryagin duality using $\chi_{\text{LT}}$. We will make use of different aspects of both.

		\begin{defn} \label{comp_descent}
			Let $S$, $\eta$ be a $G_L$ characters such that $S$ is co-$L$-analytic and $\eta$ is $L$-analytic. Assume $p > e(L/\mathbb{Q}_p)+1$. There exists a natural map of complexes $$ \text{comp}_{\eta}: C^\bullet(L,S(\eta^{-1})) \rightarrow C^{\text{an}}_{\Psi,\Gamma}(D^\dagger_{\text{rig}}(S(\tau^{-1}\eta^{-1}))) $$ defined so that in degree 1 satisfies the following commutative diagram with exact rows:
			\[ \begin{tikzcd}
				0 & \mathfrak{K} & H^1_{Iw}(L_\infty/L,S) \otimes \Lambda_\infty^L & D_{\text{rig}}^{\dagger}(S(\tau^{-1}))^{\Psi=1} \\
				0 & \mathfrak{K}_\eta & H^1(L,S(\eta^{-1})) & \mathcal{H}^1_{\Psi,\Gamma}(D_{\text{rig}}^{\dagger}(S(\tau^{-1}\eta^{-1}))).
				\arrow[from=1-1, to=1-2]
				\arrow[from=1-2, to=1-3]
				\arrow["\text{comp}"', from=1-3, to=1-4]
				\arrow[from=2-1, to=2-2]
				\arrow[from=2-2, to=2-3]
				\arrow["\text{comp}_\eta"', from=2-3, to=2-4]
				\arrow[from=1-2, to=2-2]
				\arrow["\eta"', from=1-3, to=2-3]
				\arrow["\alpha_{\eta}"', from=1-4, to=2-4]
			\end{tikzcd} \]
		\end{defn}	

            \par To show this map is well defined; given an element $x \in H^1(L,S(\eta^{-1}))$ choose a lift $y \in H^1_{\text{Iw}}(L_\infty/L,S) $ such that $\eta(y \otimes 1)=x$ i.e. we just choose any sequence in the inverse limit defining Iwasawa cohomology whose $G_L$-cohomology term is a twist of $x$ by $\eta^{-1}$. Then we define $\text{comp}_{\eta}(x)=\alpha_\eta(\text{comp}(y \otimes 1))$. This is well defined; if $y^\prime$ is another choice of lift then $\eta(y-y^\prime)=0$. Since the comparison map is $\Lambda^L_\infty$-equivariant, the vertical map $\alpha_\eta^1$ will send $\text{comp}(y-y^\prime)$ to 0 (this is easier to see after writing $\alpha_\eta^1$ out explicitly below).

            \par  The vertical base change map $\alpha_\eta^1$, can be written out explicitly as the degree 1 part of a map of complexes as follows. The map $$\alpha_{\eta}: D_{\text{rig}}^{\dagger}(S(\tau^{-1}))[\Psi-1] \rightarrow C_{\Psi,\Gamma}(D_{\text{rig}}^{\dagger}(S(\tau^{-1}\eta^{-1})))$$ on complexes of $(\phi, \Gamma)$-modules given by:
		\[ \begin{tikzcd} 
			0 &  0  \\
			0 &  D_{\text{rig}}^{\dagger}(S(\tau^{-1}\eta^{-1}))  \\
			D_{\text{rig}}^{\dagger}(S(\tau^{-1})) & D_{\text{rig}}^{\dagger}(S(\tau^{-1}\eta^{-1}))^{\oplus 2} \\
			D_{\text{rig}}^{\dagger}(S(\tau^{-1})) & D_{\text{rig}}^{\dagger}(S(\tau^{-1}\eta^{-1})) \\
			0 & 0 
			\arrow[from=2-1, to=2-2]
			\arrow[from=3-1, to=3-2, " {(0,\eta^{-1})} "]
			\arrow["\eta^{-1}"', from=4-1, to=4-2]
			\arrow[from=1-1, to=2-1]
			\arrow[from=2-1, to=3-1]
			\arrow["\Psi-1"', from=3-1, to=4-1]
			\arrow[from=4-1, to=5-1]
			\arrow[from=1-2, to=2-2]
			\arrow["{(Z, \Psi-1)}"', from=2-2, to=3-2]
			\arrow[" 1-\Psi \oplus Z"', from=3-2, to=4-2]
			\arrow[from=4-2, to=5-2]
		\end{tikzcd} \]
        
		\begin{rmk} \label{herr_stalks}
			Like the definition of \text{comp} given in Definition \ref{comp_defn}, the specialised version is a natural map given by base change. Since the objects on the top row are coadmissible $\Lambda_\infty^L$-modules due to \citep{Ste22b} and freeness of Iwasawa cohomology, $\text{comp}$ is really a map of coherent analytic sheaves over $\mathfrak{X}_L$. Then for an $L$-analytic $G_L$-character $\eta$ corresponding to a closed point $x_\eta \in \mathfrak{X}_L$, $\text{comp}_{\eta}$  is the map of stalks at $x_\eta$. This coincides with the map in Definition \ref{comp_descent} by the base change result of Prop \ref{herr_descent}.\lozengeend
		\end{rmk}

		\par We can also descend to the continuous Herr complex, which sees overconvergent rather than analytic extension classes in degree 1, but the advantage is that we have a better notion of duality. The disadvantage is that we are truly interested in analytic classes. Nevertheless we define such a descent map, which must factor through $\alpha_\eta$ by the universal property of the derived tensor product. First we will define this overconvergent version of $\text{comp}_\eta$ then we wil state the duality statement which we will use. 
		
		We define an analogous map $\beta_\eta$ for the continuous $\Psi$-Herr complex. Like the analytic setting, we can define it naturally as a base change map from Remark \ref{herr_descent_oc} or we can write it explicitly and check it commutes by hand. We do both here; explicitly, denoting $D^\prime = D^*(\chi_{\text{LT}} \eta^{-1})$, we write out a map $\beta_{\eta}: D[\Psi-1] \rightarrow  K_{\Psi, \Gamma}(D^\prime)[2]$ given below, using maps $A$, $B$ and $C$ written out in the definition of the Koszul complex:
		\[ \begin{tikzcd}[column sep=6cm]
			0 & D^\prime \otimes D^\prime \\
			D & (D^\prime \otimes D^\prime)^{\oplus3} \\
			D & (D^\prime \otimes D^\prime)^{\oplus3} \\
			0 & D^\prime \otimes D^\prime \\
			\arrow[from=1-1, to=2-1]
			\arrow[from=2-1, to=3-1, "\Psi-1"]
			\arrow[from=3-1, to=4-1]
			\arrow[from=1-2, to=2-2, "A"]
			\arrow[from=2-2, to=3-2, "B"]
			\arrow[from=3-2, to=4-2, "C"]
			\arrow[from=1-1, to=1-2]
			\arrow[rightarrow, from=2-1, to=2-2, "{((\gamma_1-1)\mathcal{T},(\gamma_2-1)\mathcal{T}, 0 )}"]
			\arrow[rightarrow, from=3-1, to=3-2, "{(0, (\gamma_1-1)\mathcal{T}, (\gamma_2-1)\mathcal{T})}"]
			\arrow[from=4-1, to=4-2]
		\end{tikzcd} \]	where $D^\prime = D^*(\chi_{\text{LT}} \eta^{-1})$ and $D \xrightarrow{\mathcal{T}} D^\prime$ is the appropriate twisting map.
		
		\begin{defn} \label{comp_descent_oc}
			Let $S$, $\eta$ be a $G_L$ characters such that $S$ is co-$L$-analytic and $\eta$ is $L$-analytic. There is a map of complexes $$ \text{comp}^\dagger_\eta:C^\bullet(L,S(\eta^{-1})) \rightarrow C^{\text{cts}}_{\Psi,\Gamma}(D^\dagger_{\text{rig}}(S(\tau^{-1}\eta^{-1})))[1-d] $$ which, in degree 1, provides a map $\text{comp}^{\dagger}_{\eta}$ satisfying the following commutative diagram and factoring through Definition \ref{comp_descent}:
			\[ \begin{tikzcd}
				0 & \mathfrak{K} & H^1_{Iw}(L_\infty/L,S) \otimes \Lambda_\infty^L & D_{\text{rig}}^{\dagger}(S(\tau^{-1}))^{\Psi=1} & 0 \\
				0 & \mathfrak{K}^{\dagger}_\eta & H^1(L,S(\eta^{-1})) & H^d_{\Psi,\Gamma}(D_{\text{rig}}^{\dagger}(S(\tau^{-1}\eta^{-1}))) 
				\arrow[from=1-1, to=1-2]
				\arrow[from=1-2, to=1-3]
				\arrow["\text{comp}"', from=1-3, to=1-4]
				\arrow[dotted,from=1-4, to=1-5]
				\arrow[from=2-1, to=2-2]
				\arrow[from=2-2, to=2-3]
				\arrow["\text{comp}^{\dagger}_\eta"', from=2-3, to=2-4]
				\arrow[from=1-2, to=2-2]
				\arrow["\eta"', from=1-3, to=2-3]
				\arrow["\beta_{\eta}^1"', from=1-4, to=2-4]
			\end{tikzcd} \]
		\end{defn}	
		
		\par Note that by \citep[Corollary 4.4]{Ste23} and vanishing of $\mathcal{H}^2_{\Psi,\Gamma}(D_{\text{rig}}^{\dagger}(S(\tau^{-1}\eta^{-1})))$, we have that  $H^d_{\Psi,\Gamma}(D_{\text{rig}}^{\dagger}(S(\tau^{-1}\eta^{-1})))=H^1_{\Psi,\Gamma}(D_{\text{rig}}^{\dagger}(S(\tau^{-1}\eta^{-1}))) $. Thus we can think of $\beta_\eta^1$ as a map of degree 1 terms of the above map of complexes. The map $\text{comp}_{\eta}^{\dagger}$ is well defined by a similar reasoning to $\text{comp}_{\eta}$. Moreover by an analogous argument to Remark \ref{herr_stalks}, we can think of this map as the map of stalks at $x_\eta$ of the map of coherent sheaves on $\mathfrak{X}$, and by the base change result of Remark \ref{herr_descent_oc} this coincides with Definition \ref{comp_descent_oc}.

	\section{A Locally Analytic Main Conjecture}
		
		\subsection{The $L$-analytic regulator map}
		
		\par As mentioned previously, an obstruction to the inert prime theory is that Perrin-Riou's big logarithm map does not generalise to extensions which are not locally cyclotomic. A Lubin--Tate version big logarithm map has been defined in \citep{SV23} using $p$-adic Hodge theory applied to the Lubin--Tate extension $\Gamma_L$ of an arbitrary finite extension $L/\mathbb{Q}_p$. This regulator map can only be applied when the dual of $V$ is $L$-analytic, so that the twist $V(\tau^{-1})$ with $\tau = \chi_{\text{cyc}}^{-1}\chi_{\text{LT}}$ is $L$-analytic. We say such $V$ is co-$L$-analytic or sometimes co-analytic. This twist by $\tau$ is a compensation for the difference between Pontryagin and local Tate duality theories - this will be explained in a later remark when our language has developed further.
		
		\begin{defn} \label{analytic_reg}
			Let $S$ be a crystalline $\Gamma_L=\mathrm{Gal}(L_\infty/L)$-representation such that $S$ is $co-L$-analytic and has negative Hodge--Tate weight at the identity embedding and $H^0(L,V(\tau^{-1}))=0$. Then $L$-analytic regulator map is defined in \citep[\S3]{SV23} as a map $$\mathcal{L}_S: H^1_{Iw}(L_\infty/L, S) \rightarrow \left( \Lambda^{L}_\infty \hat{\otimes}_L \mathbb{C}_p \right) \otimes D_{\text{cris},L}(S(\tau^{-1}))$$
		\end{defn}
		
		\par The rest of this section will be spent interpreting this map and developing tools to study it in order to obtain an $L$-analytic main conjecture using it. Note that the vanishing of the above $H^0$ group doesn't fall into the hypothesis of Assumption \ref{assumption_V_eta}, but we will assume it below in Assumption \ref{assumption_L-an}.  We also fix a basis $\xi$ of $D_{\text{cris}}(S(\tau^{-1}))$ which will be important when proving an explicit reciprocity law in the future but arbitrary now - so we will not include it in the notation.
		
		\par Since the regulator map is contstructed using the $p$-adic Fourier tranform, we are forced to to base change by an extension $K/\widehat{L_\infty}$ as in Remark \ref{padic_fourier} so that the $\mathfrak{X}_L$ is $K$-isomorphic to the closed unit disk by methods of \citep{ST01}. Without losing much generality, the authors of \textit{op. cit.} take $K=\mathbb{C}_p$. Since $\mathcal{O}_K$ is not noetherian we lose the hope of getting nice integral results. Since $K$ is a complete $E_p$-vector space therefore a faithfully flat $E_p$-module and this base change behaves nicely. In particular it commutes with taking cohomology of the Selmer complex and with all exact sequences. Thus we can formulate the same main conjecture over $\Lambda_{\infty,K}^{E_p}=\Lambda_\infty^{E_p} \hat{\otimes}_{E_p} K$. The definition of characteristic ideal will also be the same, as after base changing to $K$ we still have a Pr\"ufer domain by Remark \ref{char_prufer}.
		
		\begin{cor} \label{C_p-kmc}
			Kato's Main conjecture (Corollary \ref{analytic_kmc}) holds as a statement of $\Lambda_{\infty,K}^{E_p}$-modules after base change by $\hat{\otimes}_L K$
		\end{cor}
		
		\par Taking $S$ to be the middle graded piece of the ordinary filtration of $V=V_{\mathfrak{P}}(\Pi)$, we will later consider the image of the Euler system as an element $$\text{comp}(c^\Pi \otimes 1) \in D_{\text{rig}}^{\dagger}(S(\tau^{-1}))^{\Psi=1}$$ which is non-zero by Assumption \ref{assumption_V_delta}; this is the representation $S$ to which we apply the regulator map. We use the relationships between the different categories of $(\phi_L, \Gamma_L)$-modules to do this; using Definition \ref{LT_kisin-ren} we can view the Euler system as a $\Psi$-invariant element of $D_{\text{LT}}(S(\tau^{-1}))$. Since $S$ is an overconvergent character this this has a model over $\mathbb{B}^\dagger_L$ by \citep[Remark A.09]{SV23}. To get a regulator map with domain $D_{\text{rig}}^{\dagger}(S(\tau^{-1}))^{\Psi=1}$ we would need to base change via $\iota: \mathbb{B}_L^\dagger \hookrightarrow \mathcal{R}_L \rightarrow \mathcal{R}_K$ from but this will not give the required map. The universal property of tensor product of $\Lambda$-modules, however, gives us a map
		$$ \overline{\mathcal{L}}_S: H^1_{\text{Iw}}(L_\infty/L,S) \otimes_{\Lambda} \Lambda_{\infty,K}^L \rightarrow \Lambda_{\infty,K}^{L} $$ such that $\overline{\mathcal{L}}_S(c \otimes \xi) = \xi\mathcal{L}_S(c)$.

		\par We let $V^1 \subset V(\eta^{-1})$ be the Panchishkin subrepresentation attached the the rank 1 region $\Sigma^{(2)}$ of Figure \ref{fig1}, with $\eta$ a twist in the rank 0 region $\Sigma^{(0)}$. Let $S$ be the 1-dimensional quotient of $V^1$ such that $V^1/S$ only has positive Hodge--Tate weights (this $S$ is just the analogue of $V^1/V^0$ from the split prime case). Since we assumed $V$ is crystalline, $S$ must also be crystalline. By the properties of the Euler system construction in \citep{LSZ21} and our assumption of non-vanishing, $\text{im}(c^\Pi)$ is non-trivial in the local Iwasawa cohomology of $S$ over the Lubin--Tate extension. We can apply the regulator map provided that $S$ is co-$L$-analytic. This condition is that the middle Hodge--Tate weight of $S$ at the conjugate embedding is 1, which forces the condition $r=a$ in the initial choice of twist of Hecke character; this is a line on the parameter space Figure \ref{fig1} which we can see intersects both the Euler system region $\Sigma^{(2)}$ and the rank 0 region $\Sigma^{(0^{\prime})}$. We label this in red in the figure below. However, we lose a degree of freedom in the space of Hecke characters we can twist by. It remains to be seen whether this loss of variation is necessary or can be recovered.
		
		\begin{figure}[H]
			\includegraphics[width=14cm]{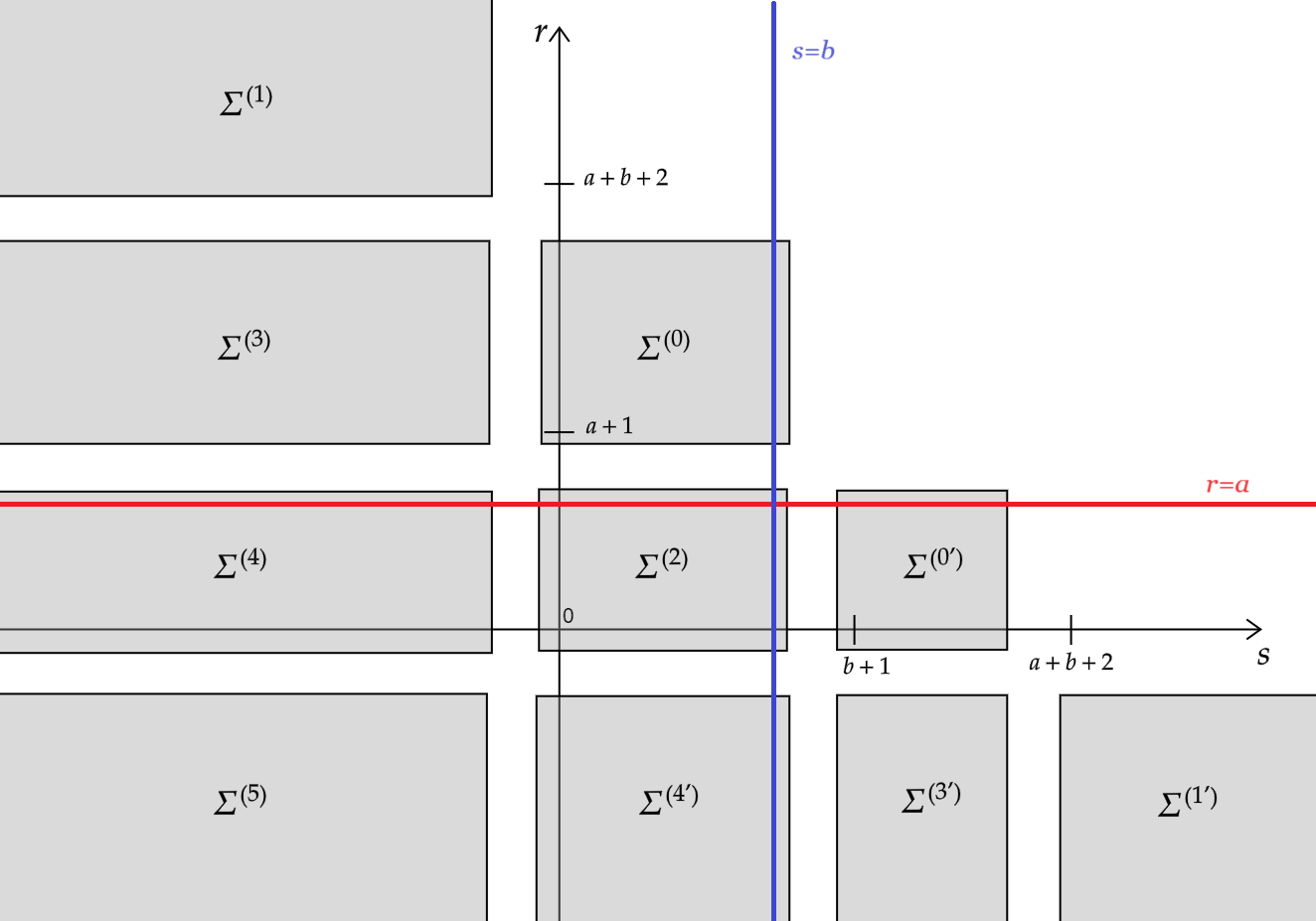}
			\caption{The locus horizontal $r=a$ of Hecke characters $\eta$ such that the 1-dimensional middle piece $S(\eta^{-1})$ of $V(\eta^{-1})$ has $S$ co-$E_p$-analytic. The vertical line $s=b$ is the conjugate locus where $S$ is co-$\bar{L}$-analytic.}
			\label{fig3}
		\end{figure}
		
		\begin{rmk}
			The locus in Figure \ref{fig3} intersects both the Euler system region $\Sigma^{(2)}$, and one $p$-adic $L$-function box similar to our case with $\Sigma^{(0^{\prime})}$, which is exactly what we want; we can think of this locus as a bridge over which we can apply the regulator map and hopefully an explicit reciprocity law down the line. The conjugate $p$-adic $L$-function box $\Sigma^{(0)}$ is not accessible on this locus but we could access it by using a notion of $\bar{E_p}$-analyticity i.e. applying the conjugate element of $\text{Gal}(E_p/\mathbb{Q}_p)$ to everything. This will give a locus $s=b$ (vertical) on the parameter space which will play exactly the conjugate role. It appears that the $E_p$-analytic and $\bar{E_p}$-analytic theories must exist entirely separately as the intersection of both conditions is only a point of the diagram, twisting by a Hecke character of fixed $\infty$-type $(b,a)$ only. \lozengeend
		\end{rmk}
		
		\begin{ass} \label{assumption_L-an}
			The 1-dimensional subquotient $S$ of $V$ is co-$L$-analytic with Hodge-Tate weight $ < 0$ at the identity embedding, $\eta$ is an $L$-analytic character, and $H^0(E_p,S(\eta^{-1})=0$.
		\end{ass}
		
		\par Similar to the split case, when $S$ is co-$L$-analytic we implicitly fix a basis of $D_{\text{cris},L}(S(\tau^{-1}))$ so we can define $L_{p, E_p\text{-an}}^*(\Pi) := \mathcal{L}_S(c^\Pi) \in \Lambda^{E_p}_\infty \hat{\otimes}_{E_p} \mathbb{C}_p$ to be the $L$-analytic $p$-adic distribution which we conjecture to interpolate complex $L$-values of $\Pi$. In this inert setting we have no comparison to an analytic construction, so we are even further from an explicit reciprocity law. As the $L$-analytic regulator map is a far more modern construction, work is still ongoing to develop so called $L$-analytic explicit reciprocity laws. This are discussed thoroughly in the survey article \citep{Ven23}, and evidence is certainly presented that it is a suitable regulator map.

		\par An important property of a regulator map is that it interpolates the Bloch--Kato dual exponential and logarithm maps after specialising to characters. This maintains a link between the image of the Euler system and $H^1_f$, which is the kernel of $\text{exp}^*$. Our $L$-analytic regulator map satisfies a Lubin--Tate version of the interpolation property, stated below. This result will not use twisted versions Bloch--Kato logarithm and dual exponential maps by the action of $\tau$, c.f. \citep[(185)]{SV23}, given by;
		$$ \widetilde{\text{exp} }_{L, W^*(1)}: H_f^1\left(L, W^*(1)\right) \rightarrow  \text{Fil}^{0,-}D_{\text{dR}, L}\left(W^*\left(\chi_{L T}\right)\right), $$
		$$ \widetilde{\log }_{L, W^*(1)}: H_f^1\left(L, W^*(1)\right) \rightarrow D_{\text{cris}, L}\left(W^*\left(\chi_{L T}\right)\right). $$ We also fix a basis element $ e_j \in D_{\text {cris }, L}\left(L\left(\chi_{L T}^j\right)\right)$ as defined in op. cit.
		
		\begin{thm}\citep[Theorem 5.2.26]{SV23} \label{analytic_interpolation}
			Assume that $T$ is co-$L$-analytic and crystalline and that it satisfies $\mathrm{Fil}^{-1} D_{\mathrm{cris }, L}\left(T^*(1)\right)=D_{\mathrm{cris }, L}\left(T^*(1)\right)$ and $D_{\mathrm{cris }, L}\left(T^*(1)\right)^{\phi_L=\varpi_L^{-1}}=D_{\mathrm{cris }, L}\left(T^*(1)\right)^{\phi_L=1}=0$. Suppose further that the operators $1-\varpi_L^{-1} \phi_L^{-1}, 1-\frac{\varpi_L}{q} \phi_L$ are invertible on $D_{c r i s, L}\left(T\left(\tau^{-1} \chi_{L T}^j\right)\right)$, respectively. Then for a period $\Omega \in \mathbb{C}_p$;
			\begin{itemize} \item for $j \geqslant 0$
				$$
				\Omega^j \mathcal{L}_T(y)\left(\chi_{L T}^j\right) =j !\left(\left(1-\varpi_L^{-1} \phi_L^{-1}\right)^{-1}\left(1-\frac{\varpi_L}{q} \phi_L\right) \widetilde{\exp }_{L, T\left(\chi_{L T}^{-j}\right), \mathrm{id}}^*\left(y_{\chi_{L T}^{-j}}\right)\right) \otimes e_j 
				$$
				\item for $j \leqslant-1$ :
				$$
				\Omega^j \mathcal{L}_T(y)\left(\chi_{L T}^j\right) =\frac{(-1)^j}{(-1-j) !}\left(\left(1-\varpi_L^{-1} \phi_L^{-1}\right)^{-1}\left(1-\frac{\varpi_L}{q} \phi_L\right) \widetilde{\log }_{L, T\left(\chi_{L T}^{-j}\right), \mathrm{id}}\left(y_{\chi_{L T}^{-j}}\right)\right) \otimes e_j 
				$$
			\end{itemize}
		\end{thm}
		
		%Interpolation property (what about for general locally L-analytic characters?)
		
		Heuristically, once we specialise at a power of the Lubin--Tate character, the regulator map looks like $(1-\varpi_L^{-1}\phi_L^{-1})$ times either $\text{exp}_{BK}^*$ or $\text{log}_{BK}$ depending on whether its a positive or negative power; when $L=E_p$ recall that $\varpi_{E_p}=p$ and this operator is $(1-\tfrac{1}{p} \phi)$. This has been generalised to hold for all crystalline $L$-analytic characters, which look like powers of $\chi_{\text{LT}}$ twisted by a finite order character, in a follow up paper \citep{SaV24}. This interpolation formula will be key to proving an explicit reciprocity law in the future, and we will use it in this work to show that the local conditions we construct compare with local Bloch--Kato Selmer groups at $p$.

        \begin{lem}
            The regulator map $\overline{\mathcal{L}}_S$ factors further to $$ \widetilde{\mathcal{L}}_S: D_{\text{rig}}^\dagger(S(\tau^{-1}))^{\Psi=1} \hat{\otimes}_L K \rightarrow \Lambda_{\infty,K}^L .$$
        \end{lem}
        \begin{proof}
            Let $\widehat{H} \subset H^1_\text{Iw}(L_\infty/L,S) \otimes_\Lambda \Lambda_\infty^L $ be the subset of elements whose specialisations at all $L$-analytic characters $\eta$ lie in $H^1_f(L,S(\eta^{-1})).$ Later in this paper (specifically Proposition \ref{L-an_crys}) we show that $$\ker(\text{comp}) \subseteq \widehat{H}. $$ Moreover, the interpolation property of Bloch--Kato exponential maps, proved in necessary generality in \citep{SV24}[Theorem 2.6.3] shows that $$ \widehat{H} \subseteq \ker(\mathcal{L}_{S}).$$ In particular this shows the regulator map factors through the comparison map of Definition \ref{comp_defn}. Assuming Conjecture \ref{steingart_conj}, we obtain the factorization required.
        \end{proof}
        
		\begin{lem} \label{ker_reg}
			Suppose $S$ is a crystalline co-$L$-analytic $G_L$ representation with negative Hodge--Tate weight at the identity embedding  and $\eta$ is an $L$-analytic character satisfying $H^0(L,S(\eta^{-1}))=0$. Then, $\ker \overline{\mathcal{L}}_S = 0.$ 
		\end{lem}
		\begin{proof}
			\par First we show  $\ker \mathcal{L}_S = 0.$ The regulator map is given by the chain of maps \citep[(141), \S5]{SV23}. All of the maps are a priori injective except for $$ (1-\tfrac{\phi}{p}): N(S(\tau^{-1}))^{\Psi-1} \rightarrow N(S(\tau^{-1}))^{\phi=0}$$ where $N(-)$ is the associated Wach module of the $(\phi, \Gamma)$-module, since $S$ is crystalline. This operator is injective if $p$ is not the crystalline eigenvalue of $S$. Since $S$ is crystalline, the argument of Assumption \ref{assumption_V_eta} shows this follows from $H^0(E_p,S(\eta^{-1}))=0$.
			\par Now to show $\ker \widetilde{\mathcal{L}}_S = 0$ we simply need to show that $$ D_{\text{LT}}(S(\tau^{-1}))^{\Psi-1} \rightarrow D_{\text{rig}}^\dagger(S(\tau^{-1}))^{\Psi-1} $$ is injective. Since $S$ is overconvergent, it is sufficient to show that $\mathbb{B}_L^\dagger \rightarrow \mathcal{R}_L$ is flat. Since $\mathbb{B}_L^\dagger$ is a field this is true and the result holds.
		\end{proof}

		\begin{lem} \label{coker_reg}
			Under the assumptions of Lemma \ref{ker_reg}, $\text{coker}(\widetilde{\mathcal{L}}_S)$ is a torsion $\Lambda^{E_p}_{\infty, \mathbb{C}_p}$-module.
		\end{lem}
		\begin{proof}
			$\widetilde{\mathcal{L}}_S$ presents the regulator as a map between 1-dimensional free $\Lambda_{\infty,K}^L$-modules. If the kernel is 0, this implies immediately that the cokernel is torsion. 
		\end{proof}
		\par It is probably achievable to show this cokernel is 0 under some suitable conditions, but this is not needed for the following argument.
		
		\subsection{The $E_p$-analytic main conjecture}
		%Discussion of p-adic L-functions
		%Exact sequence argument to get IMC over \Lambda_\infty^L with "p-adic L-function"
		
		To state a main conjecture in the inert case we still need to define the analogous `rank 0' local conditions when $p$ is inert, which we can now do given the tools developed so far. At primes $l \neq p$ we still use the unramified local condition, and at the inert prime $p$ we will do something different. Recall the rank 1 local condition was defined by the rank 1 Panchishkin subrepresentation, which still exists for $p$ inert so we can define $\Delta_p^1$ as before. By our construction after Definition \ref{selmer_complex}, we have an isomorphism $$ \iota: R\Gamma(E_p,V(\eta^{-1});\Delta^1_p) \xrightarrow{\sim} R\Gamma(E_p,V^1).$$ When $\eta$ is a twist in region $\Sigma^{(0)}$ and lying on the horizontal $L$-analytic locus of Figure \ref{fig3}, $V^1$ has all the non-positive Hodge--Tate weights plus an extra weight equal to 1, i.e. the fixed weight from the co-$E_p$-analyticity condition, so $V^1$ is not quite the Panchishkin subrepresentation. A lot of the technical work so far has been to remedy this difference. Let $\mathrm{pr}: V^1 \twoheadrightarrow S$ be the projection of the local restriction of the Galois representation to the middle $L$-analytic piece, which lifts to a map $\mathrm{pr}^*$ of group cohomology.
		
		%We need to sort out notation of local and global Selmer groups - and \tilde should only be used globally
		
		\begin{defn} \label{analytic_rank0_selmer}
			When $p$ is inert in $E$ we define a local condition at $p$, $C^\bullet_{\Delta^0_p}$, as the mapping cone of $$ \mathrm{comp} \circ [(\mathrm{pr}^* \circ \iota^{-1})\otimes^{\mathbb{L}} \Lambda_\infty^{E_p}]: R\Gamma(E_p,V;\Delta^1_p) \otimes^{\mathbb{L}}_\Lambda \Lambda_\infty^{E_p} \rightarrow D^\dagger_{\text{rig}}(S(\tau^{-1}))[\Psi-1][1]. $$
			Using this definition at $p$ and unramified local conditions inside $R\Gamma(E_v,V) \otimes^{\mathbb{L}}_\Lambda \Lambda_\infty^{E_p}$ for $v \neq p$ we obtain the rank 0 Selmer complex $\widetilde{R\Gamma}(E,\mathbb{V}_\infty^{E_p};\Delta^0)$. We will denote its Selmer groups by $$\widetilde{H}^i(E, \mathbb{V}_\infty^{E_p};\Delta^0).$$ In degree 2 this will live inside $H^2(E,\mathbb{V}_\infty) \otimes_{\Lambda_\infty} \Lambda_\infty^{E_p}$ by the exact sequences in Propostition \ref{inert_descent2} coming from the base change spectral sequence.
			
		\end{defn}
		
		\par We can compute the Euler characteristic as a complex of $\Lambda_\infty^L$-modules as follows, using Remark \ref{psi_euler_char}; $$ \chi_{\Lambda_\infty^L}(\widetilde{R\Gamma}(E,V;\Delta^0)) = \chi_{\Lambda_\infty^L}(D^\dagger_{\text{rig}}(S(\tau^{-1}))[\Psi-1][1]) - \chi_{\Lambda}(\widetilde{R\Gamma}(E,\mathbb{V};\Delta^1)) = 1-1 = 0 $$
		
		\begin{lem} \label{analytic_rank0_vanishing}
			If $L^*_{p,E_p\text{-an}} \neq 0$ then $\widetilde{H}^1(E,\mathbb{V}_\infty^{E_p};\Delta^0)=0$
		\end{lem}
		\begin{proof}
			By definition we have an injection $\widetilde{H}^1(E,\mathbb{V}_\infty^{E_p};\Delta^0) \hookrightarrow \widetilde{H}^1(E, \mathbb{T};\Delta^1) \otimes_{\Lambda_\infty} \Lambda_\infty^{E_p}$, the latter of which is torsion-free of $\Lambda_\infty^{E_p}$-rank 1 by Corollary \ref{analytic_kmc}. The cokernel of this injection is $ D^\dagger_{\text{rig}}(S(\tau^{-1}))^{\Psi-1}$, in which we can consider the image of the Euler system. This is non-zero by the assumption combined with Lemma \ref{ker_reg}. Thus the image of the injection must be 0.
		\end{proof}
		
		This lemma is a sanity check and an important ingredient in the following proof of the `rank 0' main conjecture. 
		
		\begin{thm} \label{analytic_imc}
			Suppose $V$ and $\delta$ satisfy the assumptions of \ref{assumption_V_delta}, and the middle piece $S$ of $V$ is co-$L$-analytic. Then $$\mathrm{char}_{e_\delta\Lambda_{\infty, \mathbb{C}_p}^{E_p}}e_\delta\widetilde{H}^2(E,\mathbb{V}_\infty^{E_p};\Delta^0) \bigm \vert (e_\delta L^*_{p,E_p\text{-an}}(\Pi)) $$
		\end{thm}
		\begin{proof}
			By Definition \ref{analytic_rank0_selmer} and using \citep[11.2.1]{KLZ17}, the $\Delta^0$ Selmer complex fits into the distinguished triangle:
			
			$$ \widetilde{R\Gamma}(E,\mathbb{V}_\infty^{E_p};\Delta^0) \rightarrow \widetilde{R\Gamma}(E,\mathbb{V}_\infty^{E_p};\Delta^1) \rightarrow {D}_{\text{rig}}^\dagger(S)[\Psi-1][1] $$
			
			\par Note that $ H^2(D^\dagger_{\text{rig}}(S(\tau^{-1}))[\Psi-1])= D^\dagger_{\text{rig}}(S(\tau^{-1}))/(\Psi-1)$ is a quotient of a local $H^2$ Iwasawa cohomology group which vanishes (see Lemma \ref{steingart_kernel}), therefore it is 0. We also have the vanishing of the $\Delta^0$ Selmer group in degree 1 from Lemma \ref{analytic_rank0_vanishing}. Thus the distinguished triangle gives us an exact sequence of $\Lambda_\infty^{E_p}$-modules:
			
			\[ \begin{tikzcd} 0 & \dfrac{ \widetilde{H}^1(E,\mathbb{V}_\infty^{E_p};\Delta^1)}{im(c^\Pi)} & \dfrac{D_{\text{rig}}^\dagger(S(\tau^{-1}))^{\Psi=1}}{im(c^\Pi)} \\ & \widetilde{H}^2(E,\mathbb{V}_\infty^{E_p};\Delta^0) & \widetilde{H}^2(E,\mathbb{V}_\infty^{E_p};\Delta^1) & 0 
				\arrow[from=1-1, to=1-2]
				\arrow[from=1-2, to=1-3]
				\arrow[from=1-3, to=2-2]
				\arrow[from=2-2, to=2-3]
				\arrow[from=2-3, to=2-4]
			\end{tikzcd} \]
			
			Since the first two and last modules are $\Lambda_\infty^{E_p}$-torsion, it follows that $\widetilde{H}^2(E,\mathbb{V}_\infty^{E_p};\Delta^0)$ is also torsion i.e. the hypothesis of the theorem is well defined.

			Applying the regulator map and $ - \hat{\otimes}_{E_p} \mathbb{C}_p$, we also get the following exact sequence:
			
			$$ 0 \rightarrow \dfrac{D_{\text{rig}}^\dagger(S(\tau^{-1}))^{\Psi=1}}{im(c^\Pi)} \xrightarrow{\widetilde{\mathcal{L}}_{S}} \dfrac{\Lambda_{\infty,K}^{E_p}}{L^*_{p,E_p\text{-an}}(\Pi)} \rightarrow \mathrm{coker}(\widetilde{\mathcal{L}}_{S}) \rightarrow 0 $$
			where $\mathrm{coker}(\widetilde{\mathcal{L}}_{S})$ is a torsion module by Lemma \ref{coker_reg}. We combine these exact sequences with multiplicativity of $\mathrm{char}_{\Lambda_{\infty, \mathbb{C}_p}^{E_p}}$. Recalling that $\Lambda^{E_p}_{\infty, \mathbb{C}_p}$ is Pr\"ufer and therefore non-zero ideals are invertible, we get
			$$ \dfrac{\mathrm{char}_{\Lambda_{\infty,K}^{E_p}}\left( \dfrac{\widetilde{H}^1(E,\mathbb{V}_\infty^{E_p};\Delta^1)}{im(c^\Pi)} \right) }{\mathrm{char}_{\Lambda_{\infty,K}^{E_p}} ( \widetilde{H}^2(E,\mathbb{V}_\infty^{E_p};\Delta^1)) } =  \dfrac{ \left( L^*_{p,E_p\text{-an}}(\Pi) \right)}{\mathrm{char}_{\Lambda_{\infty,K}^{E_p}}( \widetilde{H}^2(E,\mathbb{V}_\infty^{E_p};\Delta^0)) \mathrm{char}_{\Lambda_{\infty,K}^{E_p}} (\text{coker}(\widetilde{\mathcal{L}}_{S})) }$$
			
			\par We now apply the $e_\delta$ idempotents; by the divisibility of Corollary \ref{analytic_kmc} the left hand side is a proper ideal. Since the cokernel term is on the bottom of the fractional ideal, we obtain the divisibility of ideals we are looking for.
		\end{proof}
		
		\section{Bounding the Bloch--Kato Selmer group}
		
		\subsection{Descended Selmer structures}
		
		\par We are now interested in a descended statement of a main conjecture; in the split case we considered base change along the projection $\Lambda \rightarrow \mathcal{O}$ (in two steps), but we can no longer do this. We have inverted $p$ in defining the distribution algebras, and there is not a good notion of integrality of distributions inside $\Lambda_\infty^L$ when $L \neq \mathbb{Q}_p$. 	
		
		\begin{rmk} \label{ardakov_berger}
			This is best expressed in the work of Ardakov-Berger \citep{AB23} where the question is asked how bounded distibutions in $\Lambda_\infty^L$ (in the LF topology) compare to the image of $\Lambda \hookrightarrow \Lambda_\infty^L$. In the cyclotomic case they are known to coincide, but it is unclear in the general $L$-analytic case. Our `motivic' $L$-analytic $p$-adic $L$-functions live inside $\Lambda_\infty^L$ but they are attached to ordinary representations and thus they should really live in $\Lambda$, but the work of op. cit. shows that it is currently insufficient to just show they are bounded. As with the other deep algebraic issues occurring in the $L$-analytic setting, this is to do with the fact that $\mathfrak{X}_L$ is a transcendentally twisted disk. The authors also present some criteria, for example Theorems 1.7.3 and 1.8.1, in which bounded distributions are exactly $\Lambda$ in our setting of $L=\mathbb{Q}_{p^2}$ but these are not easily computable or sufficiently general for the purposes of this paper.  \lozengeend
		\end{rmk}
		
		\par Nevertheless we are still interested in base change to field level. Let $\eta$ be an $L$-analytic character, giving us a point on the character variety $\mathfrak{X}_L$. Thus it has a lift to a unique character of $\Lambda_\infty \rightarrow K$ factoring through $\pi_{E_p}:\Lambda_\infty \rightarrow \Lambda_\infty^{E_p}$ for a suitably large field extension $K$ of $\mathbb{Q}_p$. This gives us two possible descended versions of the $\Delta^0$ Selmer complex, analogous to the two descended versions of $\text{comp}$ by a character $\eta$ we constructed.

		\begin{defn} \label{rank0_descent_an}
			We define the analytic descended rank 0 local condition at $p$ by
			$$R\Gamma(E_p,V(\eta^{-1});\Delta^0_{\text{an},p})) = \text{Cone}(\text{comp}_\eta \circ (pr^* \circ \iota^{-1})$$ where $\iota$ and $pr^*$ are the same maps as in Definition \ref{analytic_rank0_selmer} and $\text{comp}_\eta$ is defined in Definition \ref{comp_descent}. Combining this with the unramified condition at $v \neq p$ gives us the Selmer complex $\widetilde{R\Gamma}(E, V(\eta^{-1});\Delta^0_{\text{an}})$.
		\end{defn}

            \par We remark that this is well defined since our assumpions imply vanishing of local cohomology of $H^2(L,S)$ and thus $R\Gamma_{\text{Iw}}(L_\infty/L,S)$ is concentrated in degree 1. Thus $H^1\left(R\Gamma_{\text{Iw}}(L_\infty/L,S) \otimes^{\mathbb{L}}_\Lambda \Lambda_\infty^L \right)=H^1_{\text{Iw}}(L,S \otimes \Lambda_\infty^L).$ This does not hold for the global Iwasawa complex.
		
		\begin{defn} \label{rank0_descent_cts}
			We define the analytic descended rank 0 local condition at $p$ by
			$$R\Gamma(E_p,V(\eta^{-1});\Delta^0_{\dagger,p}) = \text{Cone}(\text{comp}^\dagger_\eta \circ (pr^* \circ \iota^{-1}))$$ where $\iota$ and $pr^*$ are the same maps as in Definition \ref{analytic_rank0_selmer} and $\text{comp}^\dagger_\eta$ is defined in Definition \ref{comp_descent_oc}. Combining this with the unramified condition at $v \neq p$ gives us the Selmer complex $\widetilde{R\Gamma}(E, V(\eta^{-1});\Delta^0_{\dagger})$.
		\end{defn}
		
		\par The analytic $\Delta^0_{\text{an}}$ complex is really the one we want to study, but we are hindered due to a lack of tools in the study of analytic cohomology. Whilst the $\Delta^0_\dagger$ complex is the wrong one, it comes with a better toolkit and we can later compare back to the analytic descended complex. We will think of it as follows; they are the $L$-analytic and $\mathbb{Q}_p$-analytic base change of the $\Psi=1$ complex. We will demonstrate this using Proposition \ref{herr_descent} and Remark \ref{herr_descent_oc}.
		
		\begin{prop} \label{rank_0_an_basechange}
			Suppose $V$ has co-$L$-analytic middle piece $S$  and $\eta$ is an $L$-analytic character of $\Lambda$ valued in a field $K/L$ containing $\Omega_{\text{LT}}$ we have a quasi-isomorphism;
			
			\[
			\widetilde{R\Gamma}(E, V(\eta^{-1});\Delta^0_{\text{an}}) \cong_{\text{q-iso}}
			\widetilde{R\Gamma}(E, \mathbb{V}_\infty^{E_p};\Delta^0) \otimes^{\mathbb{L}}_{\Lambda_\infty^{E_p}} K \]
			
		\end{prop}
		\begin{proof}
			\par We have the following commutative diagram with distinguished triangles as rows 
			
			\[ \begin{tikzcd}[scale cd=0.8]
				\widetilde{R\Gamma}(E, \mathbb{V}_\infty^{E_p};\Delta^0) &\widetilde{R\Gamma}(E, \mathbb{V}_\infty^{E_p};\Delta^1) & D_{\text{rig}}^\dagger(S(\tau^{-1}))[\Psi-1][1] \\
				\widetilde{R\Gamma}(E, \mathbb{V}_\infty^{E_p};\Delta^0) \otimes^{\mathbb{L}}_{\Lambda_\infty^{E_p}, \eta} K & \widetilde{R\Gamma}(E, \mathbb{V}_\infty^{E_p};\Delta^1) \otimes^{\mathbb{L}}_{\Lambda_\infty^{E_p}, \eta} K & D_{\text{rig}}^\dagger(S(\tau^{-1}))[\Psi-1][1] \otimes^{\mathbb{L}}_{\Lambda_\infty^{E_p}, \eta} K \\
				\widetilde{R\Gamma}(E, V(\eta^{-1});\Delta^0_{\text{an}}) & R\Gamma(E,V^1(\eta^{-1})) & C^{\text{an}}_{\Psi,\Gamma}(D_{\text{rig}}^\dagger(S(\tau^{-1}\eta^{-1})))
				\arrow[from=1-1, to=1-2]
				\arrow[from=1-2, to=1-3]
				\arrow[from=2-1, to=2-2]
				\arrow[from=2-2, to=2-3]
				\arrow[from=3-1, to=3-2]
				\arrow[from=3-2, to=3-3]
				\arrow[from=1-1, to=2-1, " \otimes^{\mathbb{L}}_{\Lambda_\infty^{E_p}, \eta} K "]
				\arrow[from=1-2, to=2-2, " \otimes^{\mathbb{L}}_{\Lambda_\infty^{E_p}, \eta} K "]
				\arrow[from=1-3, to=2-3, " \otimes^{\mathbb{L}}_{\Lambda_\infty^{E_p}, \eta} K "]
				\arrow[from=2-1, to=3-1, "\vartheta"]
				\arrow[from=2-2, to=3-2, "\sim"]
				\arrow[from=2-3, to=3-3, "\sim"]
			\end{tikzcd} \]
			There are some things to justify --- first note that the middle row is a distinguished triangle by \citep[\S1.2]{Bal10} as long the complexes in the top row are all perfect complexes of $\Lambda_\infty^{E_p}$-modules, since the perfect derived category $D_{\text{perf}}(\Lambda_\infty^{E_p})$ is a tensor triangulated category. We show they are perfect; $D_{\text{rig}}^\dagger(S(\tau^{-1}))[\Psi-1]$ perfect by \citep[Theorem 4.8]{Ste22b} as long as $S$ is trianguline (and in our case it has rank 1!). $R\Gamma_{Iw}(E[p^\infty]/E,V^1)$ is a bounded complex of modules over a regular local ring $\Lambda$ with cohomology groups that are finite $\Lambda$-modules, and therefore it is perfect over $\Lambda$ by \citep[15.74.14]{SP}; and its base change is therefore perfect over $\Lambda_\infty^{E_p}$. Then since the derived category of perfect modules is triangulated, $\widetilde{R\Gamma}(E, \mathbb{V}_\infty^{E_p};\Delta^0)$ is perfect too.
			
			\par The top two rows commute by construction. The vertical maps between the second and third row exist by the universal property of the derived tensor product. By Proposition \ref{herr_descent} we know that the bottom right vertical arrow is a quasi-isomorphism. The bottom middle vertical arrow is also a quasi-isomorphism since $\eta$ is an $L$-analytic character. Now by uniqueness of mapping cones, the map $\vartheta$ is a quasi-isomorphism too. This completes the result for the analytic setting. 
			
		\end{proof}
		
		\begin{cor} \label{delta_0_descent}
			Suppose $V$ and $\eta$ satisfy Assumption \ref{assumption_V_eta} and Assumption \ref{assumption_L-an}. Then $\widetilde{H}^1(E,\mathbb{V}_\infty^{E_p};\Delta^0)[\nabla-\eta(\nabla)]=0$, we have an isomorphism of $K$-vector spaces
			$$\widetilde{H}^2(E, \mathbb{V}_\infty^{E_p};\Delta^0) \otimes_{\Lambda_\infty^{E_p}, \eta} K \cong \widetilde{H}^2(E, V(\eta^{-1});\Delta^0_{\text{an}}), $$ and the following long exact sequence  of $K$-vector spaces
			$$ 0 \rightarrow \widetilde{H}^1(E,\mathbb{V}_\infty^{E_p};\Delta^0) \otimes_{\Lambda_\infty^{E_p}, \eta} K \rightarrow \widetilde{H}^1(E,V(\eta^{-1});\Delta^0_{\text{an}}) \rightarrow \widetilde{H}^2(E,\mathbb{V}_\infty^{E_p};\Delta^0)[\nabla-\eta(\nabla)] \rightarrow 0 $$
		\end{cor}
		\begin{proof}
			\par The cohomologies of both Selmer complex are concentrated in degree 1 and 2 so we forget all other terms. We can write out the Tor spectral sequence of the base change map $\tilde{\eta}: \Lambda_\infty \rightarrow K$ with $C=\widetilde{R\Gamma}(E, \mathbb{V}_\infty^{E_p};\Delta^0)$, taking $d=3$ in Remark \ref{tor_spectral_sequence}. Applying Proposition \ref{rank_0_an_basechange}, this looks like
			$$ \text{Tor}_{\Lambda_\infty^{E_p}}^{-j}(\widetilde{H}^{3-i}(E,\mathbb{V}_\infty^{E_p};\Delta^0),K) \Rightarrow \widetilde{H}^{3-i-j}(E,V(\eta^{-1});\Delta^0_{\text{an}}).$$ We note that $C^\bullet=\widetilde{R\Gamma}(E,\mathbb{V}^{E_p}_\infty;\Delta^0)$ is a perfect complex of $\Lambda_\infty$-modules, shown in the proof of Proposition \ref{rank_0_an_basechange}. Thus we have an exact sequence of complexes $$ 0 \rightarrow C^\bullet \xrightarrow{\nabla-\eta(\nabla)} C^\bullet \xrightarrow{\eta} C^\bullet \otimes^{\mathbb{L}}_{\Lambda_\infty^{E_p}, \eta} K \rightarrow 0. $$ Computing the long exact sequence as we do in Proposition \ref{inert_descent2}, noting that $M \otimes_{\Lambda_\infty^{E_p}, \eta} K \cong M/(\nabla-\eta(\nabla))$,  we get exact sequences
			\[ 0 \rightarrow \widetilde{H}^i(E,\mathbb{V}_\infty^{E_p};\Delta^0) \otimes_{\Lambda_\infty^{E_p}, \eta} K \rightarrow \widetilde{H}^i(E,V(\eta^{-1});\Delta^0_{\text{an}}) \rightarrow \widetilde{H}^{i+1}(E,\mathbb{V}_\infty^{E_p};\Delta^0)[\nabla-\eta(\nabla)] \rightarrow 0 ,\]
			\par We apply this for $i=0,1,2$. Recalling vanishing of the global $H^0$ and $H^3$ groups in our assumptions, we get the result.
		\end{proof}
		
		%Remark about the base change for continuous Delta^0 complex using degree shift
		
		\begin{rmk}
			\par If we define the rank 0 Selmer structures in sufficient generality, for general $L/\mathbb{Q}_p$ the above argument will work a similar way by Remark \ref{general_analytic_descent}. We can also attempt a similar base change of the $\Delta^0$ complex to the continuous version for general $d=[L:\mathbb{Q}_p]$; recall from Remark \ref{herr_descent_oc} that 
			$$  C^{\text{cts}}_{\Psi,\Gamma}(D(\eta^{-1})) \cong_{\text{q-iso}} D[\Psi-1][d] \otimes^{\mathbb{L}}_{\Lambda_\infty^L} K. $$
			
			Then we can compute the degree $d$ cohomology to see that 
			$$H^0(D[\Psi-1] \otimes^{\mathbb{L}}_{\Lambda_\infty} K) = H^d_{\Psi,\Gamma}(D) = \mathcal{H}^1_{\Psi,\Gamma}(D)  $$ after applying vanishing of analytic Herr cohmology in degree 0 and 2 to \citep[Corollary 4.4]{Ste23}. Going through the same perfectness computations and the argument of Proposition \ref{rank_0_an_basechange} and Corollary \ref{delta_0_descent} tells us that we can obtain similar exact sequences but with a degree shift of $d-1$ in continuous Herr complex cohomology. Since we do not use such a result, it is left as an exercise to the avid reader. \lozengeend
		\end{rmk}
		
		As they are defined by mapping cones involving $\text{comp}_\eta$ and $\text{comp}^\dagger_\eta$, the local cohomology of these complex in degree 1 are equipped with maps $$H^1(E,V(\eta^{-1});\Delta^0_{\text{an},p})) \rightarrow \mathfrak{K}_\eta,$$ $$H^1(E,V(\eta^{-1});\Delta^0_{\dagger,p})) \rightarrow \mathfrak{K}^\dagger_\eta.$$ By understanding these kernels and relating them to $H^1_f(E_p,S(\eta^{-1}))$, we can therefore study the relationship between the descended $\Delta^0$ Selmer complexes and the Bloch--Kato Selmer group.
		
		\par Let us write out the long exact sequences of cohomology arising from the distinguished triangles defining both of these descended `rank 0' Selmer complexes --- or more accurately the distinguished triangles coming from \citep[(11.2.1)]{KLZ17} when we take $\Delta=\Delta^0_{?}$ and $\Delta^\prime=\Delta^1$ for $? \in \{\text{an}, \dagger \}$ Selmer structures of $K$-modules. The analytic version gives the following exact sequence, assuming previous vanishing conditions of $H^0$ groups:
		
		\[ \begin{tikzcd} 
			0 & \widetilde{H}^1(E,V(\eta^{-1}); \Delta^0_{\text{an}}) & \widetilde{H}^1(E,V(\eta^{-1}); \Delta^1) & \mathcal{H}^1_{\Psi,\Gamma}(D_{\text{rig}}^\dagger(S(\tau^{-1}\eta^{-1}))) \\
			& \widetilde{H}^2(E,V(\eta^{-1}); \Delta^0_{\text{an}}) & \widetilde{H}^2(E,V(\eta^{-1}); \Delta^1) & 0
			\arrow[from=1-1, to=1-2]
			\arrow[from=1-2, to=1-3]
			\arrow[from=1-3, to=1-4]
			\arrow[from=1-4, to=2-2]
			\arrow[from=2-2, to=2-3]
			\arrow[from=2-3, to=2-4]
		\end{tikzcd} \]
		where we use Remark \ref{herr_phi_computations} to show that $\mathcal{H}^0_{\Psi,\Gamma}(D_{\text{rig}}^\dagger(S(\tau^{-1}\eta^{-1})))=\mathcal{H}^2_{\Psi,\Gamma}(D_{\text{rig}}^\dagger(S(\tau^{-1}\eta^{-1})))=0.$ This tells us in fact that $\widetilde{H}^1(E,V(\eta^{-1}); \Delta^0_{\text{an}})=0$ in degrees 0 and 3, which is what we want.
		\par We do the same for the continuous version, which is a bit longer since we have a non-vanishing $H^2$ term of the Herr complex. The exact sequence is:
		
		\[ \begin{tikzcd} 
			0 & \widetilde{H}^1(E,V(\eta^{-1}); \Delta^0_{\dagger}) & \widetilde{H}^1(E,V(\eta^{-1}); \Delta^1) & H^1_{\Psi,\Gamma}(D_{\text{rig}}^\dagger(S(\tau^{-1}\eta^{-1}))) \\
			& \widetilde{H}^2(E,V(\eta^{-1}); \Delta^0_{\dagger}) & \widetilde{H}^2(E,V(\eta^{-1}); \Delta^1) & H^2_{\Psi,\Gamma}(D_{\text{rig}}^\dagger(S(\tau^{-1}\eta^{-1}))) \\
			& \widetilde{H}^3(E,V(\eta^{-1}); \Delta^0_{\dagger}) & 0
			\arrow[from=1-1, to=1-2]
			\arrow[from=1-2, to=1-3]
			\arrow[from=1-3, to=1-4]
			\arrow[from=1-4, to=2-2]
			\arrow[from=2-2, to=2-3]
			\arrow[from=2-3, to=2-4]
			\arrow[from=2-4, to=3-2]
			\arrow[from=3-2, to=3-3]
		\end{tikzcd} \]

		\begin{rmk} \label{rank_0_euler_char}
			\par We can compute the Euler characteristics of our new descended `rank 0' Selmer complexes. By Remark \ref{herr_phi_computations} we find that for an $L$-analytic rank 1 $(\phi, \Gamma)$-module $D$, generated by a generic character in the sense of Assumption \ref{assumption_generic}, $$\chi(C^{\text{cts}}_{\Psi,\Gamma}(D)) = 0, $$ $$ \chi(C^{\text{an}}_{\Psi,\Gamma}(D)) = 1. $$
			Using the distinguished triangles emerging from Definitions \ref{rank0_descent_an} and \ref{rank0_descent_cts} we find that 
			$$\chi(\widetilde{R\Gamma}(E, V(\eta^{-1});\Delta^0_{\dagger})) = -1, $$ $$ \chi(\widetilde{R\Gamma}(E, V(\eta^{-1});\Delta^0_{\text{an}})) = 0. $$
			\par This justifies the earlier comment that the analytic version is the `right one', since Euler characteristics are preserved by base change and $$ \chi_{\Lambda_\infty^L}(\widetilde{R\Gamma}(E, \mathbb{V}_\infty^L;\Delta^0)) = 0 .$$ 
			We will see later where the non-zero Euler characteristic of the $\Delta^\dagger_0$ complex causes a problem. \lozengeend
		\end{rmk}
		
		\subsection{Recovering the Bloch--Kato Selmer group}
		
		\begin{lem} \label{analytic_crys_comp}
			For a crystalline $L$-analytic character $W$  with Hodge--Tate weight $\leq 0$ at the identity which satisfies Assumption \ref{assumption_generic}, $$H^1_f(L,W) = \mathrm{Ext}^1_{L\text{-an}}(L, W) .$$
		\end{lem}
		\begin{proof}
			\par Suppose $W$ is as above with Hodge--Tate weights $d \leq 0$ at the identity embedding and let $M$ be a crystalline extension of $L$ by $W$. Since it is de Rham, the exact sequence $$ 0 \rightarrow L \rightarrow M \rightarrow W \rightarrow 0 $$ respects filtrations, so we can compute its Hodge--Tate weights as $\{d , 0\}$ at the identity embedding and $\{0 , 0\}$ at the conjugate embedding, so $M$ is $L$-analytic.
			\par For the reverse inclusion we can compare dimensions as $L$-vector spaces. Due to our $H^0$ vanishing conditions the dimension formula for local $H^1_f$ tells us that $$ \text{dim} H^1_f(E_p, W) = \#\{ \text{Hodge--Tate \enspace weights \enspace of \enspace} W \leq 0 \} = 1. $$
			\par Likewise, \citep[Theorem 0.3]{FX12} shows that as long as $W$ is not generic in the sense of Assumption \ref{assumption_generic}, $\dim_L \mathrm{Ext}^1_{E_p\text{-an}}(E_p, W) =1$. By the inclusion above, the spaces must be equal. 
		\end{proof}
		
		\begin{lem}
			Assume $S$ and $\eta$ satisfy Assumption \ref{assumption_L-an}, then the kernel $\mathfrak{K}_\eta^{\dagger}$ of the comparison map $\text{comp}_{\eta}^{\dagger}$ satisfies $$\mathfrak{K}_\eta^{\dagger} = H^1_{\perp}(L,S(\eta^{-1})).$$
		\end{lem}
		\begin{proof}
			\par Suppose $c \in \mathfrak{K}_\eta^\dagger$, then by commutativity of the diagram of Definition \ref{comp_descent_oc} we have for all $d \in H^1_{\dagger}(E_p,S(\eta^{-1})^*(1))$, $$(c,i(d))_{\text{Tate}} = (\text{comp}_\eta^\dagger(c),d)_{\dagger}=0.$$ Hence $c \in H^1_{\dagger}(E_p,S(\eta^{-1})^*(1))^{\perp} = H^1_{\perp}(E_p,S(\eta^{-1}))$.
			\par Conversely suppose $c \in H^1_{\perp}(E_p,S(\eta^{-1}))$, then again for all such $d$ $$(\text{comp}_\eta^\dagger(c),d)_{\dagger}= (c,i(d))_{\text{Tate}} = 0. $$ In the proof of \citep[Proposition 5.2.19]{SV23} it is stated that the pairing of Koszul complexes is nondegenerate, thus $\text{comp}_\eta^\dagger(c)=0$.
		\end{proof}	
		
		\par It would be preferable to have an analytic version of this result, saying perhaps that $$\mathfrak{K}_\eta = Ext^1_{L-\text{an}}(L,S^*(1)))^\perp.$$ However, we can't use the Lubin--Tate character as a dualising character for the analytic Herr complex due to Remark \ref{herr_phi_computations}, so we do not get a result directly. However, due to some of our earlier assumptions, $L$-analytic extension classes, overconvergent classes and crystalline classes all coincide. So this analytic version is true incidentally. The next Proposition will show this. 
		
		\begin{prop} \label{L-an_crys}
			Assuming the vanishing conditions of Assumption \ref{assumption_generic} as well as Assumption \ref{assumption_L-an}, $$\mathfrak{K}^\dagger_\eta = H^1_f(E_p,S(\eta^{-1})).$$
		\end{prop}
		\begin{proof}
			\par Under the hypotheses, by \citep[Corollary 4.4]{FX12} we have an equality of the analytic and non-analytic $\phi$-complexes, i.e. $$\text{Ext}^1_{E_p-\text{an}}(E_p,S(\eta^{-1})^*(1))=\text{Ext}^1_{\dagger}(E_p,S(\eta^{-1})^*(1)). $$
			Now applying Lemma \ref{analytic_crys_comp} and by local Tate duality, $$ \mathfrak{K}_\eta^{\dagger} = \dfrac{H^1(E_p,S(\eta^{-1})^*(1))}{H^1_{\perp}(E_p,S(\eta^{-1})^*(1))} = \dfrac{H^1(E_p,S(\eta^{-1})^*(1))}{H^1_f(E_p,S(\eta^{-1})^*(1))} = H^1_f(E_p,S(\eta^{-1})). $$ 
		\end{proof}
		
		\par In the split case we used the Panchiskin subrepresentation $V^0$ and the work of \citep[4.1.7]{FK06}, stated in Lemma \ref{fukaya_kato}, to compare $\widetilde{H}^1(E,T(\eta^{-1});\Delta^0)$ to $H^1_f(T(\eta^{-1}))$ locally at $p$. We no longer have such $V^0$; the closest subrepresentation is $V^1$ which has the wrong Panchishkin rank. Nevertheless we can still use a weaker form of Fukaya--Kato's arguments.
		
		\begin{lem} \label{FK_inert}
			When $\eta$ lies in the region $\Sigma^{(0)}$ of Figure \ref{fig3}, the following diagram is Cartesian
			\[\begin{tikzcd}
				H^1_f(E_p, V(\eta^{-1})) & H^1(E_p,V^1(\eta^{-1})) \\
				H^1_f(E_p, S(\eta^{-1})) & H^1(E_p,S(\eta^{-1})) 
				\arrow[hook, from=1-1, to=1-2]
				\arrow[hook, from=2-1, to=2-2]
				\arrow[from=1-1, to=2-1]
				\arrow[from=1-2, to=2-2]
			\end{tikzcd} .\]
		\end{lem}
		\begin{proof}
			Since $\eta$ lies in the given region, there is no Panchishkin subrepresentation. But by \citep{FK06} $H^1_f(E_p, V^1)=H^1_f(E_p, V)$ and therefore the top embedding is well defined. Moreover, there is no map $H^1(E_p,V(\eta^{-1}))$ to $H^1(E_p,S(\eta^{-1}))$ but $S(\eta^{-1})$ is a quotient of $V^1(\eta^{-1})$ so the right vertical map is also well defined. Now we can compare dimensions as $L$-vector spaces; by vanishing of local $H^0$ and $H^2$ groups in Assumption \ref{assumption_V_eta} and the local Euler characteristic formula, we have that $$ \dim_{\mathbb{Q}_p} H^1(E_p,V^1(\eta^{-1})) = 4$$ $$ \dim_{\mathbb{Q}_p} H^1(E_p,V^1(\eta^{-1})) = 2 .$$ By virtue of $\eta$ lying in the region $\Sigma^{(0)}$, $V(\eta^{-1})$ has 3 positive Hodge--Tate weights and 3 non-negative ones (looking at the 3 weights from each embedding). The representation $S(\eta^{-1})$ is a subquotient picking up exactly one negative weight at the identity embedding, and with weight 1 at the non-identity embedding. By the Bloch--Kato Selmer group dimension formula of \citep[Proposition 2.8]{Bel09}, $$ \dim_{\mathbb{Q}_p} 	H^1_f(E_p, V(\eta^{-1})) = 3 $$ $$ \dim_{\mathbb{Q}_p} H^1_f(E_p, S(\eta^{-1})) = 1. $$
			\par Therefore the cokernels of both embeddings are isomorphic. We obtain the universal property of a pullback from here, since all terms are $\mathbb{Q}_p$-vector spaces.
		\end{proof}

		\begin{prop} \label{delta_0_cts_crys}
			Suppose $V$, $\eta$ satisfies the Assumption \ref{assumption_V_eta} and Assumption \ref{assumption_L-an} such that $\eta$ lies in region $\Sigma^{(0)}$ of Figure \ref{fig3}. Then
			$$ \widetilde{H}^1(E,V(\eta^{-1});\Delta^0_\dagger) = H^1_f(E,V(\eta^{-1})) .$$
		\end{prop}
		\begin{proof}
			Since all the Selmer complexes involved use the unramified condition at all primes $l \neq p$, we only need to show the local conditions at $p$ match. From Lemma \ref{FK_inert} it follows that $$ H^1(E_p,V^1(\eta^{-1})/H^1_f(E_p,V(\eta^{-1})) \cong H^1(E_p,S(\eta^{-1})/H^1_f(E_p,S(\eta^{-1})).$$ The latter is isomorphic to $H^1_{\phi,\Gamma}(D_{\text{rig}}^\dagger(E_p,S(\tau^{-1}\eta^{-1})))$ by Proposition \ref{L-an_crys}, which is also the cokernel of the map $\widetilde{H}^1(E,V(\eta^{-1});\Delta^0_\dagger) \rightarrow \widetilde{H}^1(E,V(\eta^{-1});\Delta^1)$ by construction. We consider this in the commutative diagram:
			\[ \begin{tikzcd}[scale cd=0.9]
				0 & H^1_f(E_p,V(\eta^{-1})) & H^1(E_p, V^1(\eta^{-1})) & \dfrac{H^1(E_p,V^1(\eta^{-1}))}{H^1_f(E_p,V(\eta^{-1}))} & 0  \\
				0 & H^1(E_p,V(\eta^{-1});\Delta^0_{\dagger,p}) & H^1(E_p,V(\eta^{-1});\Delta^1_p) & H^1_{\phi,\Gamma}(D_{\text{rig}}^{\dagger}(S(\tau^{-1}\eta^{-1}))) & 0
				\arrow[from=1-1, to=1-2]
				\arrow[from=1-2, to=1-3]
				\arrow[from=1-3, to=1-4]
				\arrow[from=1-4, to=1-5]
				\arrow[from=2-1, to=2-2]
				\arrow[from=2-2, to=2-3]
				\arrow[from=2-3, to=2-4]
				\arrow[from=2-4, to=2-5]
				\arrow[from=1-2, to=2-2]
				\arrow[from=1-3, to=2-3, "\sim"]
				\arrow[from=1-4, to=2-4, "\sim"]
			\end{tikzcd} \]
			By the five lemma we have an isomorphism in the right vertical arrow.
		\end{proof}
		
		\par We ideally want to work with simple Selmer structures, where we can use Proposition \ref{selmer_pt_dual}, that $\widetilde{H}^2(V;\Delta) \cong \widetilde{H}^1(V^*(1);\Delta^\vee)^\vee$ where $\Delta^\vee$ are the dual local conditions with respect to local Tate duality. We have two issues if we try to do this with the $\Delta_\dagger^0$ complex; firstly it doesn't have Euler characteristic 0, and therefore we can't have both $\widetilde{H}^1(E,V(\eta^{-1});\Delta^0_\dagger)$ and $\widetilde{H}^2(E,V(\eta^{-1});\Delta^0_\dagger)$ both equal to 0 even assuming non-vanishing of $\eta(L_{p,E_p-\text{an}}^*)$. The Bloch--Kato conjecture of course stipulates that if $\eta(L_{p,E_p-\text{an}}^*)$ is non-vanishing, combining with vanishing result of Lemma \ref{analytic_rank0_vanishing}, both must be 0. Thus $\widetilde{H}^2(E,V(\eta^{-1});\Delta^0_\dagger)$ does not see the Bloch--Kato Selmer group of $V^*(1)$, or to put it another way $\Delta_\dagger^0$ is not a simple Selmer complex - as we will prove in the immediate lemma. This is where we need to jump back to the $\Delta^0_{\text{an}}$ Selmer complex, which we show is simple and therefore does satisfy the nice Poitou--Tate compatibility of Proposition \ref{selmer_pt_dual}. This non-simplicity is a new phenomenon when $L \neq \mathbb{Q}_p$.
		
		\begin{lem} \label{cts_nonsimple}
			Suppose $V$, $\eta$ satisfy Assumption \ref{assumption_V_eta} and Assumption \ref{assumption_L-an}. Then the following statements of Selmer structures of $K$-modules hold;
			\begin{itemize}
				\item $\Delta^0_{\dagger}$ is not a simple Selmer structure.
				\item $\Delta^0_{\text{an}}$ is a simple Selmer structure.
			\end{itemize}
		\end{lem}
		\begin{proof}
			We show that the local complex $U_p^+$ attached to $\Delta^0_{\dagger}$ is non-zero in degree 2, which breaks one of the three conditions for simplicity. Moreover it will be clear that this happens as soon as $L \neq \mathbb{Q}_p$. Using the distinguished triangle of Definition \ref{rank0_descent_cts}, $$ U_p^+ = R\Gamma(E_p,V(\eta^{-1});\Delta^0_{\dagger,p}) \rightarrow R\Gamma(E_p, V^1(\eta^{-1})) \rightarrow C^{\text{cts}}_{\Psi,\Gamma}(D^\dagger_{\text{rig}}(S(\tau^{-1}\eta^{-1}))), $$ we can get a long exact sequence of cohomology; looking at $H^2$ terms tells us: $$ H^2(U_p^+) \xrightarrow{\sim} H^2_{\Psi,\Gamma}(D^\dagger_{\text{rig}}(S(\tau^{-1}\eta^{-1}))). $$ By Corollary \ref{herr_psi_computations} we see that the latter is a 1-dimensional vector space, so the former is also non-zero. Note that the continous Herr complex has cohomological degree $[L:\mathbb{Q}_p]+1$ and generic vanishing of top degree cohomology. 
			\par If we repeat the above argument for the $\Delta^0_{\text{an}}$ complex, we find that the local complex at $p$ is isomorphic to $\mathcal{H}^2_{\Psi,\Gamma}(D^\dagger_{\text{rig}}(S(\tau^{-1}\eta^{-1})))$ which is 0, whilst the other terms are the same. Thus it is a simple Selmer structure.
		\end{proof}
		
		\begin{rmk}
			Since we are working with a non-simple Selmer complex, we no longer have the nice characterisation of its global $H^2$ term using \citep[11.2.9]{KLZ17}. We have to use the more general Poitou--Tate duality statement of 11.2.8 of op. cit. We unravel these definitions; their $U_p^+$ is the same as our $U_p^+$ in the proof of Lemma \ref{cts_nonsimple}, hence $H^2(U_p^+) = H^2_{\Psi,\Gamma}(D^\dagger_{\text{rig}}(S(\tau^{-1}\eta^{-1}))) $. Now we interpret the kernel of $H^1$ groups using the distinguished triangle $$ (U_p^\vee)^+ \rightarrow R\Gamma(E_p, V(\eta^{-1})) \rightarrow (U_p^\vee)^-  .$$ From the arising exact sequence we see that $H^1((U_p^\vee)^-) = \dfrac{H^1(E_p,V(\eta^{-1})^*(1))}{H^1_f(E_p,V(\eta^{-1})^*(1))} $ by applying Proposition \ref{delta_0_cts_crys}. Thus the Poitou--Tate statement reduces to the exact sequence $$0 \rightarrow H^1_f(E,V^*(1)) \rightarrow \widetilde{H}^2(E,V(\eta^{-1});\Delta^0_{\dagger}) \xrightarrow{\theta} H^2_{\Psi,\Gamma}(D^\dagger_{\text{rig}}(S(\tau^{-1}\eta^{-1}))). $$
			The left most term has dimension 1 and by the Bloch--Kato conjecture we expect the left most term to be zero when the twist of the `$p$-adic $L$-function' $\eta(L_{p,E_p-\text{an}}^*)$ vanishes. But showing the map $\theta$ is injective under this hypothesis is not straightforward. It will be a consequence of the following work. \lozengeend
		\end{rmk}
		
		\begin{lem} \label{delta_0_an_crys}
			Suppose $V$, $\eta$ satisfy Assumption \ref{assumption_V_eta} and Assumption \ref{assumption_L-an} such that $\eta$ lies in region $\Sigma^{(0)}$ of Figure \ref{fig3}. Then
			$$ \widetilde{H}^1(E,V(\eta^{-1});\Delta^0_{\text{an}}) = H^1_f(E,V(\eta^{-1})) .$$
		\end{lem}
		\begin{proof}
			By the previous lemma we only need to show the $\Delta^0_{\text{an}}$ and $\Delta^0_{\dagger}$ complexes compute the same cohomology in degree 1. We take a look at the exact sequences of cohomology following their definitions. From the long exact sequences above Remark \ref{rank_0_euler_char}, $$ \widetilde{H}^1(E,V(\eta^{-1});\Delta^0_{\text{an}}) \cong \ker \left( H^1(E,V^1(\eta^{-1})) \rightarrow \mathcal{H}^1_{\Psi,\Gamma}(D^\dagger_{\text{rig}}(S(\tau^{-1}\eta^{-1}))) \right), $$
			$$ \widetilde{H}^1(E,V(\eta^{-1});\Delta^0_{\dagger}) \cong \ker \left( H^1(E,V^1(\eta^{-1})) \rightarrow H^1_{\Psi,\Gamma}(D^\dagger_{\text{rig}}(S(\tau^{-1}\eta^{-1}))) \right). $$ By Corollary \ref{herr_psi_computations}, the images of these maps are equal, as $$H^1_{\Psi,\Gamma}(D^\dagger_{\text{rig}}(S(\tau^{-1}\eta^{-1}))) = \mathcal{H}^1_{\Psi,\Gamma}(D^\dagger_{\text{rig}}(S(\tau^{-1}\eta^{-1})))$$ and all degree 2 terms match. The required equality follows.
		\end{proof}
		
		\begin{prop} \label{delta_0_an_crys_dual}
			Suppose $\eta$ is an $L$-analytic character such that $V$ and $\eta$ satisfy both Assumption \ref{assumption_V_eta} and Assumption \ref{assumption_L-an}. Then
			$$ \widetilde{H}^2(E,V(\eta^{-1});\Delta^0_{\text{an}}) = H^1_f(E,V(\eta^{-1})^*(1)) .$$
		\end{prop}
		\begin{proof}
			This follows from Poitou--Tate duality if $\Delta^0_{\text{an}}$ is a simple Selmer structure under the running hypotheses, which is true by Lemma \ref{selmer_simple}.
		\end{proof}

		\begin{thm} \label{analytic_BK_bound}
			Suppose $V$, $\eta$ satisfy Assumptions \ref{assumption_V_eta} and Assumption \ref{assumption_L-an}, the infinity type of $\eta$ lies in the $\Sigma^{(0)}$ region of Figure \ref{fig3} and $\eta(L^*_{p,E_p\text{-an}}) \neq 0 $. Then $$ H^1_f(E,V(\eta^{-1})^*(1)) = 0 .$$  
		\end{thm}
		\begin{proof}
			\par By Corollary \ref{delta_0_descent}, we have the isomorphism of $K$-vector spaces
			$$\widetilde{H}^2(E, \mathbb{V}_\infty^{E_p};\Delta^0) \otimes_{\Lambda_\infty^{E_p}, \eta} K \cong \widetilde{H}^2(E, V(\eta^{-1});\Delta^0_{\text{an}}). $$ We now apply \citep[Lemma 14.15]{Kat04}, or rather the Frechet--Stein version developed in Proposition \ref{inert_descent}; taking $A=\Lambda_\infty^{E_p}$, $M$ a coadmissible torsion $A$-module with $\{ \mathfrak{q} \mid \ker(\eta) \} \nsubseteq \text{Supp}(M) $, $a=\nabla-\eta(\nabla)$ we get the following equality in the Grothendieck group of the category of coadmissible $\Lambda_\infty^{E_p}$-modules
			\begin{eqnarray*} 
				[M \otimes_{\Lambda_\infty^{E_p}, \eta} K] & = & \sum_{\mathfrak{q}} \text{length}_{\Lambda_{\infty, \mathfrak{q}}^{E_p}}(M_\mathfrak{q}) \cdot \left[\dfrac{\Lambda_\infty^{E_p}}{\mathfrak{q} + \ker(\eta)} \right]  \\
				& = & \sum_{\mathfrak{q} \mid \ker(\eta)} \text{length}_{\Lambda_{\infty, \mathfrak{q}}^{E_p}}(M_\mathfrak{q}) \cdot [K] .
			\end{eqnarray*}
			
			\par We apply this when $M= \widetilde{H}^2(E,\mathbb{V}_\infty^{E_p};\Delta^0).$ Recall by Theorem \ref{analytic_imc}, taking $\delta=\overline{\eta}$, we have
			$$\mathrm{char}_{e_{\overline{\eta}}\Lambda_{\infty,K}^{E_p}}e_{\overline{\eta}} \widetilde{H}^2(E,\mathbb{V}_\infty^{E_p};\Delta^0) \bigm \vert (e_{\overline{\eta}} L^*_{p,E_p\text{-an}}(\Pi)). $$ We will have to change the base field $K$ to account for the main conjecture over the $\mathbb{C}_p$-valued distribution algebra.	Since $\eta(L^*_{p,E_p\text{-an}}) \neq 0 $
			the divisibility tells us that $\{ \mathfrak{q} \mid \ker(\eta)\} \cup \text{Supp}(\widetilde{H}^2(E,\mathbb{V}_\infty^{E_p};\Delta^0)) = \emptyset.$
			Now the above equality in the Grothendieck group tells us that 
			$$ \dim_K \widetilde{H}^2(E,\mathbb{V}_\infty^{E_p};\Delta^0) \otimes_{\Lambda_\infty^{E_p}, \eta} K = 0. $$	
			Finally applying Proposition \ref{delta_0_an_crys_dual} gives the result.
		\end{proof}

		\begin{rmk}
			This tells us that $$ \text{rank}_{\mathcal{O}_K} H^1_f(E,T(\eta^{-1})^*(1)) = 0$$ just like the split prime case in \citep[Corollary 4.4]{Man22}, but we do not get a bound on the number of elements of the finite Bloch--Kato Selmer group like in op. cit. This is due to the lack of integrality of our results henceforth, discussed in Remark \ref{ardakov_berger}. \lozengeend
		\end{rmk}

	\bibliographystyle{amsalpha}
	\bibliography{bibliography}

\providecommand{\bysame}{\leavevmode\hbox to3em{\hrulefill}\thinspace}
\providecommand{\MR}{\relax\ifhmode\unskip\space\fi MR }
\providecommand{\MRhref}[2]{%
  \href{http://www.ams.org/mathscinet-getitem?mr=#1}{#2}
}
\providecommand{\href}[2]{#2}
\begin{thebibliography}{BLGHT11}

\bibitem[AB23]{AB23}
Konstantin Ardakov and Laurent Berger, \emph{Bounded functions on the character variety}, arXiv preprint arXiv:2301.13650 (2023).

\bibitem[Bal10]{Bal10}
Paul Balmer, \emph{Tensor triangular geometry}, Proceedings of the International Congress of Mathematicians, vol.~2, 2010, pp.~85--112.

\bibitem[Bel09]{Bel09}
Jo\"el Bellaiche, \emph{An introduction to the conjecture of {Block-Kato}}, https://virtualmath1.stanford.edu/~conrad/BSDseminar/refs/BKintro.pdf (2009).

\bibitem[Ber16]{Ber16}
Laurent Berger, \emph{Multivariable ($\phi, {\Gamma}$)-modules and locally analytic vectors}.

\bibitem[BLGHT11]{BLGHT11}
Tom Barnet-Lamb, David Geraghty, Michael Harris, and Richard Taylor, \emph{A family of {Calabi--Yau varieties} and potential automorphy {II}}, Publications of the Research Institute for Mathematical Sciences \textbf{47} (2011), no.~1, 29--98.

\bibitem[BO17]{BO17}
Kazim Buyukboduk and Tadashi Ochiai, \emph{Main conjectures for higher rank nearly ordinary families -- {I}}, Journal of Number Theory \textbf{252} (2017).

\bibitem[Bou98]{Bou98}
Nicolas Bourbaki, \emph{Commutative algebra: chapters 1-7}, vol.~1, Springer Science \& Business Media, 1998.

\bibitem[BSX20]{BSX20}
Laurent Berger, Peter Schneider, and Bingyong Xie, \emph{Rigid character groups, {Lubin--Tate} theory, and ($\varphi_{L}$, ${\Gamma}_{L}$)-modules}, vol. 263, American mathematical society, 2020.

\bibitem[Col10]{Col10}
Pierre Colmez, \emph{Repr{\'e}sentations de ${GL}_2(\mathbb{Q}_p)$ et ($\varphi$, ${\Gamma}$)-modules}, Ast{\'e}risque \textbf{330} (2010), no.~281, 509.

\bibitem[Con11]{Con11}
Brian Conrad, \emph{Lifting global representations with local properties}, preprint (2011).

\bibitem[DJ95]{dJ95}
Arthur~J De~Jong, \emph{Crystalline dieudonn{\'e} module theory via formal and rigid geometry}, Publications Math{\'e}matiques de l'IH{\'E}S \textbf{82} (1995), 5--96.

\bibitem[FK06]{FK06}
Takako Fukaya and Kazuya Kato, \emph{A formulation of conjectures on $p$-adic zeta functions in noncommutative {Iwasawa} theory}, Proceedings of the St. Petersburg Mathematical Society, 2006, pp.~1--85.

\bibitem[FX12]{FX12}
Lionel Fourquaux and Bingyong Xie, \emph{Triangulable $\mathcal{O}_{F}$-analytic $(\varphi_q,{\Gamma})$-modules of rank 2}, arXiv preprint arXiv:1206.2102 (2012).

\bibitem[Her98]{Her98}
Laurent Herr, \emph{On the galois cohomology of $ p $-adic fields}, Bulletin of the Mathematical Society of France \textbf{126} (1998), no.~4, 563--600.

\bibitem[Kat04]{Kat04}
Kazuya Kato, \emph{$p$-adic {Hodge} theory and values of zeta functions of modular forms}, Ast{\'e}risque \textbf{295} (2004), 117--290.

\bibitem[KLZ17]{KLZ17}
Guido Kings, David Loeffler, and Sarah~Livia Zerbes, \emph{{Rankin--Eisenstein} classes and explicit reciprocity laws}, Cambridge Journal of Mathematics \textbf{5} (2017), no.~1, 1--122.

\bibitem[KR09]{KR09}
Mark Kisin and Wei Ren, \emph{Galois representations and {Lubin-Tate} groups}, Documenta Mathematica \textbf{14} (2009), 441--461.

\bibitem[KV22]{KV22}
Benjamin Kupferer and Otmar Venjakob, \emph{{Herr}-complexes in the {Lubin}--{Tate} setting}, Mathematika \textbf{68} (2022), no.~1, 74--147.

\bibitem[Laz67]{Laz67}
Daniel Lazard, \emph{Disconnexit{\'e}s des spectres d'anneaux et des pr{\'e}sch{\'e}mas}, Bulletin de la Soci{\'e}t{\'e} Math{\'e}matique de France \textbf{95} (1967), 95--108.

\bibitem[LSZ21]{LSZ21}
David Loeffler, Christopher Skinner, and Sarah~Livia Zerbes, \emph{An {Euler} system for {GU}(2,1)}, Mathematische Annalen (2021), 1--51.

\bibitem[LZ20]{LZ20}
David Loeffler and Sarah~Livia Zerbes, \emph{{Euler} systems with local conditions}, Development of Iwasawa Theory—the Centennial of K. Iwasawa's Birth, Mathematical Society of Japan, 2020, pp.~1--26.

\bibitem[Man22]{Man22}
Muhammad Manji, \emph{{Euler} systems and {Selmer} bounds for {GU(2,1)}}, arXiv preprint arXiv:2208.01102 (2022).

\bibitem[MR04]{MR04}
Barry Mazur and Karl Rubin, \emph{{Kolyvagin} systems}, American Mathematical Society, 2004.

\bibitem[MS25]{MR25}
Muhammad Manji and Rustam Steingart, \emph{{Lubin--Tate} {Iwasawa} cohomology and $(\phi,{\Gamma})$-modules}, Upcoming work (2025).

\bibitem[MSVW24]{MSVW24}
Milan Malcic, Rustam Steingart, Otmar Venjakob, and Max Witzelsperger, \emph{$\epsilon$-isomorphisms for rank one $(\varphi,{\Gamma}) $-modules over {Lubin-Tate} {Robba} rings}, arXiv preprint arXiv:2404.09974 (2024).

\bibitem[Nek06]{Nek06}
Jan Nekov{\'a}{\v{r}}, \emph{{Selmer} complexes}, vol. 310, Soci{\'e}t{\'e} math{\'e}matique de France, 2006.

\bibitem[Pot12]{Pot12}
Jonathan Pottharst, \emph{Cyclotomic {Iwasawa} theory of motives}, preprint (2012), 115.

\bibitem[Pot13]{Pot13}
\bysame, \emph{Analytic families of finite-slope {Selmer} groups}, Algebra \& Number Theory \textbf{7} (2013), no.~7, 1571--1612.

\bibitem[ST01]{ST01}
Peter Schneider and Jeremy Teitelbaum, \emph{$p$-adic {Fourier} theory}, arXiv preprint math/0102012 (2001).

\bibitem[ST02]{ST02}
\bysame, \emph{Algebras of $p$-adic distributions and admissible representations}, arXiv preprint math/0206056 (2002).

\bibitem[{Sta}18]{SP}
The {Stacks Project Authors}, \emph{\textit{Stacks Project}}, \url{https://stacks.math.columbia.edu}, 2018.

\bibitem[Ste22a]{Ste22a}
Rustam Steingart, \emph{Analytic cohomology of families of {L}-analytic {Lubin--Tate} ($\varphi_{L}, {\Gamma}_{L})$-modules}, Ph.D. thesis, 2022.

\bibitem[Ste22b]{Ste22b}
\bysame, \emph{{Iwasawa} cohomology of analytic $(\varphi_{L}, {\Gamma}_{L})$-modules}, arXiv preprint arXiv:2212.02275 (2022).

\bibitem[Ste23]{Ste23}
\bysame, \emph{Comparisons of {Lie} algebra cohomologies of $(\varphi,{\Gamma})$-modules}, arXiv preprint arXiv:2311.07799 (2023).

\bibitem[SV15]{SV15}
Peter Schneider and Otmar Venjakob, \emph{{Coates--Wiles} homomorphisms and {Iwasawa} cohomology for {Lubin--Tate} extensions}, Elliptic Curves, Modular Forms and Iwasawa Theory-Conference in honour of the 70th birthday of John Coates, Springer, 2015, pp.~401--468.

\bibitem[SV23a]{SV23b}
\bysame, \emph{Compairing categories of {Lubin--Tate} $(\varphi_{L},{\Gamma}_{L})$-modules}, arXiv preprint arXiv:2301.11617 (2023).

\bibitem[SV23b]{SV23}
\bysame, \emph{Reciprocity laws for $(\varphi_{L},{\Gamma}_{L})$-modules over {Lubin--Tate} extensions}, arXiv preprint arXiv:2301.11606 (2023).

\bibitem[SV24a]{SV24}
Takamichi Sano and Otmar Venjakob, \emph{On {L}ubin--{T}ate regulator maps and {K}ato's explicit reciprocity law}, arXiv preprint arXiv:2412.12429 (2024).

\bibitem[SV24b]{SaV24}
\bysame, \emph{On lubin-tate regulator maps and kato's explicit reciprocity law}, arXiv preprint arXiv:2412.12429 (2024).

\bibitem[Ven23]{Ven23}
Otmar Venjakob, \emph{Explicit reciprocity laws in {Iwasawa} theory-- {A} survey with some focus on the {Lubin--Tate} setting}, arXiv preprint arXiv:2311.08237 (2023).

\end{thebibliography}
	
\end{document}